%% file: main.tex
\documentclass[11pt, letterpaper]{amsart}

\usepackage{tikz}
\usetikzlibrary{calc,patterns,shapes,fit,arrows,positioning,decorations.pathreplacing,intersections,arrows.meta}
\usepackage[left=1in,right=1in,bottom=1in,top=1in]{geometry}

\usepackage{pgfplots}
\usepackage{xcolor}
\pgfplotsset{compat=1.17}
\usepgfplotslibrary{fillbetween}

\usepackage{caption}
\usepackage{subcaption}

\newif\ifshowextra
\showextrafalse 

\newcommand{\maybeextra}[1]{\ifshowextra \textcolor{red}{#1}\fi}

\usepackage{pgfplots}
\usepackage{xcolor}
\pgfplotsset{compat=1.17}
\usepgfplotslibrary{fillbetween}

\usepackage[left=1in,right=1in,bottom=1in,top=1in]{geometry}
\usepackage{amsfonts}
\usepackage{amsmath, amssymb}
\usepackage{graphicx}
\usepackage[font=small,labelfont=bf]{caption}
\usepackage[pdfpagelabels,hyperindex]{hyperref}
\usepackage{xcolor}
\usepackage{amsthm}
\usepackage{setspace}
\usepackage{float}
\usepackage{pgfplots}
\usepackage{listings}
\usepackage{longtable}

\usepackage{mathrsfs}

\usepackage{cleveref}
\Crefname{thm}{Theorem}{Theorems}

\usepackage{booktabs}
\usepackage{makecell}
\usepackage{multirow}
\usepackage{tabularray}
\usepackage{xspace}
\usepackage{enumerate}
\usepackage{lineno}

\usepackage{thmtools}
\usepackage{thm-restate}

\usepackage[sort]{cite}

\usepackage{mathtools}
\usepackage[english]{babel}

\usepackage{mleftright}

\usepackage{todonotes}

\input{newcommands}

\hypersetup{
pdftitle={Rumour Spreading in Social Network Models},
pdfsubject={Geometric Inhomogeneous Random Graphs, Rumour Spreading, Real-World Network Models, MCD-GIRGs}
pdfauthor={},
pdfkeywords={}}

\newtheorem{thm}{Theorem}[section]
\newtheorem*{thm*}{Theorem}

\newtheorem{lemma}[thm]{Lemma}

\newtheorem{claim}[thm]{Claim}

\theoremstyle{definition}
\newtheorem{defn}[thm]{Definition}

\theoremstyle{remark}
\newtheorem{rem}[thm]{Remark}

\definecolor{energy}{RGB}{114,0,172}
\definecolor{freq}{RGB}{45,177,93}
\definecolor{spin}{RGB}{251,0,29}
\definecolor{signal}{RGB}{203,23,206}
\definecolor{circle}{RGB}{217,86,16}
\definecolor{average}{RGB}{203,23,206}
\colorlet{shadecolor}{gray!20}
\pgfplotsset{compat=1.9}
\usepgflibrary{fpu}

\numberwithin{equation}{section}

\newif\ifanon

\ifanon
  \author{Author(s)} 
\else
\author[M. Kaufmann]{Marc Kaufmann}
\author[K. Lakis]{Kostas Lakis}
\author[J. Lengler]{Johannes Lengler}
\author[R. R. Ravi]{Raghu Raman Ravi}
\author[U. Schaller]{Ulysse Schaller}
\author[K. Sturm]{Konstantin Sturm}
\fi

\ifanon
\else
\thanks{Department of Computer Science, Institute of Theoretical Computer Science, ETH Z\"{u}rich, Switzerland.  \texttt{\{kamarc,klakis,lenglerj,raravi,ulysses,kosturm\}@ethz.ch}. K.L. gratefully acknowledges support from the John S. Latsis Public Benefit Foundation. This scientific paper was supported by the Onassis Foundation - Scholarship ID: F ZT 044-1/2023-2024. M.K. and U.S. gratefully acknowledge support by the Swiss National Science Foundation [grant number 200021\_192079].}
\fi


\title[Rumour Spreading in Social Network Models]{Rumour Spreading Depends on the Latent Geometry and Degree Distribution in Social Network Models}
\begin{document}

\ifanon
\linenumbers
\fi

\begin{abstract}
\input{abstract}
\end{abstract}
\maketitle
\thispagestyle{empty}
\clearpage
\setcounter{page}{1}

\newpage

\section{Introduction}\label{sec:intro}
\input{introduction}


\section{Results}\label{sec:results}
\input{results}

\section{Proof Outlines}\label{sec:tools}
\input{tools}

\section{Ultrafast regime in Euclidean GIRGs}\label{sec:ultrafast_euclidean}
\input{ultrafast_euclidean}

\section{Fast regime in Euclidean GIRGs}\label{sec:polylogupper}

\input{polylogarithmic_upper_bound}

\section{Logarithmic lower bound in Euclidean GIRGs}\label{sec:loglower}
\input{log_lower_bound}

\section{Slow regime in Euclidean GIRGs}\label{sec:polynomiallower}

\input{polynomial_lower_bound}

\section{Rumour spreading in MCD-GIRGs}\label{sec:MCD}
\input{mcd_girgs}

\ifanon
\else
\section*{Acknowledgments}
We thank Familie Türtscher for their hospitality during the research retreat in Buchboden in June 2024.
\fi

\bibliographystyle{abbrv}	
\bibliography{bibliography}

\end{document}

%% file: newcommands.tex
\newcommand{\girgs}{\textsc{GIRGs}}
\newcommand{\girg}{\textsc{GIRG}}
\newcommand{\whp}{whp}

\newcommand{\myO}[1]{\ensuremath{O\mleft(#1\mright)}}
\newcommand{\myo}[1]{\ensuremath{o\mleft(#1\mright)}}
\newcommand{\myOmega}[1]{\ensuremath{\Omega\mleft(#1\mright)}}
\newcommand{\myTheta}[1]{\ensuremath{\Theta\mleft(#1\mright)}}
\newcommand{\myomega}[1]{\ensuremath{\omega\mleft(#1\mright)}}
\newcommand{\longEdgeEpsilon}{\ensuremath{\delta}}
\newcommand{\edge}[2]{\ensuremath{(#1, #2)}}
\newcommand{\degreeof}[1]{\deg(#1)}
\newcommand{\weightof}[1]{W_{#1}}
\newcommand{\slowdown}{slowdown}
\newcommand{\geomdist}[2]{\ensuremath{|#1 - #2|}}
\newcommand{\rateof}[1]{\ensuremath{s_{#1}}}
\newcommand{\expectationof}[1]{\ensuremath{\mathbb{E}\mleft[#1\mright]}}

\newcommand{\maxWeightEpsilon}{\ensuremath{\epsilon_{w}}}
\newcommand{\wmax}{\ensuremath{w_{max}}}
\newcommand{\prob}[1]{\ensuremath{\mathbb{P}\mleft[#1\mright]}}
\newcommand{\breakvalue}[1]{\overline{#1}}
\newcommand{\markovEdgesEpsilon}{\epsilon_{M}}
\newcommand{\logConstant}{\ensuremath{C}}
\newcommand{\roundsEpsilon}{\epsilon_{T}}
\newcommand{\slowTheoremEpsilon}{\rho}
\newcommand{\hierarchyN}{\ensuremath{L}}
\newcommand{\hierarchyGamma}{\ensuremath{\gamma}}
\newcommand{\LRP}{LRP}

\newcommand{\hierarchySingleStepEventConst}{\mathcal{T}}
\newcommand{\isedge}[3][]{\ensuremath{#2 \sim_{#1} #3}}

\newcommand{\TypeIwMaxEpsilon}{\ensuremath{\epsilon}}
\newcommand{\myneg}[1]{\ensuremath{\neg{#1}}}
\newcommand{\myexp}[1]{\ensuremath{e^{#1}}}

\newcommand{\mymin}[2]{\ensuremath{\min\left\{#1, \; #2\right\}}}
\newcommand{\girgprob}[2]{\mymin{1}{\left(\frac{\weightof{#1} \weightof{#2}}{{\geomdist{x_{#1}}{x_{#2}}}^{d}}\right)^{\alpha}}}

\newcommand{\beady}{alternating}

\newcommand{\rumourtime}[1]{\ensuremath{\mathcal{T}_r\left(#1\right)}}

\renewcommand{\Pr}[1]{\ensuremath{\mathbb{P} \left(#1 \right) }}
\newcommand{\eps}{\varepsilon}

\newcommand{\Vol}{\ensuremath{\text{Vol}}}

\newcommand{\N}{\mathbb{N}}
\newcommand{\Z}{\mathbb{Z}}

\newcommand{\R}{\mathbb{R}}
\newcommand{\T}{\mathbb{T}}

\newcommand{\calL}{\mathcal{L}}

\newcommand{\calT}{\mathcal{T}}
\newcommand{\calU}{\mathcal{U}}
\newcommand{\calV}{\mathcal{V}}

\newcommand{\calX}{\mathcal{X}}

\newcommand{\polyloglogweight}{w_{H_1}}

\newcommand{\beadydeg}{\ensuremath{k}}

\newcommand{\timeblock}{\ensuremath{Z}}

\newcommand{\gaptimeblock}[2]{\ensuremath{GT_{#1}\left(#2\right)}}

\newcommand{\edgetimeblock}[2]{\ensuremath{ET_{#1}\left(#2\right)}}

\newcommand{\pr}{\mathbb{P}}

\newcommand{\Var}[1]{\ensuremath{\mathrm{Var}\left[#1\right]}}

\newcommand{\myDelta}{\Delta}

\newcommand{\ball}[3]{\mathcal{B}^{#3}_{#1}(#2)}

\renewcommand{\epsilon}{\eps}
\renewcommand{\phi}{\varphi}

%% file: abstract.tex
We study push-pull rumour spreading in ultra-small-world models for social networks where the degrees follow a power-law distribution. In a non-geometric setting, Fountoulakis, Panagiotou and Sauerwald have shown that rumours always spread ultra-fast~(SODA~2012). On the other hand, Janssen and Mehrabian have found that rumours spread slowly in a spatial preferential attachment model~(SIDMA~2017). We study the question systematically for the model of Geometric Inhomogeneous Random Graphs (GIRGs), which has been found to be a good theoretical and empirical fit for social networks. Our results are two-fold: first, with classical Euclidean geometry \emph{slow, fast and ultra-fast} (i.e.\ polynomial, polylogarithmic and doubly logarithmic number of rounds) rumour spreading may occur, depending on the exponent of the power law and the strength of the geometry in the networks, and we fully characterise the phase boundaries between these regimes. The regimes do not coincide with the graph distance regimes, i.e., polylogarithmic or even polynomial rumour spreading may occur even if graph distances are doubly logarithmic. We expect these results to hold with little effort for related models, e.g.\ Scale-Free Percolation. Second, we show that rumour spreading is always (at least) fast in a non-metric geometry. 
The considered non-metric geometry allows to model social connections where resemblance of vertices in a single attribute, such as familial kinship, already strongly indicates the presence of an edge. Classical Euclidean geometry fails to capture such ties.

For some regimes in the Euclidean setting, the efficient pathways for spreading rumours differ from previously identified paths. For example, a vertex of degree $d$ can transmit the rumour efficiently to a vertex of larger degree by a chain of length $3$, where one of the two intermediaries has constant degree, and the other has degree $d^{c}$ for some constant $c<1$. Similar but longer chains of vertices, all having non-constant degree, turn out to be useful as well.

%% file: introduction.tex

Rumour spreading is a classical process on graphs used to model the dissemination of information in networks. Starting in one vertex, the rumour spreads to other vertices by travelling over edges in a round-based protocol. In each round, every vertex chooses a neighbour uniformly at random, and the two vertices share the rumour if one of them already knows it. In other words, if an informed vertex chooses an uniformed neighbour, then it pushes the rumour to that neighbour, and if an uninformed vertex chooses an informed neighbour, then it pulls the rumour from that neighbour. This is probably the most well-studied version of rumour spreading and is also called the \emph{synchronous push-pull protocol}.

A central question is how the structure of the underlying graph influences the speed of rumour spreading. Throughout this article, we call rumour spreading \emph{slow} if it takes $n^{\Omega(1)}$ rounds to inform $\Omega(n)$ vertices when starting from a random vertex, \emph{fast} if this takes $(\log n)^{O(1)}$ rounds, and \emph{ultra-fast} if it takes $O(\log \log n)$ rounds. Rumours have been shown to spread fast in cliques~\cite{frieze1985shortest,karp2000randomized}, hypercubes~\cite{feige1990randomized}, Erd\H{o}s-R\'enyi random graphs~\cite{fountoulakis2010reliable,elsasser2006communication}, random regular graphs~\cite{fountoulakis2013rumor}, Cayley graphs~\cite{elsasser2007broadcasting}, vertex expanders~\cite{giakkoupis2012rumor,fountoulakis2013rumor}, and graphs with high conductance~\cite{chierichetti2018rumor}. They even spread ultra-fast in preferential attachment (PA) graphs~\cite{doerr2011social,doerr2012rumors,chierichetti2011rumor} and Chung-Lu (CL) random graphs~\cite{fountoulakis2012ultra}. In showing this ultra-fast spread, one builds paths where every other node has small degree. In such a path using $k$ edges, the (random) number of rounds needed for rumour transmission is dominated by the sum of $k$ independent geometric random variables, each expectation being equal to the minimum degrees among the two vertices incident to the corresponding edge. We also make use of such constructions, among others.

Some of this work, especially on CL and PA random graphs, was specifically motivated by social networks, and at their time those graphs were considered state-of-the-art models for social networks due to their heavy-tail degree distribution. However, they exhibit some severe deviations from real-world social networks. They have very small clustering coefficients, meaning that the number of triangles is small, and they have no strong communities (vertex subsets that induce low conductance), while real-world social networks have large clustering coefficients, strong communities and small conductances~\cite{leskovec2008statistical}.

While PA and CL models remain useful to understand many aspects of social networks, the last years have shown that they are not well-suited for modelling spreading processes on these networks, as those processes may be strongly impacted by the underlying geometric clustering of the networks. As an example, it is well-known that the number of cases during an infection may spread either exponentially or polynomially fast in real-world networks~\cite{polyepidemicsurvey}. However, spreading in PA or CL models is always at least exponential (or the process dies out) since it can be coupled to branching processes in which spreading is at least exponential. Thus, those models can only show part of the picture.

To alleviate these shortcomings, improved models of social networks have been developed by combining the heavy-tail degree distribution of PA and CL models with an embedding of the vertices in an underlying geometric space. The resulting models include geometric inhomogeneous random graphs (GIRG)~\cite{bringmann2019geometric}, hyperbolic random graphs (HRG)~\cite{krioukov2010hyperbolic}, scale-free percolation (SFP)~\cite{deijfen2013scale}, and spatial preferential attachment (SPA)~\cite{AieBonCooJanss08,jacob2015spatial}. It turns out that these models retain desirable properties from CL and PA graphs, such as ultra-small typical distances (i.e.\ of order $\log\log n$), while showing an impressive number of other properties that are empirically observed in real-world networks. Those include large clustering coefficients~\cite{gugelmann2012random,bringmann2019geometric,fountoulakis2021clustering}, strong communities~\cite{bringmann2019geometric}, small entropy, good compressibility~\cite{bringmann2019geometric}, and good navigability~\cite{bringmann2022greedy}. They have also been shown to reproduce closely the behaviour of algorithms that are known to behave well in practice, for example greedy routing schemes~\cite{boguna2010sustaining}, sublinear algorithms for shortest paths~\cite{blasius2024external,cerf2024balancedbidirectionalbreadthfirstsearch} and other algorithms for diameter, vertex cover, graph clustering, maximal cliques and chromatic numbers~\cite{blasius2024external}. Moreover, they provide a much richer and more plausible phase diagram for epidemiological modelling~\cite{komjathy2023four1,komjathy2023four2}. For a general discussion of their expressivity as social network models, we refer the reader to~\cite{bläsius2018towards}.

Overall, these models have proven to be a much better fit for social networks than the earlier CL and PA models. This motivates us to revisit the question of how fast rumours spread in social networks, and whether the ultra-fast rumour spreading from CL and PA graphs is preserved in those newer models. Note that the generic rumour spreading results for vertex expanders and graphs with large conductance do not apply since neither the models~\cite{bringmann2019geometric,kiwi2018spectral} nor real-world social networks~\cite{leskovec2008statistical} fall into those categories. The only result so far in this direction is a study by Janssen and Mehrabian, who showed that rumour spreading is slow on certain instances of spatial preferential attachment graphs~\cite{janssen2017rumors}. This is at odds with the general observation that rumours can at least sometimes spread fast in social networks.

In this article we consider the time until the rumour has spread to a constant fraction of vertices in the unique giant component of a GIRG, which is arguably the most natural notion in the context of social networks.\maybeextra{\footnote{In different contexts the time to inform all vertices may be more important, for example for the task of disseminating information in a distributed computing network such as an IT infrastructure.}} Moreover, we always condition on the starting vertex being in the giant component, which is the unique component containing a linear number of vertices. 

\paragraph{\textbf{Our contribution.}} We study the synchronous push-pull protocol in (two versions of) the GIRG model, one of the more modern models for social networks mentioned above. Compared to the study on a SPA model in~\cite{janssen2017rumors}, we allow a much richer family of graphs.\footnote{The GIRG model and the SPA model in~\cite{janssen2017rumors} are not directly comparable as they use different kernels for the connection probability, but the GIRG model comes with a much broader palette of configurations. In particular, the model in~\cite{janssen2017rumors} has no parameter for weak ties. For a comparison of different kernels we refer the reader to the overview in~\cite{gracar2022recurrence}, but note that the PA kernel there also deviates from the kernel in~\cite{janssen2017rumors}. Since the GIRG model (which includes the HRG model as special case~\cite{bringmann2016average}) has been much more extensively validated with real-world data, we consider it the more suitable choice. See also~\cite{dayan2024expressivity,boguna2010sustaining,blasius2018efficient} for different approaches to map real-world networks into GIRG/HRG.} In the GIRG model, the vertex set is given by a Poisson point process (PPP) of intensity $n$ in a $d$-dimensional unit hypercube, each vertex draws a random weight from a power-law distribution with exponent $\tau >2$, and two vertices $u,v$ of weights $w_u,w_v$ at distance $r$ from each other are connected with probability $\Theta(\min\{1,\frac{w_uw_v}{n\Vol(B_r)}\}^\alpha)$, where $B_r$ is a ball with radius $r$, and the parameter $\alpha>1$ controls the prevalence of \emph{weak ties}~\cite{granovetter1973strength}. It is important to note that the weight of a vertex is up to constant factors equal to its expected degree. The parameter $\alpha$ is also called the \emph{inverse temperature} of the model. Throughout the paper we will assume that the parameters $\tau$, $\alpha$ and $d$ are constant. We now give a more formal and complete definition.

We will always consider simple undirected graphs whose vertex set $\mathcal{V}$ is the result of a Poisson Point Process of intensity $n$ on the $d$-dimensional torus $\T^d$ (i.e., the number of vertices is Poisson-distributed with expectation $n$, and their positions are chosen independently and uniformly at random). We denote the edge set by $\mathcal{E}$. 
In order to define the graph models we consider, we first require the definition of a power-law-distributed random variable.

\begin{defn} \label{def:power-law}
    Let $\tau>1$. A discrete random variable $X\ge 1$ is said to follow a \emph{power-law with exponent} $\tau$ if $\prob{X = x} = \Theta(x^{-\tau})$. A continuous random variable $X\ge 1$ is said to follow a \emph{power-law with exponent} $\tau$ if it has a density function $f_X$ satisfying $f_X(x) = \Theta(x^{-\tau})$.
\end{defn}
There exist both stronger and weaker versions of power-law distributions~\cite{voitalov2019scale}. Assuming a density function removes some technical complexity, but would not be necessary for our results.

In GIRGs, every vertex samples its position in a ground space, which we describe next.

\paragraph{\textbf{Ground space.}} As our ground space we use the $d$-dimensional torus $\T^d\coloneqq\R^d/\Z^d$. Intuitively, this is the $d$-dimensional unit hypercube $[0,1]^d$, where opposite faces are identified. In this topology, we can use the absolute difference to measure coordinate-wise distances of vertices. Let  $a,b \in [0,1]$, then we define $\geomdist{a}{b}_T \coloneqq \min\{\geomdist{a}{b}, 1-\geomdist{a}{b}\}$.
This enables us to define the two distance functions we will use to construct our graphs. Let $x_v, x_u \in \T^d$ with coordinates $x_v=(x_{v,1}, \dots, x_{v,d})$, $x_u=(x_{u,1}, \dots, x_{u,d})$. We then denote the $\infty$-norm on $\T^d$ by
\begin{align*}
   |x_u-x_v|_{\infty}\coloneqq\max_{1\le i\le d}\geomdist{x_{u,i}}{x_{v,i}}_T.
\end{align*}
We will also use the Minimum-Component Distance (MCD):
\begin{align*}
     |x_u-x_v|_{\min}\coloneqq\min_{1\le i\le d}\geomdist{x_{u,i}}{x_{v,i}}_T.
\end{align*}
Note that each induces a measurable volume function, $V_{\infty}(r)\coloneqq \Vol(\{x\in \T^d \mid |x|_{\infty}\le r\})$ and $V_{\min}(r)\coloneqq \Vol(\{x\in \T^d \mid |x|_{\min}\le r\})$ respectively, which is translation-invariant, symmetric and surjective. Observe that the volume function for the $\infty$-norm scales as $V_{\infty}(r)=\Theta(r^d)$ and the MCD-volume function scales as $V_{\min}(r)=\Theta(r)$ for $r\to 0$~\cite{lengler2017existence}. 

We are now in a position to define (MCD-)Geometric Inhomogeneous Random Graphs.

\begin{defn}\label{def:simple-girg}
Let $\tau>2$, $\alpha>1$, $0 < \theta_1 \le \theta_2 \le 1$, $d,n\in\N$, and let $\mathcal{D}$ be a power-law distribution on $[1,\infty)$ with exponent $\tau$. A \emph{(MCD-)Geometric Inhomogeneous Random Graph ((MCD-)GIRG)} is obtained by the following three-step procedure:
    \begin{enumerate}
        \item The vertex set $\calV$ is given by a Poisson Point Process of intensity $n$ on $\T^d$. For the sake of clarity, we write $v$ to refer to the vertex, and $x_v$ to refer to its position in $\T^d$.
        
        \item Every vertex $v\in\mathcal{V}$ draws i.i.d.\ a \emph{weight} $W_v \sim \mathcal{D}$.

        \item For every two distinct vertices $u,v \in \mathcal{V}$, add an edge between $u$ and $v$ in $\mathcal{E}$ independently with probability $p_{u,v} = p_{u,v}(x_u,x_v,W_u,W_v)$, such that for $m = \mymin{\left(\frac{W_uW_v}{n V(|x_u-x_v|)}\right)^{\alpha}}{1}$,
        \begin{align}\label{eq:girg-connection}
            \theta_1 m \le p_{uv}(x_u,x_v,W_u,W_v) \le \theta_2 m,
        \end{align}
        where setting $V(|x_u-x_v|)\coloneqq V_{\infty}(|x_u-x_v|_{\infty})$ induces a GIRG and setting $V(|x_u-x_v|)\coloneqq V_{\min}(|x_u-x_v|_{\min})$ induces an MCD-GIRG.
    \end{enumerate}
\end{defn}

 It should be noted that the parameters $\theta_1, \theta_2$ only affect the hidden constants in Landau notation of our results. Another choice we have made is to use a Poisson Point Process instead of simply sampling exactly $n$ points u.a.r.\ from the ground space. This only makes the proofs less technical; the same results can be shown for the alternative definition. Having defined the model, we can now start the exposition of our results.

\smallskip\noindent\textbf{Euclidean GIRGs.}
We show that rumour spreading on GIRGs can be slow, fast, or ultra-fast, depending on the exact model parameters. The following holds whp\footnote{Whp means with high probability, i.e., with probability $1-o(1)$ as $n\to \infty$.} for Euclidean GIRGs:
\begin{itemize}
    \item If $\alpha < \frac{1}{\tau-2}$, the rumour spreads \textbf{ultra-fast} (in $\myO{\log{\log{n}}}$ rounds) along paths which alternate between vertices of constant degree and vertices of doubly exponentially increasing degree. These paths are similar to the ones for CL and PA paths.
    \item If $\tau < \frac{5}{2}$, the rumour spreads \textbf{ultra-fast} in a similar way to the previous case, only that the constant degree vertex is replaced by a sequence of vertices with carefully chosen non-constant degrees. The closer $\tau$ is to $\frac{5}{2}$, the more vertices each such degree-increasing subpath uses.  We are not aware of previous work which uses such paths for rumour spreading. 
    \item If $\tau > \frac{5}{2}$ and $\alpha > \frac{1}{\tau - 2}$, then rumour spreading is \textbf{at most fast}: we show that at least $\Omega(\frac{\log n}{\log \log n})$ rounds are needed to spread the rumour to a constant fraction of vertices.
    \item If $\alpha < \frac{\tau-1}{\tau-2}$ or $\tau <\phi+1$, where $\phi \approx 1.618$ is the golden ratio, then the rumour spreads \textbf{at least fast} (in poly-logarithmic time) along paths that form a hierarchical structure. The idea - originally used for analysing graph distances in~\cite{biskup2004scaling} - is to find a long edge that covers most of the distance between the starting vertex and some other vertex, which breaks down the original problem into two smaller problems, and use this iteratively to construct a hierarchy. Notably, the long edges are asymmetric. In the case $\alpha < \frac{\tau-1}{\tau-2}$ they connect a vertex of large degree with a vertex of constant degree. In the case $\tau <\phi+1$ they connect two vertices of large degree, but one of them has substantially smaller degree than the other. This imbalance allows the rumour to be transmitted more efficiently.
    \item If both $\alpha > \frac{\tau-1}{\tau-2}$ and $\tau > \phi+1$  then rumour spreading is \textbf{slow} (takes polynomial time). We prove the polynomial lower bound by showing that no geometrically long edge is used in the first $n^{\myOmega{1}}$ steps, following an idea  from~\cite{janssen2017rumors}. 
\end{itemize}
Note that the above points combined fully characterize the speed of rumour spreading into the three categories, apart from boundary cases. The different regimes are visualized in~\Cref{fig:tau_alpha_regions}.

\begin{figure}[t]
\centering
\begin{tikzpicture}
\begin{axis}[
    width=9cm,
    height=6cm,
    xmin=2, xmax=3,
    ymin=1, ymax=3,
    xlabel={$\tau$},
    ylabel={$\alpha$},
    domain=2:3,
    samples=200,
    thick,
    axis on top,
]

\def\tfivehalf{2.5}
\def\tphi{2.618} 

\path[name path=alpha_one_line] (axis cs:2,1) -- (axis cs:2.5,1);
\path[name path=topline] (axis cs:2,3) -- (axis cs:2.5,3);
\addplot[green!30] fill between[
    of=alpha_one_line and topline,
    soft clip={domain=2:2.5}];

\addplot[name path=ultrafastBound, domain=2.5:3] {1/(x-2)};
\path[name path=alpha_one_line_2] (axis cs:2.5,1) -- (axis cs:3,1);
\addplot[green!30] fill between[
    of=alpha_one_line_2 and ultrafastBound,
    soft clip={domain=2.5:3}];

\path[name path=topline2] (axis cs:2.5,3) -- (axis cs:2.618,3);
\addplot[name path=ultraBoundShort, domain=2.5:2.618] {1/(x-2)};
\addplot[yellow!30] fill between[
    of=ultraBoundShort and topline2,
    soft clip={domain=2.5:2.618}];

\addplot[name path=fastLow, domain=2.618:3] {1/(x-2)};
\addplot[name path=fastHigh, domain=2.618:3] {(x-1)/(x-2)};
\addplot[yellow!30] fill between[
    of=fastLow and fastHigh,
    soft clip={domain=2.618:3}];

\path[name path=topline3] (axis cs:2.618,3) -- (axis cs:3,3);
\addplot[red!30] fill between[
    of=fastHigh and topline3,
    soft clip={domain=2.618:3}];

\addplot[domain=2.5:3, black, very thick, dashed] {1/(x-2)}
node[pos=0.5, above right, yshift=-1pt] {$\frac{1}{\tau - 2}$};

\addplot[domain=2.618:3, black, very thick, dotted] {(x-1)/(x-2)}
node[pos=0.6, above left, yshift=7pt] {$\frac{\tau - 1}{\tau - 2}$};

\draw[dotted, thick] (axis cs:2.5,1) -- (axis cs:2.5,3) node[above] {$\tau=2.5$};
\draw[dotted, thick] (axis cs:2.618,1) -- (axis cs:2.618,3) node[above] {$\tau=\phi+1$};

\node at (axis cs:2.3,2.5) {\textbf{Ultra-fast}};
\node at (axis cs:2.56,2.5) {\textbf{Fast}};
\node at (axis cs:2.9,2.5) {\textbf{Slow}};

\end{axis}
\end{tikzpicture}
\caption{Utrafast, fast, and slow rumour spreading in Euclidean GIRGs based on $\tau$, $\alpha$.}
\label{fig:tau_alpha_regions}
\end{figure}
\smallskip\noindent\textbf{Non-Euclidean GIRGs (MCD-GIRGs).}
All the above results were for Euclidean geometry. \maybeextra{\footnote{The results are robust with respect to constant-factor deviations. Since all norms on $\R^d$ are equivalent under this lens, the results also hold for every other norm on $\R^d$.}} However, GIRGs can also be defined for other geometries~\cite{bringmann2016average}. Of particular interest is the \emph{minimum component distance} (MCD), defined as $|x-y|_{\min} \coloneqq \min_{1\le i \le d} |x_i-y_i|$. This distance function is arguably a better model for social networks since it is not restricted by the triangle inequality, see the paragraph on MCD-GIRGs below for a discussion. We show that the choice of underlying \textbf{geometry does influence whether rumour spreading is fast}. More precisely, we show that rumour spreading on sufficiently super-critical\footnote{I.e., for ranges in which there is a giant component.} MCD-GIRGs is always fast, and is ultra-fast whenever typical distances are ultra-small. The proof uses paths that alternate between vertices of small degree and (in the ultra-small regime) vertices of increasingly large degree, where two consecutive edges are always ``aligned'' along two different dimensions of the underlying space, taking advantage of the fact that ``balls'' induced by the minimum-component distance are cross-shaped.
\paragraph{\textbf{Background on MCD-GIRGs.}}\label{sec:intro-MCD} The minimum component distance defined above does not satisfy the triangle inequality. Take for instance the origin $x_0=(0,0)$ in $\R^2$, which has MCD $0$ from both $x_1 = (0,1)$ and $x_2 = (1,0)$, but the MCD between $x_1$ and $x_2$ is $|x_1-x_2|_{\min} = 1$. Arguably, this violation of the triangle inequality makes the MCD more realistic for social networks. The geometry is supposed to capture properties of the nodes, for example their professions, their hobbies, family ties, and so on. In real-world networks, two nodes may be considered close if they are very similar along \emph{at least one} of those axes. For example, two researchers are close and likely know each other if they work on similar topics, even if they have vastly different hobbies. Moreover, if person A and person B are close in terms of profession (colleagues), and B is close to person C in terms of family ties (e.g., siblings), then there is no reason to believe that A is close to C: in general, we do not know the siblings of our co-workers, nor do we have overlapping social spheres with them beyond this one co-worker. Hence, the underlying geometry of social networks does not follow the triangle inequality, making the MCD a more realistic proximity measure.

GIRGs equipped with the minimum-component distance still share many properties with other GIRGs. In particular, the degree distribution and the ultra-small world properties hold regardless of the geometry~\cite{bringmann2016average}. Moreover, MCD-GIRGs also have a large clustering coefficient of $\Omega(1)$ because the MCD satisfies a \emph{stochastic version} of the triangle inequality: in the $d$-dimensional unit cube, for any position $x$ and any radius $r<1$, if we pick two positions $x_1,x_2$ uniformly at random from $B_r^{\text{MCD}}(x)\cap [0,1]^d$ then $\prob{\geomdist{x_1}{x_2}\le 2r} \ge \frac{1}{d} = \Omega(1)$. This holds because the probability that $x_1$ and $x_2$ are close to $x$ along \emph{the same} dimension is at least $\frac{1}{d}$. This suffices to induce a large clustering coefficient in MCD-GIRGs. On the other hand, in contrast to Euclidean GIRGs, MCD-GIRGs do not have sublinear edge separators of the giant component~\cite{lengler2017existence}.

Nevertheless, the aim of this paper is not to argue whether MCD or Euclidean GIRGs (or interpolations between the two~\cite{kaufmann2024sublinear}) are globally a better model. Instead, we study both versions to show that the underlying geometry plays a crucial role, and depending on the geometry and the model parameters, slow regimes may or may not be present. Altogether, our results show that rumour spreading can be a much more complex process than indicated by previous research.

\paragraph{\textbf{Notation and Terminology.}} We write $[n] \coloneqq \{1,..., n\}$. We denote by $\log$ the natural logarithm, by $\log_2$ the logarithm with base $2$, and by $\log^{*k}$ the $k$-fold iterated logarithm, e.g.\ $\log^{*3}x = \log\log\log x$. For $w\in \mathbb{R}_{\ge 0}$, we denote by $\mathcal{V}_{\ge w}\coloneqq\{v \in \mathcal{V}\mid W_v\ge w\}$ and $\mathcal{V}_{\le w}\coloneqq\{v \in \mathcal{V}\mid W_v\le w\}$ the subset of vertices of weight at least $w$ respectively at most $w$. The sets $\mathcal{V}_{> w}$ and $\mathcal{V}_{< w}$ are defined similarly. We use $\isedge{u}{v}$ to say that $u, v$ are connected by an edge. We define\footnote{Notice that this definition is slightly different than the one in~\cite{biskup2004scaling}.} $\myDelta(\gamma) \coloneqq 1/\log_2{(1/\gamma)}$. 
The symbol $\phi \approx 1.618$ refers to the golden ratio. When the minimum in the edge probability of an edge is $1$, i.e.\ when $W_u W_v \ge n V(|x_u-x_v|)$, we call such an edge \emph{strong}, otherwise we call it \emph{weak}. For a vertex $u$, we call the ball of volume $\min\{1,W_u/n\}$ around $x_u$ the \emph{ball of influence} of $u$ and denote it by $BoI(u)$. Notice that all edges from $u$ to vertices within this ball are strong.

\paragraph{\textbf{Some useful properties of GIRGs.}} Studying the rumour spreading process meaningfully requires some form of connectivity from the considered graph. While GIRGs are generally not connected, it is shown in~\cite{bringmann2016average} that with high probability they have a unique giant component, i.e.\ a connected component with size $\myOmega{n}$, while all other connected components have size at most $\log^{\myO{1}}{n}$. Thus, we study the case where the rumour starts in some random vertex in the giant.

Degrees play a significant role in our scenario and while we actively look for edges in order to build paths, we also want to upper bound the degrees of the involved vertices. Of crucial importance for this is that the degree of a vertex $u$ with weight $W_u$ is Poisson-distributed with expectation $\Theta(W_u)$. In particular, if $W_u$ is large then the degree is sharply concentrated around its expectation. More precisely, it was shown in~\cite{bringmann2016average} that the marginal edge probability between a vertex $u$ with fixed position and weight $W_u$ and some vertex $v$ of weight $W_v$ whose position is still random is $\myTheta{\frac{W_u W_v}{n}}$. Thus we can assume that \whp{} all degrees are within a $\log{n}$ factor of their expectation, and also more detailed such claims can be shown, e.g.~\Cref{lem:degrees_stay_small_ultrafast}. Since weights are power-law-distributed with exponent $\tau >2$, the expected weight/degree is $\Theta(1)$. On the other hand, the maximum weight is roughly $n^{1/(\tau - 1)}$ in expectation. Moreover, a constant fraction of the neighbours of a vertex have constant weights and similarly a constant fraction of them are in the ball of influence of the vertex.
\maybeextra{A property that we will frequently use is that a Poisson Point Process is independent across disjoint geometric regions, as well as across disjoint weight ranges. That is, one can e.g.\ expose $\mathcal{V}_{< w}$ without influencing (the randomness of) $\mathcal{V}_{\ge w}$. Similarly, we can expose all vertices inside some geometric ball and later independently expose vertices outside of this ball. For the purpose of showing our upper bounds, we generally expose some set of vertices restricted by some geometric and/or weight condition, and show that \emph{within} this set of vertices degrees remain small. Then, once we have identified the vertices constituting our path, we expose the rest of the graph and use the concentration bounds to control any further increase in the degrees of those vertices.}

\paragraph{\textbf{Further motivation for the choice of the model.}} Closely related variants of the GIRG model have been introduced (sometimes independently) and intensively studied in different communities, ranging from Physics (Hyperbolic Random Graphs HRG~\cite{krioukov2010hyperbolic}) to Mathematics (Scale-Free Percolation SFP~\cite{deijfen2013scale}). Apart from showing that the model is natural, this also means that we expect our results to transfer to those variants. In particular, HRG is even a special case of GIRG, so all results automatically hold for HRG. Moreover, the model has two main tunable parameters, relating to the degree distribution ($\tau$) and the prevalence of weak ties ($\alpha$) in the network, both of which are considered very important in social networks~\cite{strengthOfWeakTies,barabasi1999emergence}. We show that these parameters matter for the speed of rumour spreading. Vertices in the graph have a geometric position, which encodes their traits, making the geometry fairly interpretable. Moreover, this geometry can also be modified in the model (Euclidean Distance vs.\ Minimum-Component-Distance, as well as other geometries~\cite{lengler2017existence,sublinearCutsException}), while still being interpretable. GIRG also lends itself quite nicely to mathematical analysis. All those features make the model in our eyes more suited than other models of social networks, like the Forest Fire Model~\cite{leskovec2007}, Kronecker graphs~\cite{leskovec2009}, or Affiliation Networks~\cite{lattanzi2009}. 

\paragraph{\textbf{Outline of the paper.}}
The rest of the paper is organised as follows.~\Cref{sec:results} gives an overview of our results. We outline the arguments used in our proofs in~\Cref{sec:tools}. We then proceed to treat the different regimes in Euclidean GIRGs. In~\Cref{sec:ultrafast_euclidean}, we prove the upper bound for the ultra-fast regime, followed by the proof of the polylogarithmic upper bound in~\Cref{sec:polylogupper}. We conclude our analysis of Euclidean geometry in~\Cref{sec:loglower,sec:polynomiallower}, where we prove the roughly logarithmic and polynomial lower bound, respectively. Finally,~\Cref{sec:MCD} covers our results for MCD-GIRGs.

%% file: results.tex
Our results show that in Euclidean GIRGs, whether rumour spreading is slow or fast or even ultra-fast depends on the decay rate of the power-law and the prevalence of weak ties, even for regimes in which graph distances are ultra-small. On the other hand, when the Minimum-Component Distance governs the formation of edges, rumour spreading is aligned with graph distances: rumours always spread fast, and they spread ultra-fast if graph distances are ultra-small, which is the case for $\tau <3$.

Our first main result concerns the Euclidean setting and states that, given sufficiently many weak edges - requiring more precisely that $\alpha <\frac{\tau - 1}{\tau - 2}$, rumours spread fast, taking at most polylogarithmic time. Recall the function $\myDelta{}(\gamma) = 1/\log_2(1/\gamma)$.

\begin{thm} \label{thm:intro_polylog}
    Assume that 
    $\alpha <\frac{\tau - 1}{\tau - 2}$, let $\mathcal{G}=(\mathcal{V},\mathcal{E})$ be a GIRG and $u, v$ be two vertices in the giant component, chosen uniformly at random. Let $\gamma = \frac{\alpha (\tau - 1)}{\alpha + \tau - 1}$ and denote by $\mathcal{T}$ the event that the rumour is transmitted from $u$ to $v$ via the push-pull protocol within the first $\left(\log{n}\right)^{\myDelta{}(\gamma) +o(1)}$ rounds. Then, $\mathcal{T}$ happens with high probability.
\end{thm}

If the distribution of the weights has a heavy enough tail, one can also deduce that rumours spread fast regardless of the geometric penalty for long edges.

\begin{restatable}{thm}{fast_strong}\label{thm:fast_strong}
    Assume that 
    $\tau < \phi + 1$, let $\mathcal{G}=(\mathcal{V},\mathcal{E})$ be a GIRG and $u, v$ be two vertices in the giant component, chosen uniformly at random. Let $\gamma = \frac{\tau(\tau - 1)}{2\tau - 1}$ and denote by $\mathcal{T}$ the event that the rumour is transmitted from $u$ to $v$ via the push-pull protocol within the first $\left(\log{n}\right)^{ \myDelta{}(\gamma) + \myo{1}}$ rounds. Then, $\mathcal{T}$ happens with high probability.
\end{restatable}

If both large-degree vertices and weak ties are sufficiently rare, that is, when $\tau > \phi + 1$ and simultaneously $\alpha>\frac{\tau-1}{\tau-2}$, rumours spread slowly in the network, taking polynomial time to reach even a large sublinear number of vertices in the network.

\begin{restatable}{thm}{slowTheorem} \label{thm:intro_slow}
    Assume $\tau > \phi + 1$ and $\alpha>\frac{\tau-1}{\tau-2}$. Let $\mathcal{G}=(\mathcal{V},\mathcal{E})$ be a GIRG and let the vertex where the rumour starts be arbitrary. There exist $\slowTheoremEpsilon_1, \slowTheoremEpsilon{}_2 > 0$ such that \whp{} at most $\myO{n^{1-\slowTheoremEpsilon{}_1}}$ vertices are informed of the rumour after $n^{\slowTheoremEpsilon_2}$ rounds of the push-pull protocol.
\end{restatable}

Apart from establishing this tight phase transition from fast to slow, we also establish the precise criteria for rumour spreading to be ultra-fast in Euclidean GIRGs. We first show that if the effect of the geometry is weak enough, rumours spread ultra-fast.

\begin{restatable}{thm}{ultrafastGIRGWeak}\label{thm:ultrafastGIRG_weak}
Assume that $\alpha < \frac{1}{\tau - 2}$. Let $\mathcal{G}=(\mathcal{V},\mathcal{E})$ be a GIRG and $u, v$ be two vertices in the giant component, chosen uniformly at random. Denote by $\mathcal{T}$ the event that the rumour is transmitted from $u$ to $v$ via the push-pull protocol within the first $(4 + o(1)) \cdot \frac{\log \log n}{|\log (\alpha(\tau - 2))|}$ rounds. Then, $\mathcal{T}$ happens with high probability.
\end{restatable}

The next two theorems utilize high-weight vertices to spread the rumours ultra-fast. The constructions are similar, but we split them into two theorems for technical reasons and because the control of the leading constant in the latter is less accurate.\

\begin{restatable}{thm}{ultrafastGIRGStrong}\label{thm:ultrafastGIRG_strong}
Assume that $\tau < \sqrt{2} + 1$. Let $\mathcal{G}=(\mathcal{V},\mathcal{E})$ be a GIRG and $u, v$ be two vertices in the giant component, chosen uniformly at random. Denote by $\mathcal{T}$ the event that the rumour is transmitted from $u$ to $v$ via the push-pull protocol within the first $(6 + o(1)) \cdot \frac{\log \log n}{|\log{(\tau(\tau - 2))}|}$ rounds. Then, $\mathcal{T}$ happens with high probability.
\end{restatable}

\begin{restatable}{thm}{ultrafastGIRGnew}\label{thm:ultrafastGIRGnew}
Assume that $\tau < \frac{5}{2}$. Let $\mathcal{G}=(\mathcal{V},\mathcal{E})$ be a GIRG and $u, v$ be two vertices in the giant component, chosen uniformly at random. There exists some $C(\tau)$ such that for $\mathcal{T}$ being the event that the rumour is transmitted from $u$ to $v$ via the push-pull protocol within the first $(C(\tau) + o(1)) \cdot \log \log n$ rounds, we have that $\mathcal{T}$ happens with high probability.
\end{restatable}

We show that the previous constructions are optimal in the sense that ultra-fast transmission cannot happen if not through them. In particular, we show the following roughly logarithmic lower bound. Its proof is inspired by techniques in~\cite{gracar2022chemical, Gracar_2021, Lakis_2024} and uses a first moment method argument.

\begin{restatable}{thm}{loglowerboundthm}\label{thm:loglowerboundthm}
Assume that $\tau > \frac{5}{2}$ and $\alpha > \frac{1}{\tau - 2}$. Let $\mathcal{G}=(\mathcal{V},\mathcal{E})$ be a GIRG and suppose the rumour starts in a random vertex. There exists a $\rho > 0$ such that \whp{} at most $O(n^{1 - \rho})$ vertices are informed of the rumour after $k = \Omega(\frac{\log{n}}{\log{\log{n}}})$ rounds of the push-pull protocol.
\end{restatable}

This concludes the results on Euclidean GIRGs, which are also summarized in~\Cref{fig:tau_alpha_regions}. In geometries induced by the Minimum-Component Distance in dimension $d\ge 2$, we observe ultra-fast rumour-spreading for the entire range $2<\tau<3$. 

\begin{thm} \label{thm:intro_mcd_ultrafast}

Let $2<\tau<3$ and $d\ge 2$. Consider an MCD-GIRG $\mathcal{G}=(\mathcal{V},\mathcal{E})$  on $\T^d$ and two vertices $u, v$ in the giant component, chosen uniformly at random. Denote by $\mathcal{T}$ the event that the rumour is transmitted from $u$ to $v$ via the push-pull protocol within the first $(4 + o(1)) \cdot \frac{\log \log n}{|\log (\tau - 2)|}$ rounds. Then, $\mathcal{T}$ happens with high probability.
\end{thm}

If $\tau \ge 3$ then typical distances in MCD-GIRGs are $\Omega(\log n)$, since the $k$-neighbourhood of a vertex grows at most exponentially. This is usually considered the less interesting case, since most social networks have power-law exponents between $2$ and $3$~\cite{latora2017complex}. Moreover, in general there may not be a linear-sized component if the average degree is small. However, we do note that for sufficiently large average degree, rumours do spread in MCD-GIRGs in time $O(\log n)$. 

\begin{thm} \label{thm:intro_mcd_fast}
For all $\tau >2$ and $d\ge 2$ there is a constant $C$ such that the following holds. Consider an MCD-GIRG $\mathcal{G}=(\mathcal{V},\mathcal{E})$ on $\T^d$ in which the expected degree of a vertex of weight $1$ is at least $C$. Let $u, v$ be two vertices in the giant component, chosen uniformly at random. Then \whp{}, the rumour is transmitted from $u$ to $v$ via the push-pull protocol within the first $O(\log n)$ rounds.
\end{thm}


%% file: tools.tex
\input{hierarchies}

\input{polylog_figures}

\input{weight_increasing_paths}

\input{ultrafast_figures}

\input{timing}

\paragraph{\textbf{Logarithmic lower bound.}} We use the first moment method here. In particular, we first couple the rumour spreading process on GIRGs into another graph model $\mathcal{G'}$ in such a way that is suffices to lower bound graph distances in the latter. Then, we use techniques also employed in~\cite{Lakis_2024, Gracar_2021} to show that the probability that $u$ spreads the rumour to a random vertex $v$ is $\myO{n^{-\myOmega{1}}}$. In particular, we do so by formulating an inductive statement on the probability of $u$ and $v$ transmitting the rumour between them over a path of length $i$. The aforementioned probability bound implies that the expected number of vertices that learn the rumour is $\myO{n^{1 - \myOmega{1}}}$, showing~\Cref{thm:loglowerboundthm}.

\paragraph{\textbf{Polynomial lower bound.}} This proof is inspired by arguments presented in \cite{janssen2017rumors}. Let us pretend for a moment that all edges in the graph cover a geometric distance of at most $n^{- \longEdgeEpsilon{}}$, where $\longEdgeEpsilon > 0$. After $n^{\longEdgeEpsilon / 2}$ rounds, the rumour cannot have geometrically traversed a distance of more than $\myO{n^{\longEdgeEpsilon / 2} n^{- \longEdgeEpsilon{}}} = \myO{n^{- \longEdgeEpsilon{} / 2}}$ by the triangle inequality\maybeextra{(note that this argument does not work for MCD-GIRGs, since in this geometry the triangle inequality does not hold)}. The number of vertices that are within such a distance from the rumour origin are \whp{} $\myO{n^{1-d \longEdgeEpsilon / 2}} = \myo{n}$.\maybeextra{Hence, in such a scenario, the polynomial lower bound would be straightforward.}

In GIRGs, long edges do exist, as indicated by the ultra-small graph distances. However, they are not \emph{used} in the first $n^{\longEdgeEpsilon / 2}$ rounds.\maybeextra{More precisely, consider the minimum degree of the incident vertices of a long edge.} For a single round, if $\mymin{\degreeof{u}}{\degreeof{v}} = x$, the probability that $e = (u, v)$ is used is at most $2/x$\maybeextra{, by a simple union bound over $u$ and $v$ choosing that edge}. Thus, we will count long edges based on the minimum degree of their endpoints, a quantity which we will refer to as the \emph{\slowdown{}} of the edge. Intuitively, the condition $\alpha > \frac{\tau - 1}{\tau - 2}$ ensures that no long edges of smaller \slowdown{} than some threshold $n^{\rho}$ exist. The other condition 
$\tau > \phi + 1$ guarantees that for large values $x \ge n^\rho$ the number of edges with \slowdown{} at least $x$ is at most $x^{1-\eps}$, so the probability to use such an edge in a given round is at most $2x^{-\eps} \le 2n^{-\rho\eps}$. By a union bound, no such edge will be used in the first $n^{\rho\eps/2}$ rounds.

\paragraph{\textbf{Upper bounds in MCD-GIRGs.}} 
For rumour spreading on MCD-GIRGs, we show in Appendix~\ref{sec:MCD} that the spreading time to a random vertex is at most $\frac{4+o(1)}{|\log(\tau-2)|}\log\log n$ if $2 < \tau < 3$. The spreading remains fast for $\tau \ge 3$ if the average degree is sufficiently large. Both cases agree with graph distances in MCD-GIRGs up to constant factors. Notably, in the ultra-fast regime our bound exceeds the one for average distances by a factor of $2$, similarly to the bounds for CL graphs~\cite{fountoulakis2012ultra}. We first sketch the proof for the ultra-small regime.

The ball of influence of a vertex $u$ in the MCD geometry can be thought of as the union of $d$ ``(hyper-)plates'', where each plate consists of all the vertices which are close to $u$ in a given dimension. See Figure \ref{fig:mcd} for an illustration of the plates of influence and for one spreading step. Assume now that the rumour has already reached a vertex $u$ of weight $W_u = \omega(1)$ and consider a ``plate of influence'' along dimension $1$. Then $u$ has $\Omega(W_u)$ neighbours of constant weight in this plate. The crucial fact here is that the coordinates of these neighbours are randomly distributed along all other dimensions. Thus, in the next step we can look at a plate along dimension $2$ for all the newly exposed vertices, and all of their positions are independent along this dimension. Hence, most of their balls of influence do not overlap along dimension $2$. Thus, together those vertices have $\Omega(W_u)$ neighbours along this dimension, and the largest of them has weight $W_u^{1/(\tau-2)-\eps}$, similarly as for CL graphs~\cite{fountoulakis2012ultra}. Again, their coordinates along dimension $1$ are uniformly random, so we can iterate this argument. Thus, alternating between dimensions $1$ and $2$, in two steps the weight increases exponentially,  by an exponent $1/(\tau-2)-\eps$, which allows us to reach a vertex of maximal weight in $\frac{2+o(1)}{|\log(\tau-2)|} \log\log n$ steps. Since every second vertex on the path has constant weight, the technique described in the ``Efficient Transmission'' paragraph above can be used to show that the rumour is indeed transmitted ultra-fast. More precisely, it takes $\frac{2+o(1)}{|\log(\tau-2)|}\log\log n$ rounds to reach a vertex of maximum weight, and using the symmetry of the push-pull protocol it needs another $\frac{2+o(1)}{|\log(\tau-2)|}\log\log n$ rounds to inform a randomly chosen target vertex from the vertex of maximum weight, yielding a total time of $\smash{\frac{4+o(1)}{|\log(\tau-2)|}\log\log n}$.

Notice that we decisively relied on the underlying geometry when using the uniform distribution of the neighbours of a vertex within its plate of influence. This guaranteed that the neighbours of neighbours in the plate of influence along another dimension do not overlap significantly. In metric geometries like the Euclidean geometry however, this is not true, which is why the argument fails.

For $\tau \ge 3$, the argument is essentially the same, except that now we only use vertices of constant weight. We still alternate between dimensions $1$ and $2$, but now we use that $x$ vertices with independent and random coordinates along dimension $1$ have in expectation at least $2x$ neighbours along dimension $2$, and vice versa.\maybeextra{This holds because we assume that the expected degree is sufficiently large.} Hence, we can couple the process to a branching process, which needs time $O(\log n)$ to reach $\Omega(n)$ vertices. 

\input{mcd_figures}

\input{bulk_lemma}

%% file: hierarchies.tex
\paragraph{\textbf{Fast regime, upper bound.}} For the fast regime, we will construct efficient communication paths via structures that have been introduced by Biskup~\cite{biskup2004scaling} in order to bound graph distances in the case $\alpha <2$. To better follow the argument, we will work with renormalised distances, i.e.\ we blow up all distances by a factor of $n^{1/d}$, and redefine the connection probability accordingly. Then the density is $1$, i.e., the nearest neighbour of a vertex is typically in distance $\Theta(1)$. The idea for graph distances goes as follows. Let us assume we are trying to connect two vertices $u$ and $v$ with $\geomdist{x_u}{x_v} = \hierarchyN{}$ (which is much larger than $1$) via a short path. Consider two balls $B_{u}, B_{v}$ of radius $\hierarchyN{}^{\hierarchyGamma{}}$ around $u$ and $v$ respectively. The parameter $\hierarchyGamma{} \in (0, 1)$ is used in building the paths and can vary depending on the model parameters. We are interested in the event where $B_{u}$ is connected to $B_{v}$ in the sense that there exist vertices $u' \in B_{u}$ and $v' \in B_{v}$ with $\isedge{u'}{v'}$. The vertices $u'$ and $v'$ could be of any weight, and let us for now disregard any requirements we might want to impose on them (but this will change later).
In the original argument, as long as $\alpha < 2$, there exists a $\hierarchyGamma{} < 1$ such that \whp{} this event occurs. In effect, we have reduced the problem of connecting $u, v$ with distance $\hierarchyN{}$ to the two subproblems of connecting $u$ to $u'$ and $v$ to $v'$ with distances $\hierarchyN{}^{\hierarchyGamma{}}$. Iterating this argument up to some level of recursion gives rise to polylogarithmic distances. More precisely, after $D:= \log_{1/\hierarchyGamma{}} \log L$ iterations the distances of the remaining problems are reduced to $\smash{L^{\gamma^D}} = \Theta(1)$. We obtain $2^{D}$ subproblems of constant size, so the length of the resulting path is roughly $2^D$, which is of the order $(\log{L})^{1/\log_2{(1/\gamma)}} = (\log{L})^{\myDelta{}(\gamma)}$. The exponent increases with increasing $\gamma$. For the case of graph distances, a more detailed calculation is given in~\cite{biskup2004scaling}, where $\gamma$ can be set arbitrarily close to $\alpha/2$. The collection of edges identified throughout this procedure is called the \emph{hierarchy} and by a \emph{gap} we refer to one of the smallest subproblems we are left with. The task of connecting these gaps is dealt with differently. In our case, we will not rely on the condition $\alpha < 2$ to build the hierarchies reliably, but instead we will exploit either $\alpha < \frac{\tau -1}{\tau - 2}$ or $\tau < \phi + 1$, yielding~\Cref{thm:intro_polylog,thm:fast_strong}. The first corresponds to weak edges between low-weight vertices and high-weight vertices and the second corresponds to strong edges between mid-weight and high-weight vertices.  See~\Cref{fig:polylog_weak,fig:polylog_strong} for a visual aid to the hierarchy building steps explained below.

When an endpoint of the edge has low weight, one simply needs to build the path and then \whp{} the rumour will be transmitted fast, since each edge will be used for communication often (because of the small degree of one endpoint). However, in the second case, this argument is not enough. Instead, we will ``hunt'' for communication events between $B_u, B_v$ and not just edges. We can no longer rely on a single efficient edge having a constant probability of being used in each round, but rather we will consider a set of (many) medium-weight vertices in $B_u$ for which it holds that in each round \emph{some} vertex in that set will communicate with the heavy weight vertex in $B_v$ \whp{}. \maybeextra{In particular, we show that whp there is a medium-weight vertex which transmits the rumour to its neighbour in $B_v$ in the next round after learning the rumour.} 
\maybeextra{For the sake of uniformity, we will actually use the same approach for both variants of the paths. That is, we will look for edges as well as corresponding communication events (for some appropriate round of the push-pull protocol) for these edges when building the hierarchy.}

%% file: polylog_figures.tex
\begin{figure}[!ht]
    \centering
    
    \begin{minipage}[t]{0.49\textwidth}
    \captionsetup{width=0.99\textwidth}
        \centering
        \includegraphics[width=\linewidth]{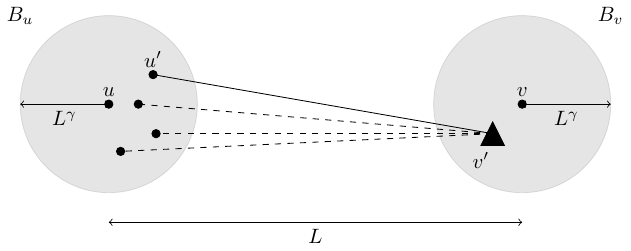}
        \caption{Case $\alpha < \frac{\tau - 1}{\tau - 2}$: the vertex $v'$ has weight approximately $L^{d \gamma \frac{1}{\tau - 1}}$, which is the maximum weight expected in $B_{v}$. All vertices depicted inside $B_{u}$ have constant weight. Most of these vertices are \emph{not} connected to $v'$, but at least one is \whp{}, here depicted as $u'$. The rumour has a constant chance of being pushed from $u'$ to $v'$ in any round.}
        \label{fig:polylog_weak}
    \end{minipage}
    \hfill
    \begin{minipage}[t]{0.49\textwidth}
    \captionsetup{width=0.99\textwidth}
        \centering
        \includegraphics[width=\linewidth]{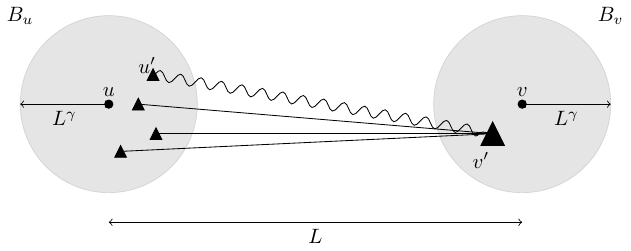}
        \caption{Case $\tau < \phi + 1$: the vertex $v'$ is found in the same way as in~\Cref{fig:polylog_weak}. Now, the candidate vertices inside $B_{u}$ have weight $w_{mid}$ such that $W_{v'} w_{mid} \ge L^{d}$. This means that a constant fraction of them are connected to $v'$. Now, most of them do \emph{not} push the rumour to $v'$ fast (after receiving it), but there is one that does, depicted here as $u'$.}
        \label{fig:polylog_strong}
    \end{minipage}
\end{figure}

%% file: weight_increasing_paths.tex
\paragraph{\textbf{Ultra-fast regime, upper bound.}} The usual approach for showing doubly logarithmic distances, or more generally transmission times, in models with scale-free degree distribution is to successively find constant-length paths from a vertex $u_w$ of weight $w$ to another vertex $u_{w'}$ of weight $w' = w^{1 + \eps}$, for some fixed $\eps > 0$. A rough outline of how this is achieved with e.g.\ paths of length one is the following. Assume for now that $\theta_1$ from~\Cref{def:simple-girg} is $1$. At a distance of $n^{-1/d}w^{(1/d)(1 + \eps')(\tau - 1)}$ from $u_w$ one expects to find at least one vertex with weight $\myTheta{w'}$, provided $\eps'$ is large enough compared to $\eps$. Now, the product of $w$ and $w'$ is $\myTheta{w^{2 + \eps}}$, whereas the denominator in the connection probability in~\Cref{def:simple-girg} is $w^{(1 + \eps')(\tau - 1)}$. Since $\tau < 3$, for appropriate choices of $\eps, \eps'$ the former is larger than the latter and therefore $\isedge{u_w}{u_{w'}}$. Repeating this for $\myO{\log{\log{n}}}$ steps we reach a vertex of weight $n^{c}$, for a desired $c < \frac{1}{\tau - 1}$. \maybeextra{This restriction on $c$ comes from the fact that even the largest weight in the graph is at most $n^{1/(\tau - 1) + \myo{1}}$, so we cannot hope for something higher (also notice that the ``search'' distance in the weight-increasing step must be at most $1$).}When one simply wants to build paths for regular graph distances, they can reach from $u$ and $v$ two vertices $u', v'$, both with weight $\myOmega{n^{1/2}}$ (note that $\frac{1}{2} < \frac{1}{\tau - 1}$ for $\tau < 3$). These vertices of weight $n^{1/2}$ are part of an induced Erd\H{o}s-R\'enyi graph with constant connection probability, where the diameter is at most 2 \whp{}. In our case, we are not simply looking for paths, but for paths that transmit the rumour efficiently. A visual explanation of the arguments to follow (in the Euclidean case) can be found in~\Cref{fig:ultrafast_weak,fig:ultrafast_strong}. In the case of MCD-GIRGs, without any restrictions on $\alpha$ and simply assuming $\tau < 3$ enables us to find a vertex $u_c$ of constant weight such that $u_w$ connects to $u_{w'}$ through $u_c$. The vertex $u_c$ pulls the rumour from $u_w$ and pushes it to $u_{w'}$ in constant time\maybeextra{(we actually find many such vertices and some transmits the rumour in 2 rounds)}. This is only possible in Euclidean GIRGs when the geometry is sufficiently soft, as now the constant-weight neighbours of $u_w$ are close together geometrically and not uniformly spread in space as for MCD-GIRGs. In particular, this construction works when $\alpha < \frac{1}{\tau - 2}$. Another way to build weight-increasing paths is through 3-hop paths which work as follows. We fix the position of vertex $u_w$ and find in a distance $n^{-1/d}w^{(1 + \eps)(\tau - 1)/d +\myo{1}}$ the vertex $u_{w'}$. Inside the ball of influence of $u_w$, we find a large set of vertices $V_{mid}$ of weight $w^{\tau - 2 + \eps'}$. Notice that the product of this weight with $w'$ is at least the distance from these vertices to $u_{w'}$ raised to the power $d$ (for $\eps'$ large enough given $\eps$), so a constant fraction of them are connected to it \whp{}. Moreover, any vertex in $V_{mid}$ has many common neighbours of constant weight with $u_w$ (in the intersection of their balls of influence), so it learns the rumour in 2 rounds. It turns out that at least one of the vertices in $V_{mid}$ then pushes the rumour to $u_{w'}$ in a single round, provided that $\tau < \sqrt{2} + 1$\maybeextra{, because then the size of $V_{mid}$ is large enough to compensate for the ``degree penalty'' of these vertices}. The same argument shows that once one reaches two very high-weight vertices, they communicate in 3 rounds. This argument again works without any assumption on $\alpha$, since it is only based on strong edges. A further optimization of this construction shows that weight-increasing paths exist already if $\tau < \frac{5}{2}$, and these paths typically use more than 3 but still a constant number of rounds for each step.\maybeextra{The central idea is to quickly reach as many vertices of weight roughly $w^{\tau - 2}$, i.e.\ maximize $|V_{mid}|$. Instead of doing this through constant weight vertices, we can use say some low weight ($\approx w^{c}$ for some $c < \tau - 2$) vertices. This allows us to geometrically search for $V_{mid}$ further out than $BoI(u_w)$. Then, the next step is to use even lower weight vertices to reach those of weight $w^c$ (thus having paths of length $4$ connecting $u_w$ to the higher weight vertex) and so on. This construction becomes better and better as we allow more intermediary vertices (i.e.\ reaches more and more vertices of weight $w^{\tau - 2}$), and $\tau < \frac{5}{2}$ is the necessary condition for some version of this construction to work.} Details on this are given towards the end of~\Cref{sec:ultrafast_euclidean}.  Notice that in all such arguments, there is a maximum $\eps$ such that the next vertex of weight $w^{1 + \eps}$ is reachable. This particular value and the number of edges per step define the leading constants we show in~\Cref{thm:ultrafastGIRG_weak,thm:ultrafastGIRG_strong,thm:ultrafastGIRGnew,thm:intro_mcd_ultrafast}.

%% file: ultrafast_figures.tex
\begin{figure}[!ht]
    \centering
    
    \begin{minipage}[t]{0.49\textwidth}
    \captionsetup{width=0.99\textwidth}
        \centering
        \includegraphics[width=\linewidth]{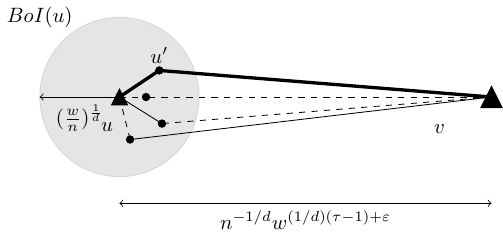}
        \caption{Case $\alpha < \frac{1}{\tau - 2}$: At the depicted distance we expect to find at least one vertex $v$ of exponentially higher weight than the current vertex $u$, which has weight $w$. Now, in the Ball of Influence ($BoI$) of $u$, there exist $\Theta(w)$ many vertices of constant weight, and $u$ connects to each of them with constant probability. Similarly, $v$ connects to a sufficient fraction of the vertices in $BoI(u)$ such that $u$ and $v$ have a common neighbour of constant weight which transmits the rumour fast, depicted here as $u'$.}
    \label{fig:ultrafast_weak}
    \end{minipage}
    \hfill
    \begin{minipage}[t]{0.49\textwidth}
    \captionsetup{width=0.99\textwidth}
        \centering
        \includegraphics[width=\linewidth]{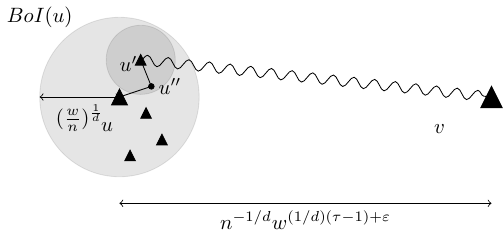}
        \caption{Case $\tau < \sqrt{2} + 1$: Similarly to~\Cref{fig:ultrafast_weak}, we find the vertex $v$ of exponentially higher weight. Now, we also find many vertices in $BoI(u)$ of weight approximately $w^{\rho}$ with $\rho = \tau - 2 < 1$ such that they have strong edges with $v$ (not shown to avoid clutter). For most such vertices, their $BoI$ lies within $BoI(u)$, and thus they have a constant-weight neighbour in common with $u$, allowing the rumour to reach them in two rounds. The transmission path shown: $u \to u'' \to u' \to v$.}
    \label{fig:ultrafast_strong}
    \end{minipage}
\end{figure}

%% file: timing.tex
\maybeextra{\paragraph{\textbf{Efficient Transmission.}} We repeatedly use vertices of given weight and want that those vertices transmit the rumour efficiently. This is indeed true if the \emph{degree} of the vertices is small. However, when building the weight-increasing paths, we can only control the \emph{weight} of the involved vertices. To overcome this, we use the following slightly subtle exposure method. 
For an edge $e=\{u,v\}$, we define $T_e$ to be the first round in which at least one of its endpoint learns the rumour. We may assume that this vertex is $u$. For simplicity, let us assume that $W_u \le W_v$ so that push is the more efficient operator, but the argument works in either direction. In round $T_e +1$, the vertex $u$ pushes the rumour to a random neighbour, and we may choose this random neighbour by the following selection process. We order the neighbours of $u$ from $1$ to $\deg(u)$, where we may assume that $v$ comes first in this ordering. Then we draw $Q_e = Q_{e,u} \in [0,1]$ uniformly at random,\label{def:Q_e} and $u$ pushes the rumour to the $\lceil Q_e \cdot\deg(u)\rceil$-th neighbour. Note crucially that the random variable $Q_e$ is independent of the value of $T_e$. In particular, we may draw it before revealing the rest of the graph, i.e., before knowing the value of $T_e$. (For completeness we also draw $Q_e$ if $u$ is not in the connected component of the rumour, in which case $T_e =\infty$ and the process does nothing along $e$.) While building the path, for all edges $e$ we will also reveal the random variables $Q_e$, and pick edges which additionally satisfy $Q_e \le 1/D_e$ for a suitable value $D_e$. This guarantees that whenever $\deg(u) \le D_e$ then the rumour will spread from $u$ to $v$ in the round after $u$ learns the rumour, i.e.\ we have no waiting time along such edges. We will choose $D_e$ such that whp the condition $\deg(u) \le D_e$ is satisfied for all edges on the path, with a few exceptions: we cannot guarantee that the initial part of the connecting path, before reaching some vertex of weight $w_0 =\omega(1)$, transmits along each edge in one round, and similarly for the last part of the paths. So for those two short sub-paths we instead compute the expected time to transmit the rumour and use Markov's inequality to show that whp the rumour does not take much longer on these parts.}

%% file: mcd_figures.tex
\begin{figure}[ht]
    \begin{minipage}[b]{1\textwidth}
    \centering
    \includegraphics{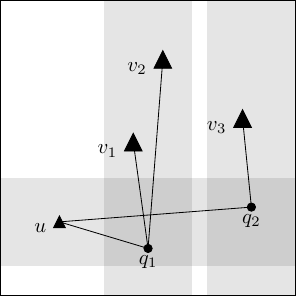}
%
%
%
%
%
%
%
%
%
%
%
        \caption{MCD case: the vertex $u$ currently considered has weight $w$. The vertices $q_1$ and $q_2$ are of constant weight and lie in the plate of influence of $u$ along dimension 1. Looking at their plates of influence along dimension 2, we find vertices $v_1$, $v_2$ and $v_3$ of weight approximately $w^{1/(\tau-2)}$.}\label{fig:mcd} 
    \end{minipage}
\end{figure}

%% file: bulk_lemma.tex
\newcommand{\wzero}{\ensuremath{w_0}}
\paragraph{\textbf{Spreading the rumour to a non-constant-weight vertex.}} In order to build hierarchies or weight-increasing paths successfully with high probability, the weights of the vertices we consider must be large enough, since they enter in the bounds for the probability of failure. However, when picking a random vertex in the graph, its weight is smaller than any growing function with high probability, due to the power-law distribution of weights. We must therefore show that the rumour quickly reaches \emph{some} vertex of sufficiently high weight $\wzero$. From then on, our other arguments succeed with high probability. It suffices to set $\wzero = \myomega{1}$. This problem has already been solved in several other papers, and we use the solution developed in~\cite{komjathy2023four1, komjathy2023four2}. By appropriately using Claim 3.11 and Corollary 3.9(iv) from~\cite{komjathy2023four2} and Lemma 6.4 from~\cite{komjathy2023four1} (to handle the case $d = 1$), one can prove the following lemma. We formulate it under blown-up distances (as explained in the proof outline for the fast regime) because it is clearer in this case. 



\begin{lemma}\label{claim:go_to_wzero}
    Let $u$ be a vertex chosen uniformly at random from the giant component and $\wzero = \myomega{1}$ a weight with $\wzero = \myo{n^{\frac{1}{\tau - 1}}}$. There exist functions $r(\wzero), W(\wzero)$ and $s(\wzero)$ such that exposing the vertices of distance at most $r(\wzero)$ from $u$ and of weight at most $W(\wzero)$ \whp{} reveals a path from $u$ to a vertex $u_{\wzero}$ with weight in the range $[\wzero, W(\wzero)]$ with length at most $s(\wzero)$.
\end{lemma}

Note that the lemma refers to the exposure of the entire weight range $[\wzero, W(\wzero)]$. However, we will usually expose the vertices by increasing weight until we find a path to one with weight $w \ge \wzero$. In this way we keep the weight range $[w, W(\wzero)]$ unexposed, which is used in later arguments.

%% file: ultrafast_euclidean.tex
In this section we give the proof that rumours spread ultra-fast in GIRGs if $\alpha < \frac{1}{\tau - 2}$ or $\tau < \frac{5}{2}$. In the process of building the path that will transmit the rumour, one must take care that the degrees of the involved vertices do not grow beyond a reasonable bound, or in any case somehow guarantee that the edges identified will actually be used for transmission of the rumour when required. The general strategy we employ is the following. We expose parts of the graph in such a way that by the end of the path's construction, all vertices have bounded degrees \emph{within the exposed graph}. Moreover, we can expose some additional information $Q_e$ about the edges (concerning the randomness of the rumour spreading protocol) that ensures deterministically that each selected vertex will transmit the rumour in one round after learning the rumour, provided that its degree grows by at most $\myO{\log{x}}$  when exposing the remaining graph, where $x$ is the length of this part of the path. Finally, we show that \whp{} all degrees grow indeed by at most this factor.
Now, we only rely on this local degree control for the first few steps of the weight-increasing path (we essentially have $x =\myTheta{\log^{*3}{n}}$). Even though this first part is asymptotically negligible, it is quite technical and we need to split it into three phases. In the first phase, we reach some weight $w_0$ by exposing only weights $\le w_0$ in a certain geometric ball using~\Cref{claim:go_to_wzero}. We have very little control about the number of steps or the size of the ball, except that both are asymptotically negligible if we choose $w_0$ small enough. Afterwards, we use vertices of weight $>w_0$ to reach a region outside of the ball, in which we have not exposed vertices of constant weight, so that we have fresh randomness for the rest of the construction. Finally, we use a weight-increasing path to reach a vertex of polylog weight. This initial construction is the same for all regimes and proofs, so we only spell out the details once.


We now come to the parts that are specific to the ultra-fast regime. We may assume that the weight is some large enough polylog value. Then we find so many 2-step paths to the next high-weight vertex with a constant-weight vertex in the middle that even if all their degrees grow by a $\myTheta{\log{n}}$ factor w.r.t.\ their weights (we can ``globally'' assume before exposing any part of the graph that all degrees grow by at most this factor), at least one of the intermediaries will pull and push the rumour in two fixed rounds. In particular, we start using this argument once we are able to find $\myOmega{(\log{n})^3}$ such paths, since then we have an expected $\myOmega{\log{n}}$ of successful two-round transmissions, and a Chernoff bound also applies. This careful handling of the latter steps allows us to give a precise leading constant in our upper bound. This general control of the degrees and communication events is the most technical part in our proofs (see also~\Cref{sec:MCD}), and to facilitate the exposition we will only spell out this technical aspect for the case $\alpha < \frac{1}{\tau - 2}$ and $d = 1$. We will later sketch the extension to higher dimensions as well as some minor qualitative differences in the proof for the cases $\tau < \sqrt{2} + 1$ and $\sqrt{2} + 1 \le \tau < \frac{5}{2}$. See~\Cref{fig:ultrafast_weak,fig:ultrafast_strong} for a visual aid to the proofs for the first two cases.

We now come to the formal proof. For the following, we assume $d=1$ and that the destination vertex $v$ is geometrically to the right of $u$. For dimensions larger than $1$, we will replace this later by a suitable set of boxes. We introduce now some notation. We denote by $\ball{u}{r}{L}$, $\ball{u}{r}{R}$ the set of vertices with distance at most $r$ from $u$ to the left/right of it respectively. We also use $\ball{u}{r}{} = \ball{u}{r}{L} \cup \ball{u}{r}{R}$. The first step to building our weight-increasing paths is to spread the rumour to a vertex of non-constant weight $w_0$. We need to do this because the probability of failure for each weight-increasing step depends on the weight of the starting vertex. In particular, for the failure probability to be $\myo{1}$, the starting weight must be $w_0=\myomega{1}$. This is enough to guarantee the success of all the steps, as the weights increase doubly exponentially and so the failure probability of the whole path is dominated by that of the first step. We do this using~\Cref{claim:go_to_wzero} in the proof of the main theorem later.

Assuming we have reached the required vertex of weight $\wzero$, we need to ``escape'' the original geometric region we have exposed when using~\Cref{claim:go_to_wzero}. We only use vertices of weight higher than $\wzero$ to do this. This means that by the end of this second step, we have reached a place where all the constant-weight vertices in the vicinity are still unexposed.

\begin{lemma}\label{lem:escape_geometric_region}
    Assume that for some $\wzero = \myomega{1}$ only the vertices of weight up to $w \le W(\wzero)$ and at a distance of at most $r(\wzero)$ from $u$ have been exposed to reveal a path from vertex $u$ to a vertex $u_{\wzero}$ of weight $w \ge \wzero$, where $r(\wzero), W(\wzero)$ are as in~\Cref{claim:go_to_wzero}. Then, exposing the vertices with weight in $[w, n r(\wzero)^{1/(\tau - 1)}]$ reveals \whp{} a further weight-increasing path of length $\myO{\log{\log{\left(nr(\wzero)\right)}}}$ from $u_{\wzero}$ to a vertex $u_{out}$ to the right of $u$ such that $W_{u_{out}} = \myTheta{nr(\wzero)^{\frac{1}{\tau-1}}}$ and $\geomdist{x_{u_{out}}}{x_u} > r(\wzero) + r(\wzero)^{1/(\tau - 1)}$. In particular, all vertices of constant weight with distance at most $r(\wzero)^{1/(\tau - 1)}$ from $u_{out}$ are still unexposed at the end of this procedure.
\end{lemma}

\begin{proof}
    The path we look for is simply a weight-increasing path without using any intermediaries, exactly as in CL or PA paths. That is, given a current vertex $v_{curr}$ (initially $v_{curr} = u_{\wzero}$) of weight $w_{curr}$, we expose the vertices of weight in the range $[w_{curr}^{1 + \eps}, 2w_{curr}^{1 + \eps}]$ to the right of $v_{curr}$ until we find one that directly connects to it (we expose the vertices by increasing distance to $v_{curr}$, starting at distance $n^{-1}w_{curr}^{(1 + \eps)(\tau - 1) + \eps}$). Once we find such a vertex, we repeat the process with that new vertex as $v_{curr}$. Let us go through the calculations to see that we find these vertices with high probability. For a single step (where we try to find the next vertex connecting to $u_{curr}$), we expose the weight range $[w_{curr}^{1 + \eps}, 2w_{curr}^{1 + \eps}]$ in $\ball{u_{curr}}{2n^{-1}w_{curr}^{(1 + \eps)(\tau - 1) + \eps}}{R} \setminus \ball{u_{curr}}{n^{-1}w_{curr}^{(1 + \eps)(\tau - 1) + \eps}}{R}$. Since the volume of this geometric region is $n^{-1}w_{curr}^{(1 + \eps)(\tau - 1) + \eps}$, we expect to find $\myTheta{w_{curr}^{(1 + \eps)(\tau - 1) + \eps} w_{curr}^{(1 + \eps)(1 - \tau)}} = \myTheta{w_{curr}^{\eps}}$ many vertices in this weight range. This number is also Poisson-distributed, so we find at least a constant fraction of the expected number of vertices with high probability. For small enough $\eps$, these vertices have constant probability to connect to $u_{curr}$, independently of each other. This is because the product of the weights is $w_{curr}^{2 + \eps}$ while the corresponding term in the denominator of the connection probability is $w_{curr}^{(\tau - 1)(1 + \eps) + \eps}$ and $\tau - 1 < 2$. Again, this means that with high probability we find at least one vertex that connects to $u_{curr}$. Because $\wzero = \myomega{1}$, this path construction fails with $\myo{1}$ probability (the failure probability in the first step dominates all subsequent failure probabilities). At the final step we reach a weight $\myTheta{n r(\wzero)^{\frac{1}{\tau - 1}}}$ at distance (from $u$) larger than $r(\wzero)^{1 + \eps} = \myomega{r(\wzero) + r(\wzero)^{\frac{1}{\tau - 1}}}$. Note that for each weight-increasing step we expose vertices of larger and larger weight, so the steps are independent. 
\end{proof}

Now we may start exposing constant-weight and high-weight vertices alternatingly to implement the weight-increasing steps. To ensure fresh randomness, we keep track of the geometric regions in which we have exposed vertices. At each step, we will only have exposed vertices of high weight in the current geometric region and no vertices far enough to the right. This means that we can always newly expose constant-weight vertices around the current geometric region and higher weight vertices in a new geometric region far to the right of the current one.

To illustrate the process more clearly (along with~\Cref{fig:weight_increasing_step}), let us walk through the first steps. Let $x$ be the position of vertex $u_{out}$ (which has weight $\myTheta{w_{out}} = nr(\wzero)^{\frac{1}{\tau - 1}}$). Note that by~\Cref{lem:escape_geometric_region}, the vertices of constant weight in the ball of influence of $u_{out}$ (i.e.\ with position in $[x - n^{-1}w_{out}, x + n^{-1}w_{out}]$) are still unexposed. Also, all vertices to the right of $x + n^{-1} w_{out}^{(1 + \beta)(\tau - 1) + \eps}$ remain unexposed. For a single step, we expose constant-weight vertices in $[x - n^{-1} w_{out}, x + n^{-1} w_{out}]$ and vertices of weight $\myTheta{w_{out}^{(1 + \beta)}}$ with position in $\left[x +  n^{-1}w_{out}^{(1 + \beta)(\tau - 1) + \eps},  x + 2 n^{-1} w_{out}^{(1 + \beta)(\tau - 1) + \eps}\right]$. It is shown in~\Cref{lem:weight_increasing_step_weak} that this succeeds \whp{} in transmitting the rumour to a vertex of weight $\myTheta{w_{out}^{(1 + \beta)}}$. Notice that the constant-weight vertices around this new high-weight vertex remain unexposed, and so do all vertices far enough to the right of it. We can therefore repeat this step, increasing the weight doubly-exponentially.

\begin{figure}[ht]
\centering
\begin{tikzpicture}[x=1.3cm,y=0.9cm,>=stealth,shorten >=1pt,shorten <=1pt]

\draw[->] (-0.1,0) -- (8,0) node[pos=-0.1,below] {Position};
\draw[->] (0,-0.1) -- (0,6) node[left] {Weight};

\node[circle,fill=black,inner sep=1.2pt] (uout) at (2,3) {};
\draw[dashed] (0,3) -- (2,3);
\node[left] at (0,3) {$w_{out}$};

\draw [decorate,decoration={brace,amplitude=4pt,mirror}] (1,0) -- (3,0) 
      node[midway,yshift=-9pt]{\small $[\,x - n^{-1} w_{out},\, x + n^{-1} w_{out}\,]$};

\node[rectangle,draw=blue,fill=blue,inner sep=1.2pt] (c1) at (1.2,0.2) {};
\node[rectangle,draw=blue,fill=blue,inner sep=1.2pt] at (1.7,0.2) {};
\node[rectangle,draw=blue,fill=blue,inner sep=1.2pt] at (2.6,0.2) {};

\draw [decorate,decoration={brace,amplitude=4pt,mirror}] (5.5,0) -- (7.5,0)
      node[midway,yshift=-9pt]{\small $\bigl[\,x + n^{-1} w_{out}^{(1+\beta)(\tau -1)+\eps},\, x + 2\, n^{-1} w_{out}^{(1+\beta)(\tau -1)+\eps}\bigr]$};

\node[circle,draw=red,fill=red,inner sep=1.3pt] (newV) at (6,5) {};
\draw[dashed] (0,5) -- (6,5);
\node[left] at (0,5) {$\Theta\bigl(w_{out}^{1+\beta}\bigr)$};

\draw[densely dotted] (2,0) -- (2,3);
\draw[densely dotted] (6,0) -- (6,5);

\draw [->,thick] (uout) -- (c1);
\draw [->,thick] (c1) -- (newV);

\end{tikzpicture}
\caption{Illustration of one weight-increasing step. 
Vertex $u_{out}$ (of weight $\myTheta{w_{out}}$, marked black dot) transmits the rumour first to a constant-weight vertex (in blue)
in the interval $[\,x - n^{-1} w_{out}, x + n^{-1} w_{out}\,]$ and then on to a vertex
of weight $\Theta\bigl(w_{out}^{1+\beta}\bigr)$ (in red), located in the interval 
$\bigl[\,x + n^{-1} w_{out}^{(1+\beta)(\tau -1)+\eps}, x + 2\,n^{-1} w_{out}^{(1+\beta)(\tau -1)+\eps}\bigr]$.}
\label{fig:weight_increasing_step}
\end{figure}

Note that we need to guarantee that when exposing $\myO{n'}$ constant-weight vertices in a region of volume $n'/n$ and their in-between edges to form the induced graph $G_{n'}$, at least a constant fraction of these vertices will have constant degree within $G_{n'}$. This is encapsulated in the following lemma, which we prove using a standard second moment argument.

\begin{lemma}\label{lem:secondmomentlemma}
Let $G_{n'}$ be a subgraph sampled as follows. We randomly select a position u.a.r.\ for each of $Cn'$ vertices inside a hypercube of volume $n'/n$, for some constant $C$. The weight of each vertex is upper-bounded by some constant $w_c$. Fix a sufficiently small constant $c$ and a sufficiently large constant $C_d$, both depending on $C, w_c$, and let $p$ be the probability that at least a $c$-fraction of the vertices of $G_{n'}$ have degree at most $C_d$ \emph{within $G_{n'}$}. Then,
\[p \ge 1 - \myO{1/n'}.\]
\end{lemma}
\begin{proof}
Let $Z$ denote the number of edges of $G_{n'}$, which is half the sum of degrees by the handshake lemma. It is easy to see that $\expectationof{Z} = \myTheta{n'}$. We will now show that $\Var{Z} \le \myO{\expectationof{Z}}$. This will imply the claim because then the sum of the degrees is at most a constant times the expectation with the required probability, by Chebyshev inequality. So, in particular the $c$ fraction of vertices with the smallest degrees must have degree at most $C_{d}$. Otherwise, the rest $1 - c$ fraction would contradict the bound on $Z$. 

Now, let us simply bound $\Var{Z}$. Note that $Z = \sum_{u \neq v}Z_{uv}$, where $Z_{uv} = 1$ if and only if $u$ is connected to $v$. We use the usual equality $\Var{Z} = \expectationof{Z^2} - \expectationof{Z}^2$ and notice that the term $\expectationof{Z^2}$ contains three types of $Z_e, Z_{e'}$ combinations as expectations. First, it contains terms $\expectationof{(Z_e)^2}$. These contribute a term $\expectationof{Z}$ in total because $(Z_e)^2 = Z_e$ for indicator variables. Another case is $\expectationof{Z_e Z_{e'}}$, where $e$ and $e'$ share no vertices. These are independent, and thus all of these terms together give us at most a term $\expectationof{Z}^2$, which cancels with the negated same one. The final case is when $e, e'$ share exactly one vertex, say $u_c$. Then, $\expectationof{Z_e Z_{e'}}$ is $\myO{(n')^{-2}}$ for the following reason. Conditioning on the position and weight of $u_c$ makes the variables $Z_e, Z_e'$ independent. Moreover, the probability that both of them are $1$ is then $\myTheta{(n')^{-2}}$.
Since there are at most $\myO{(n')^{3}}$ such combinations of edges, the overall contribution is $\myO{n'} =  \myO{\expectationof{Z}}$, which finishes the proof.
\end{proof}

Using \Cref{lem:secondmomentlemma}, we find many new vertices of constant weight that additionally have small degree in between them. Of these vertices, we want to keep those that also do not connect to \emph{any} previous vertex identified when building the path (except for the very early ones found by~\Cref{claim:go_to_wzero} and those used to escape the initial geometric region). We will show that \whp{} still a constant fraction of vertices remain when we apply this ``filter''. Moreover, we can similarly further filter for vertices connecting to the current vertex of weight $w$ as well as the next one of weight $w^{1 + \beta}$. Additionally, if the current weight $w$ is at least $(\log{n})^{C}$ for a sufficiently large constant $C=C(\alpha, \tau, \beta)$, we can even require that the connecting vertex pulls and pushes the rumour in two successive rounds. The next lemma implements this weight-increasing step. 

\maybeextra{Recall the definition of $Q_{e,u}$ for an edge $e = \{u,v\}$ from page~\pageref{def:Q_e}. In particular, recall that if $Q_{e,u} \le 1/\deg(u)$ then the rumour spreads along the edge $e$ in the first round after one of the endpoints $u,v$ learns it.}

\begin{lemma}\label{lem:weight_increasing_step_weak}
    Let $\eps>0$ be constant, assume that $\alpha < \frac{1}{\tau - 2}$ and let $u_w$ be a vertex with $W_{u_{w}} = w$. Fix $\beta > 0$ so that $1 + \alpha (1 + \beta) (2 - \tau) = \delta > 0$ and let $w' = w^{1 + \beta}$. Assume that $\mathcal{V}_{\le \myO{1}} \cap \ball{u_w}{n^{-1} w}{}$ and $\mathcal{V}_{\ge w'} \cap \mathcal{V}_{\le 2w'}$ are so far unexposed. Also, assume that at most $\myO{w}$ vertices in $\big(\mathcal{V}_{\le \myO{1}} \cap \ball{u_w}{\infty}{L}\big) \setminus \ball{u_w}{n^{-1} w^{\tau - 1}}{L}$ have been exposed. Similarly assume that at most $\myO{\log\log{w}}$ vertices have been exposed in $\big(\mathcal{V}_{\le w} \cap \ball{u_w}{\infty}{L}\big) \setminus \ball{u_w}{n^{-1} w^{\tau - 1}}{L}$, and that these are all the vertices that have been so far exposed. If $w \le (\log{n})^{6/\delta}$, let $p$ be the probability that there exists some vertex $u_{w'}$ with $x_{u_{w'}} >  x_{u_w} + n^{-1}w^{(\tau-1)(1+\beta)+\eps}$ and weight $\myTheta{w'}$ such that $u_w$ connects to $u_{w'}$ through a so-far constant degree vertex. If $w > (\log{n})^{6/\delta}$, let $p$ be the probability that the previous event happens and also the intermediary vertex $u_c$ pulls and pushes the rumour in the two rounds following the one in which $u_w$ learns the rumour.
    Then, we have
    \[p \ge 1 - \myO{1/w}.\]
\end{lemma}


\begin{proof}
    We first expose $\mathcal{V}_{\le \myO{1}} \cap \ball{u_w}{n^{-1} w}{}$ and use \Cref{lem:secondmomentlemma} to deduce that $\myOmega{w}$ vertices have constant degree in their induced subgraph. Let $u_c$ be such a vertex. The probability that $u_c$ does not connect to any \emph{exposed} vertex in $\big(\mathcal{V}_{\le \myO{1}} \cap \ball{u_w}{\infty}{L}\big) \setminus \ball{u_w}{n^{-1} w^{\tau - 1}}{L}$ is at least 
    \[\left(1 - \myO{\frac{1}{w^{\tau - 1}}} \right)^{\myO{w}}.\]
    The reason for the above is that each such vertex is at a distance of at least $\myOmega{n^{-1} w^{\tau - 1}}$ from $u_c$, and by assumption there are at most $\myO{w}$ many of them. Hence, for large enough $w$, this probability is bounded below by a constant. Notice that these events are independent across different vertices $u_c$. In a similar way we can show that with constant probability a vertex $u_c$ also does not connect to any of the previously exposed high-weight vertices to the left (the set $\big(\mathcal{V}_{\le w} \cap \ball{u_w}{\infty}{L}\big) \setminus \ball{u_w}{n^{-1} w^{\tau - 1}}{L}$) with constant probability (there are $\myO{\log\log{w}}$ many such vertices and each has weight at most $w$ and distance at least $n^{-1} w^{\tau - 1}$). Any candidate $u_c$ is at a distance of at most $n^{-1} w$ from $u_w$ and $u_w$ has weight $w$, so the probability that $u_c$ connects to $u_w$ is again constant. All the previous events depend on disjoint sets of edges. Therefore, the number of vertices satisfying the criteria is binomially distributed with expectation $\myOmega{w}$. By a Chernoff bound, the number of vertices satisfying the criteria so far is $\myOmega{w}$ with a strong enough probability for the purposes of the lemma. 

    Now, the final step is to find the next vertex $u_{w'}$ and one vertex $u_c$ that also connects to it (and transmits the rumour if $w > (\log{n})^{6/\delta}$). We expose the vertices in $\big(\mathcal{V}_{\ge w'} \cap \mathcal{V}_{\le 2w'} \cap \ball{u_w}{2n^{-1} w^{(\tau - 1)(1 + \beta) + \eps}}{R}\big) \setminus \ball{u_w}{n^{-1} w^{(\tau - 1)(1 + \beta) + \eps}}{R}$ until we find a single vertex. Again, this follows a Poisson distribution with expectation $w^{\myOmega{\eps}}$, so we succeed with good enough probability. Once we find such a vertex, we see that the number of vertices $u_c$ it connects to is a binomial random variable with expectation $\myOmega{w \left(\frac{w^{1 + \beta}}{w^{(\tau - 1)(1 + \beta) + \eps}}\right)^{\alpha}}$. Now, by the assumption on $\beta$, the exponent of this expression is equal to $\delta - \alpha \eps$. We choose $\eps < \delta/(2\alpha)$ and conclude that with the required probability we have at least $\myOmega{w^{\delta / 2}}$ many such vertices $u_c$. If $w \le (\log{n})^{6/\delta}$, we are done. If $w > (\log{n})^{6/\delta}$, then we have found at least $\myOmega{(\log{n})^{3}}$ such vertices $u_c$. Note that \whp{} we can assume that all such vertices have degree at most $\myO{\log{n}}$, since their weight is constant. For each such vertex, the probability that it pulls and pushes the rumour as required is thus $\myOmega{\frac{1}{(\log{n})^2}}$. A Chernoff bound finishes the proof.
\end{proof}

    Notice also that we have only exposed $\mathcal{V}_{\le \myO{1}} \cap \ball{u_w}{n^{-1} w}{}$ (i.e.\ constant-weight vertices close to $u_w$) and vertices of weight in the interval $[w',2w']$. This implies that the constant-weight vertices in the ball of influence of $u_{w'}$ remain unexposed at the end of this procedure, since $x_{u_{w'}} - x_{u_w} \ge n^{-1} w^{(\tau - 1)(1 + \beta) + \eps} = \myomega{n^{-1}w'}$. Additionally, $\mathcal{V}_{\ge (w')^{1 + \beta}} \cap \mathcal{V}_{\le 2(w')^{1 + \beta}}$ has not been exposed at all by the previous procedure. These facts mean that is possible to subsequently use~\Cref{lem:weight_increasing_step_weak} again with $u_w \coloneqq u_{w'}$.

We can use \Cref{lem:weight_increasing_step_weak} to iteratively connect to higher and higher weights using intermediate vertices which at the point of exposure only get constantly many edges attached to them. Of course, we must also guarantee that the vertices used to first reach a polylogarithmic weight with high enough exponent are not assigned too many edges from subsequent exposures. Importantly, the number of \emph{new} edges a vertex receives obeys a Chernoff-like bound. Thus, since we want to control $\myO{\log{\log{n}}}$ many degrees, we can easily afford to union bound over each degree blowing up by a factor $\myO{\left(\log{\log{n}}\right)^{\eps}}$, for any $\eps > 0$. We put this together in the following lemma.

\begin{lemma}\label{lem:degrees_stay_small_ultrafast}
    Assume that after some partial exposure of vertices and all their in-between edges revealing a subgraph $G'$, a path of length $\myO{\log{\log{n}}}$ exists in $G'$ between two vertices $u$ and $v$ such that each edge is incident to a vertex of weight and degree at most $\left(\log{\log{n}}\right)^{\eps}$ in $G'$, for some $\eps > 0$. Then, \whp{} the degrees of these vertices increase by a factor at most $\myO{\left(\log{\log{n}}\right)^{\eps}}$ when exposing the rest of the graph.
\end{lemma}
\begin{proof}
    Notice that when exposing the rest of the graph, each vertex gains a number of incident edges which obeys a Chernoff-like bound with expectation equal to its weight. Therefore, since there are $\myO{\log{\log{n}}}$ many vertices whose degrees we want to bound, a union bound easily shows that the number of new neighbours is at most $\myO{\left(\log{\log{n}}\right)^{\eps}}$ times the corresponding weight for each such vertex.
\end{proof}

The above arguments can be collected to prove~\Cref{thm:ultrafastGIRG_weak}. The proofs of~\Cref{thm:ultrafastGIRG_strong,thm:ultrafastGIRGnew} (where we assume $\tau < \sqrt{2} + 1$ and $\sqrt{2} + 1 \le \tau < \frac{5}{2}$) are sketched afterwards.

\begin{proof}[Proof of~\Cref{thm:ultrafastGIRG_weak}]
Fix $\beta > 0$ so that $1 + \alpha (1 + \beta) (2 - \tau) = \delta > 0$. Using~\Cref{claim:go_to_wzero}, we expose vertices in a ball of radius $r(\wzero)$ around $u$ to find a short path to a vertex of weight $\wzero =\myomega{1}$. Then, by~\Cref{lem:escape_geometric_region} we expose non-constant-weight vertices to build a path from $u_{\wzero}$ to some vertex $u_{out}$ which is far enough to the right of $u$ so that constant-weight vertices in the vicinity are still unexposed. Then, we repeatedly use~\Cref{lem:weight_increasing_step_weak} to build a path from $u_{out}$ to a vertex $u_{max}$ of weight $\myOmega{n^{\frac{1}{\tau-1}-\eps}}$ (for some small constant $\eps$), where each edge is incident to a constant-weight vertex. In particular for the first use of~\Cref{lem:weight_increasing_step_weak}, we set $u_{w} := u_{out}$ and then $w = \weightof{u_{out}}$. By letting the initial weight-increasing path from~\Cref{lem:escape_geometric_region} be long enough (with respect to $\wzero$), we see that the requirements of~\Cref{lem:weight_increasing_step_weak} are satisfied for these choices of $u_{w}$ and $w$. For subsequent uses of~\Cref{lem:weight_increasing_step_weak}, it is explained after its proof that the exposure procedure therein maintains the requirements for the next step valid. The weights are raised to the power $1 + \beta$ each time. Once we reach a vertex of weight $\ge (\log{n})^{6/\delta}$, \Cref{lem:weight_increasing_step_weak} also guarantees that the rumour takes two rounds to be transmitted from the current high-weight vertex to the next higher-weight vertex. We find an analogous path from $v$ to a corresponding $v_{max}$ of weight $\myOmega{n^{\frac{1}{\tau-1}-\eps}}$. By the same calculations as in~\Cref{lem:weight_increasing_step_weak}, we find a common neighbour of $u_{max}$ and $v_{max}$ with constant weight that transmits the rumour between them in two rounds.  

The failure probabilities for the paths of length $\myO{\log{\log{n}}}$ between $u$ and $u_{max}$, $v$ and $v_{max}$, respectively, are dominated by those of the first steps, since the weights increase doubly exponentially, and the failure probability for a step starting with weight $w$ is $\myO{1/w}$. To reiterate for clarity, we have built a path $P$ from $u$ to $u_{max}$ that can be decomposed in four (consecutive) subpaths $P_{\wzero}, P_{out}, P_{\le (\log{n})^{6/\delta}}, P_{> (\log{n})^{6/\delta}}$. By choosing a $\wzero = \myomega{1}$ that grows slowly enough, we can use~\Cref{lem:degrees_stay_small_ultrafast} to deduce that the degrees of every other vertex in subpaths $P_{\wzero}, P_{out}, P_{\le (\log{n})^{6/\delta}}$ remain at most $(\log{\log{n}})^{\myo{1}}$. Moreover, since a weight-increasing path of length $\myO{\log\log w}$ is enough to reach a vertex of weight (at least) $w$, we can guarantee that the sum of the lengths of these paths is $\myO{\log^{*3}{n}}$. Indeed, $P_{\wzero}$ can be made to grow arbitrarily slowly and afterwards the weights along the path increase doubly exponentially until a weight larger than $(\log{n})^{6/\delta}$ is encountered. This means that up until the start of $P_{> (\log{n})^{6/\delta}}$, we have paths of length $\myO{\log^{*3}{n}}$, the edges of which are incident to (at least) one vertex of degree $(\log{\log{n}})^{\myo{1}}$. 

Hence, the rumour is transmitted through these subpaths \whp{} in time $\myo{\log{\log{n}}}$. Since we then need exactly two rounds to move from a vertex of weight $w$ to one of weight $w^{1 + \beta}$ for the entirety of $P_{> (\log{n})^{6/\delta}}$, the rounds needed in total are $\tfrac{4 + \myo{1}}{|\log{(1+\beta)}|} \log{\log{n}}$. Since we may choose $1 + \beta$ arbitrarily close to $\tfrac{1}{\alpha(\tau - 2)}$, the total number of rounds is also $\tfrac{4 + \myo{1}}{|\log{(\alpha(\tau-2))}|} \log{\log{n}}$, since for any $\eps' >0 $ whp it remains below $\tfrac{4+\eps'}{|\log{(\alpha(\tau-2))}|}\log \log n$ for sufficiently large $n$. The factor $4$ comes from the fact that two such paths needs to be followed (from $u$ to $u_{max}$ and then from $v_{max}$ to $v$), and for each of these paths the weight increases from $w$ to $w^{1+\beta}$ (this is true for the path directed from $v$ to $v_{max}$) every \emph{two} steps due to the intermediary low-degree vertices. Note that we also use the symmetry of the push-pull protocol to guarantee that the rumour is transmitted from $u$ to $v$, i.e.\ the path from $v$ to $v_{max}$ is taken backwards by the rumour.
\end{proof}

\paragraph{\textbf{The cases $d \ge 2$ and/or $\tau < \sqrt{2} + 1$}.} The main difference when one goes into higher dimensions is to argue that we always move further and further away from the original vertex $u$ along some axis (say along the line segment between $u$ and the destination vertex $v$). In $d = 1$, we achieved this by always going ``to the right''. This restriction halves the size of an $r$-ball in dimension $1$, and in general reduces it by some constant factor ($2^{-\myTheta{d}}$) in higher dimensions, which does not harm the proof.

On the other hand, our construction is slightly different when we assume $\tau < \sqrt{2} + 1$ (instead of $\alpha < \frac{1}{\tau - 2}$). The main difference is that in order to implement each weight-increasing step from $u_w$ to $u_{w'}$, we first go to a constant-weight vertex close to $u_w$, then to some medium-weight vertex $u_{mid}$ still close to $u_w$ and from that one to $u_{w'}$, which is far away geometrically. The factor $6$ in \Cref{thm:ultrafastGIRG_strong} (instead of $4$ in \Cref{thm:ultrafastGIRG_weak}) comes from this 3-step construction. The obvious issue is that $u_{mid}$ has large degree. However, we find many such candidates $u_{mid}$ so that at least one of them communicates with $u_{w'}$ in a single round. To illustrate this construction, we give here the proof that such a 3-step communication is possible when $\tau < \sqrt{2} + 1$. The precise probability bound is strong enough that we can use this step to build the whole path, similarly as for the case $\alpha<\frac{1}{\tau-2}$, but we suppress most details here for the sake of exposition. Note that for $2<\tau<\sqrt{2}+1$ we have $\frac{1}{\tau(\tau - 2)} > 1$.

\begin{lemma}
    Assume that $\tau < \sqrt{2} + 1$ and let $u_w$ be a vertex of weight $w = \myomega{1}$. Let $p$ be the probability that the rumour moves in two rounds from $u_w$ to some vertex $u_{w'}$ of weight $w' = \myTheta{w^{1 + \beta}}$ for some $\beta>0$ satisfying $1 + \beta < \frac{1}{\tau(\tau - 2)}$. Then,
    \[p \ge 1 - 1/w^{\myOmega{1}}.\]
\end{lemma}
\begin{proof}
    Let $\eps>0$ be a sufficiently small constant. First of all, it is relatively easy to see that one such vertex $u_{w'}$ exists with the required probability at a distance of $w^{1/d ((\tau - 1)(1 + \beta) + \eps)}$ from $u_{w}$. Now, we set $w_{mid} = w^{\rho}$ for some $(1 + \beta)(\tau - 2) + 2\eps < \rho < \frac{1}{\tau}$. 
    Then we expect to find $\myTheta{w^{1 - \rho(\tau - 1)}}$ many vertices of weight $w_{mid}$ in the ball of influence of $u_w$. Also, this number of vertices is concentrated around its expectation \whp{} since $\rho < \frac{1}{\tau} < \frac{1}{\tau - 1}$. A constant fraction of these vertices learn the rumour in two rounds, since all of them have a large number of constant-weight vertices as common neighbours with $u_w$. Now, because of our lower bound on $\rho$ and the triangle inequality, a constant fraction of the mid-weight vertices are connected to $u_{w'}$. Assuming the degrees are close to their respective weights, each of these vertices has a $\myTheta{\frac{1}{w^{\rho}}}$ chance to send the rumour across to $u_{w'}$ in the third round. Now, the expected number of such communications is $\myTheta{w^{1 - \rho \tau}} = w^{\myOmega{1}}$ since $\rho < \frac{1}{\tau}$, which finishes the proof.
\end{proof} 
Note that the steps we have described actually succeed with even stronger probability ($1 - e^{-w^{\myOmega{1}}}$). One detail that has been dealt with in the proof for the case $\alpha < \frac{1}{\tau - 2}$ has to do with the degrees of the constant-weight vertices, which we must also control. When exposing constant-weight vertices, we used~\Cref{lem:secondmomentlemma} to guarantee that degrees stay constant, and this introduced the bottleneck in the failure probabilities. In a full proof, one would have an analogous lemma for the current case as well, and this is why we have the weaker bound $p \ge 1 - 1/w^{\myOmega{1}}$. Notice that this bound is strong enough so that we can use the lemma repetitively with doubly-exponentially increasing weights such that the failure probability of the whole path-building process is dominated by that of the first step.

Taking $1 + \beta$ in the above lemma arbitrarily close to $\frac{1}{\tau(\tau - 2)}$ one can then prove \Cref{thm:ultrafastGIRG_strong} in a similar way as \Cref{thm:ultrafastGIRG_weak} above.

\paragraph{\textbf{The case $\tau < \frac{5}{2}$}.} Here we show that it is possible to implement a weight-increasing step even with the weaker assumption $\tau < \frac{5}{2}$. 
Before the formal proof, we will first give an intuitive sketch of how the construction for $\tau < \sqrt{2}+1$ can be improved. Suppose the rumour has reached a vertex $u_{curr}$ of weight $w$ and now we want to reach a vertex $u_{next}$ of weight $\myTheta{w^{1+\eps}}$ within constantly many rounds. Previously, we did so by reaching a constant fraction of the set $U_{mid}$ of vertices with weight $w_{mid} = \myTheta{w^{\tau - 2}}$, all of which were connected to $u_{next}$ with constant probability. We looked for the set $U_{mid}$ inside the ball of influence of $u_{curr}$, such that for every $u_{mid} \in U_{mid}$ there exist many vertices $u_{const}$ of constant weight connected to both $u_{mid}$ and $u_{curr}$ and thus $u_{mid}$ is informed of the rumour in two rounds. This is illustrated in~\Cref{fig:ultrafast_strong}.

It just so happens that with this strategy, enough vertices in $U_{mid}$ are informed quickly so that $u_{next}$ is then informed in the next round (i.e.\ at least $\myOmega{w^{\tau -2}}$ vertices in $U_{mid}$ are informed of the rumour within two rounds) if $\tau < \sqrt{2} + 1$. This is because we have $\expectationof{|U_{mid}|} = \myTheta{w^{1 - (\tau - 1)(\tau - 2)}}$, and a constant fraction of these vertices is informed \whp{}.\footnote{Leading to the inequality $1 - (\tau - 1)(\tau - 2) > \tau - 2$, which is true for $\tau < \sqrt{2} + 1$.} The goal now is to quickly reach a larger such set of mid-weight vertices. Instead of focusing on $BoI(u_{curr})$, we can look for them at a distance of $n^{-\frac{1}{d}}w^{\frac{1}{d}(1+\frac{\tau - 2}{\tau})}$. Notice that at this distance, any vertex of weight $w_{low}=\Theta(w^{\frac{\tau - 2}{\tau}})$ is connected to $u_{curr}$ with constant probability. Let us use $u_{low}$ to refer to such a vertex, which now assumes the role that $u_{const}$ had in the previous construction. It turns out that any vertex of weight $w_{mid}$ in this distance ($n^{-\frac{1}{d}}w^{\frac{1}{d}(1+\frac{\tau - 2}{\tau})}$) is quickly informed by some $u_{low}$ in its "extended $BoI$" (i.e.\ the ball in which edges between vertices of weights $w_{low}$ and $w_{mid}$ are strong, which is $n^{-\frac{1}{d}}w^{\frac{1}{d}(\tau - 2 + \frac{\tau - 2}{\tau})}$) with constant probability. This is because there are $w^{\tau - 2 + \frac{\tau - 2}{\tau}} w_{low}^{-(\tau-1)}w_{low}^{-2} = \myTheta{1}$ such vertices $u_{low}$ that inform it quickly in expectation, and each one has an independent chance of doing so. Now, we have $\expectationof{|U_{mid}|} = \myTheta{w^{1 + \frac{\tau - 2}{\tau}- (\tau - 1)(\tau - 2)}}$. This is at least $\myOmega{w^{\tau - 2}}$ if $\tau < 2.48$, a clear improvement over the $\tau < \sqrt{2} + 1 \approx 2.414$ upper bound from \Cref{thm:ultrafastGIRG_strong}.

One can continue improving upon this strategy by using more intermediary vertices to connect $u_{curr}$ to $u_{mid}$. This leads to a statement of the form ``a constant fraction of the vertices of weight $\myTheta{w^{x}}$ at a distance of $n^{-\frac{1}{d}}w^{\frac{1}{d}(1+\frac{x}{f^{(k)}(\tau)})}$ from $u_{curr}$ are informed within $k$ rounds'' and we ultimately use this for $x > \tau - 2$ and constant (but large) $k$. We use induction to prove this statement for a sequence $f^{(k)}(\tau)$ such that for any $\tau < \frac{5}{2}$, there exists some $k_{\tau}$ satisfying
\[\expectationof{|U_{mid}|} = \myTheta{w^{\left(1+\frac{\tau - 2}{f^{(k_{\tau})}(\tau)} - (\tau - 1)(\tau - 2)\right)}} \ge \myOmega{w^{\tau - 2}},\]
meaning that in the next round $u_{next}$ is informed of the rumour \whp{}. In particular, we have the following definition.
\begin{equation}
\label{eq:recursive_definition}
f^{(2)}(\tau) = \tau, \quad f^{(k+1)}(\tau) = \tau - \frac{1}{f^{(k)}(\tau)} \quad \text{for } k \geq 2.
\end{equation}

A straightforward induction shows that the sequence $f^{(k)}(\tau)$ is decreasing and bounded below by $1$, and hence must have a limit. Let $f^{(\infty)}(\tau)$ denote this limit. One can see with straightforward calculations that $f^{(\infty)}(\tau) = \frac{\tau + \sqrt{\tau^2 - 4}}{2}$ and that the exponent of $\expectationof{|U_{mid}|}$ for $k_\tau=\infty$ is $1+\frac{\tau - 2}{f^{\infty}(\tau)} - (\tau - 1)(\tau - 2) > \tau-2$, as needed.

In the following, we formalize this idea. To ease exposition we assume that the degrees of vertices are within constant factors of their weights. This assumption can be removed analogously to the other constructions as in~\Cref{sec:MCD} and in the proof of~\Cref{thm:ultrafastGIRG_weak} in~\Cref{sec:ultrafast_euclidean}.

\begin{lemma}\label{lem:2.5_construction}
Let $u$ be a vertex with weight at least $w$ and assume that the vertices of weight less than $w$ are so far unexposed. Assume that $\tau \ge \sqrt{2} + 1$. Let $k \ge 2$ be a fixed integer and $w' = w^x$ for a fixed $x < (\tau - \frac{1}{f^{(k)}(\tau)} - 1)^{-1}$. Recall the function $f^{(k)}(\tau)$ from equation~\eqref{eq:recursive_definition}. There exists a function $c(k)$ such that with probability at least $1 - e^{-w^{\myOmega{1}}}$ the rumour spreads from $u$ to at least $c(k)w^{1 + \frac{x}{f^{(k)}(\tau)} -x(\tau - 1)}$ vertices with weight in the range $[w', 2w']$ and distance from $u$ at most $n^{-\frac{1}{d}}w^{\frac{1}{d}(1 + \frac{x}{f^{(k)}(\tau)})}$ within $k$ rounds.
\end{lemma}
\begin{proof}
    First note that we always have $x < (\tau - \frac{1}{f^{(k)}(\tau)} - 1)^{-1} \le (\tau - \frac{1}{f^{(2)}(\tau)} - 1)^{-1} \le 1$ for $\tau \ge \sqrt{2} + 1$.
    We proceed by induction. Let us first show the base case for $k=2$. The following arguments are visualized in~\Cref{fig:2.5_construction_base_case}. Consider the box $B$ of volume $n^{-1}w^{1 + \frac{x}{f^{(k)}(\tau)}}$ around $u$. For large enough $w$, we can find inside $B$ a collection $B_1, B_2, B_3, \dots B_b$ of boxes with volume $n^{-1}w^{x(\tau - 1)}$ with the following properties.
    \begin{enumerate}
        \item For any vertex $v$ in some $B_i$, $\geomdist{u}{v} \le \frac{1}{2}\left(\frac{w^{1 + \frac{x}{f^{(k)}(\tau)}}}{n}\right)^{\frac{1}{d}}$.
        \item $b \ge 2^{-d}\frac{1}{100}w^{1 + \frac{x}{f^{(k)}(\tau)} - x(\tau - 1)}$ (note that the exponent of $w$ in this lower bound is strictly positive since $x<(\tau - \frac{1}{f^{(k)}(\tau)} - 1)^{-1}$).
        \item For any vertices $v_1, v_2$ with $v_1 \in B_i$, $v_2 \in B_j$ and $i \neq j$, we have $\geomdist{v_1}{v_2} \ge 3(\frac{w^{x(\tau - 1)}}{n})^{\frac{1}{d}}$. 
    \end{enumerate}
    
    Notice that the above imply that the boxes are not only disjoint but also have enough space in between them to accommodate later arguments. Also, they are far away from the boundary of $B$. Now, for each such box $B_i$ we have a constant probability of finding a vertex $u'_{i}$ inside it with weight in the range $[w', 2w']$, because of the properties of the PPP and the weight distribution for the vertices. Assuming we find such a vertex $u_i'$, let $B'_{i}$ be the box centred at $u'_{i}$ with volume $n^{-1}w^{x + \frac{x}{f^{(k)}(\tau)}} = n^{-1}w^{x + \frac{x}{\tau}}$ (since $k = 2$). Notice that $B'_i$ is fully contained in $B$ since $x<1$. With at least constant probability we find in $B'_i$ at least $\myOmega{w^{x + \frac{x}{\tau} -\frac{x}{\tau}(\tau - 1)}} = \myOmega{w^{\frac{2x}{\tau}}}$ vertices with weight in the range $[w^{\frac{x}{\tau}}, 2w^{\frac{x}{\tau}}]$. Now, such vertices are connected with constant probability to both $u$ and $u'_i$ (because they are in $B_i'$ and in $B$). Their degrees are assumed to be $\myO{w^{\frac{x}{\tau}}}$ as well, and so with constant probability at least one of them will transmit the rumour from $u$ to $u'_i$ in two rounds. If all the previous steps succeed, we have found and informed a vertex in the desired weight range inside $B_i$. All the steps are independent, so in the end we have a constant probability of success for each $B_i$. The ample spacing given also guarantees (along with $\tau \ge \sqrt{2} + 1$) that each $B_i$ succeeds independently of the other boxes.\footnote{We want that $B'_i$ has smaller or equal volume to that of $B_i$, so that point (3) in the properties of the $B_i$ boxes implies that the boxes $B_i'$ are disjoint. This is true if $x + \frac{x}{\tau} \le x(\tau-1)$, which is equivalent to $\tau \ge \sqrt{2} + 1$.} Thus, we have independent Bernoulli trials, one for each $B_i$ and also the expected number of successes is $\myOmega{b} = \myOmega{w^{1 + \frac{x}{f^{(k)}(\tau)} - x(\tau - 1)}}$. We have already shown that the exponent of $w$ is positive here, and hence the claim follows by a typical Chernoff-type argument.
    For the induction step (assuming the claim is true for $k-1$), we repeat roughly the same argumentation. In particular, let $B$ and $B_i$ be defined in exactly the same way (but notice that $k$ has changed, in particular $B$ is larger now). However, we now define $B_i'$ to have volume $n^{-1}w^{x + \frac{x}{f^{(k)}(\tau) f^{(k-1)}(\tau)}}$. By rescaling (setting $w_{induction}:=w'$ and $x_{induction}:= (f^{(k)}(\tau))^{-1}$) 
    and the induction hypothesis,\footnote{We have to make sure that the constraint $x_{induction} < (\tau - \frac{1}{f^{(k-1)}(\tau)} - 1)^{-1}$ is satisfied in order to use the induction hypothesis, which holds since by equation~\eqref{eq:recursive_definition} we have $(f^{(k)}(\tau))^{-1} < (\tau - \frac{1}{f^{(k-1)}(\tau)} - 1)^{-1}$.} we can assume that for each $B_i$ we find with constant probability a vertex $u'$ with weight in the range $[w', 2w']$ that ``informs'' a set $U$ of $c(k - 1)w^{x + \frac{x}{f^{(k)}(\tau) f^{(k-1)}(\tau)} - (\tau - 1)\frac{x}{f^{(k)}(\tau)}}$ vertices with weight in $\left[w^\frac{x}{f^{(k)}(\tau)}, 2w^\frac{x}{f^{(k)}(\tau)
    }\right]$ in the corresponding box $B_i'$ within $k - 1$ rounds. That would normally mean that the vertex $u'$ would --given the rumour-- spread it to the mentioned vertices within $k-1$ rounds. However, since the protocol is symmetric, we can also assume that these vertices are such that if any of them learns the rumour, it transmits it to $u'$ within $k-1$ rounds. Now the goal is to show that at least one of them learns the rumour from $u$ in a single round, with constant probability. This is true because each one has a chance of $\myOmega{w^{-\frac{x}{f^{(k)}(\tau)}}}$ to pull the rumour and also
    \begin{align*}   
    &x + \frac{x}{f^{(k)}(\tau) f^{(k-1)}(\tau)} - (\tau - 1)\frac{x}{f^{(k)}(\tau)} -\frac{x}{f^{(k)}(\tau)} \\
    &\qquad= x\left(1 + \frac{1}{f^{(k)}(\tau) f^{(k-1)}(\tau)} - (\tau - 1)\frac{1}{f^{(k)}(\tau)} -\frac{1}{f^{(k)}(\tau)}\right) = 0.
    \end{align*}
    The final equality follows since $\frac{1}{f^{(k-1)}(\tau)} = \tau - f^{(k)}(\tau)$ by definition. 

    Now, by a Chernoff-type argument one can choose a $c(k)$ (small enough with respect to $c(k - 1)$ and the constant probabilities of success of the various steps) so that the claim is satisfied.
\end{proof}

\begin{figure}[ht]
\centering
\begin{tikzpicture}[scale=1.4]

\def\BBsize{6}     
\def\Bside{2.0}    
\def\halfBside{1.0}
\def\offset{1.4}   

\def\BsidePrime{1}     
\def\halfBsidePrime{0.5}

\draw[thick] (0,0) rectangle (\BBsize,\BBsize);
\node[above left] at (0,\BBsize) {$B$};

\coordinate (u) at (3,3);
\fill (u) circle (3pt);
\node[above right] at (u) {$u$};

\coordinate (B1c) at ($(u)+(-\offset,\offset)$);   
\coordinate (B2c) at ($(u)+(\offset,\offset)$);    
\coordinate (B3c) at ($(u)+(-\offset,-\offset)$);  
\coordinate (B4c) at ($(u)+(\offset,-\offset)$);   

\draw[thick] ($(B1c)-(\halfBside,\halfBside)$) rectangle ($(B1c)+(\halfBside,\halfBside)$);
\node at (B1c) {$B_1$};

\draw[thick] ($(B2c)-(\halfBside,\halfBside)$) rectangle ($(B2c)+(\halfBside,\halfBside)$);
\node at (B2c) {$B_2$};

\draw[thick] ($(B3c)-(\halfBside,\halfBside)$) rectangle ($(B3c)+(\halfBside,\halfBside)$);
\node at (B3c) {$B_3$};

\draw[thick] ($(B4c)-(\halfBside,\halfBside)$) rectangle ($(B4c)+(\halfBside,\halfBside)$);
\node at (B4c) {$B_4$};

\coordinate (B1v) at ($(B1c)-(0.7*\halfBside,0.7*\halfBside)$);
\coordinate (B1l) at ($(B1v) + (-0.6 * \BsidePrime, -0.8 * \BsidePrime)$);
\fill (B1v) circle (2pt);

\coordinate (B3v) at ($(B3c)-(0.7*\halfBside,0.7*\halfBside)$);
\coordinate (B3l) at ($(B3v) + (-0.6 * \BsidePrime, 0.8 * \BsidePrime)$);
\fill (B3v) circle (2pt);

\coordinate (B4v) at ($(B4c)+(0.7*\halfBside,-0.7*\halfBside)$);
\coordinate (B4l) at ($(B4v) + (-0.8 * \BsidePrime, -0.6 * \BsidePrime)$);
\fill (B4v) circle (2pt);

\draw[thick,dashed] ($(B1v)-(\halfBsidePrime,\halfBsidePrime)$) rectangle ($(B1v)+(\halfBsidePrime,\halfBsidePrime)$);
\node at (B1l) {$B'_1$};

\draw[thick,dashed] ($(B3v)-(\halfBsidePrime,\halfBsidePrime)$) rectangle ($(B3v)+(\halfBsidePrime,\halfBsidePrime)$);
\node at (B3l) {$B'_3$};

\draw[thick,dashed] ($(B4v)-(\halfBsidePrime,\halfBsidePrime)$) rectangle ($(B4v)+(\halfBsidePrime,\halfBsidePrime)$);
\node at (B4l) {$B'_4$};


\coordinate (B1vA) at ($(B1v)+(0.2,0.2)$);
\coordinate (B1vB) at ($(B1v)+(-0.2,-0.1)$);
\coordinate (B1vC) at ($(B1v)+(0.2,-0.05)$);
\fill (B1vA) circle (1pt);
\fill (B1vB) circle (1pt);
\fill (B1vC) circle (1pt);

\draw (B1vA) -- (B1v);
\draw (B1vB) -- (B1v);
\draw (B1vC) -- (B1v);
\draw (B1vA) -- (u);
\draw (B1vC) -- (u);

\coordinate (B3vA) at ($(B3v)+(0.2,-0.2)$);
\coordinate (B3vB) at ($(B3v)+(-0.2,0.2)$);
\fill (B3vA) circle (1pt);
\fill (B3vB) circle (1pt);

\draw (B3vB) -- (B3v);
\draw (B3vA) -- (u);
\draw (B3vB) -- (u);

\coordinate (B4vA) at ($(B4v)+(0.1,0.2)$);
\coordinate (B4vB) at ($(B4v)+(-0.35,-0.1)$);
\fill (B4vA) circle (1pt);
\fill (B4vB) circle (1pt);

\draw (B4vA) -- (B4v);
\draw (B4vB) -- (B4v);
\draw (B4vB) -- (u);

\end{tikzpicture}
\caption{The base case in the inductive proof of~\Cref{lem:2.5_construction}. In this particular case, $B_2$ failed to produce a vertex of weight $w^x$. The rest did, and so we expose boxes $B_i'$ around these vertices to look for vertices of weight $w^{\frac{x}{\tau}}$. Some of these vertices connect to $u$ and the vertices found in $B_i$. These are responsible for transmitting the rumour. 
}
\label{fig:2.5_construction_base_case}
\end{figure}

With~\Cref{lem:2.5_construction} at hand, we can build weight-increasing paths (for rumour spreading).

\begin{thm}\label{thm:tweak_beta_k}
Let $\beta > 0$ and $k \ge 2$. Assume that $(1 + \beta)(\tau - 2) < (\tau - \frac{1}{f^{(k)}(\tau)})^{-1}$. Let $\mathcal{G}=(\mathcal{V},\mathcal{E})$ be a GIRG and $u, v$ be two vertices in the giant component, chosen uniformly at random. Denote by $\mathcal{T}$ the event that the rumour is transmitted from $u$ to $v$ via the push-pull protocol within the first $(2(k + 1) + o(1)) \cdot \frac{\log \log n}{|\log{(1 + \beta)}|}$ rounds. Then,
\[
\prob{\mathcal{T}} \ge 1 - o(1).
\]
\end{thm}
\begin{proof}
The proof is similar to that of~\Cref{thm:ultrafastGIRG_weak}. The main difference is that now each weight-increasing step uses $k$ intermediary vertices (and thus $k+1$ edges), so we will only explain how a single such step is carried out. Assuming the weight of the current vertex $u_{curr}$ is $w$, we again look for a vertex $u_{next}$ of weight at least $w^{1+\beta}$, which can \whp{} be found at a distance $n^{-\frac{1}{d}}w^{\frac{1}{d}((1+\beta)(\tau - 1) + \eps)}$. Let $\rho = (1 + \beta)(\tau - 2) + 2\eps$. By~\Cref{lem:2.5_construction} 
we know\footnote{Using $x = \rho$. The assumption in~\Cref{thm:tweak_beta_k} guarantees that the assumption on $x$ for~\Cref{lem:2.5_construction} holds for $\eps$ small enough since $(\tau - \frac{1}{f^{(k)}(\tau)})^{-1} < (\tau - \frac{1}{f^{(k)}(\tau)}-1)^{-1}$.} that within $k$ rounds, $u_{curr}$ informs $\myOmega{w^{1 + \rho(1/f^{(k)}(\tau) -\tau + 1)}}$ vertices with weight in $[w^{\rho}, 2w^{\rho}]$ at distances at most $n^{-\frac{1}{d}}w^{\frac{1}{d}(1 + \rho/f^{(k)}(\tau))}$. All such vertices have a constant probability of connecting to $u_{next}$ (because the distance separating them from $u_{next}$ is $\myO{n^{-1/d}w^{1/d((1+\beta)(\tau - 1) + \eps)}}$ while the product of their weights is $\myTheta{w^{(1+\beta) + \rho}} = \myTheta{w^{(1+\beta)(\tau - 1) + 2\eps}}$) and also a $\myOmega{w^{-\rho}}$ chance to push the rumour to it in a single round. Conditioned on the weights, these events are independent across the vertices. Therefore, it suffices to show that the exponent of the expected number of single-round communications $1 + \rho(\frac{1}{f^{k}(\tau)} -\tau)$ is strictly positive, which is the case for small enough $\eps$ since $(1 + \beta)(\tau - 2) < (\tau - \frac{1}{f^{(k)}(\tau)})^{-1}$.

Another technical step that is not present in the proof of~\Cref{thm:ultrafastGIRG_weak} is that we have to make sure we have not exposed specific weight ranges each time we use~\Cref{lem:2.5_construction}. This was easier to deal with in~\Cref{thm:ultrafastGIRG_weak} because the weight ranges we exposed for each step were disjoint (and for the constant-weight vertices we used the geometry for separation). Now this is no longer true. It is still true however that the weight ranges we expose for a single step increase doubly exponentially. Let the weight range that we expose in step $i$ of the process be $[\wzero^{l_{i}}, \wzero^{h_i}]$. Because we increase the weights doubly exponentially, it is true that for any fixed $\beta, k$ there exists a constant $c$ such that $[\wzero^{l_{i}}, \wzero^{h_i}]$ and $[\wzero^{l_{i + c}}, \wzero^{h_{i + c}}]$ are disjoint. Therefore, to keep everything independent, it suffices to split the PPP into $c$ PPPs with equal intensity and expose them in a round-robin fashion. To give an example, suppose $c = 3$. This means that we have three PPPs, $P_1$, $P_2$ and $P_3$. For the first step, we use $P_1$, the second $P_2$ and the third $P_3$. Then, we use $P_1$ for the fourth, $P_2$ for the fifth etc. This guarantees that the weight ranges used e.g. in the first and fourth steps are disjoint, so it is not a problem that we use the same PPP for both of these steps. This does not affect the probability bounds of failure by more than a constant factor.
\end{proof}

While it is not immediately obvious, the previous theorem guarantees that ultra-fast transmissions happen \whp{} if $\tau < \frac{5}{2}$.
\ultrafastGIRGnew*
\begin{proof}
It suffices to choose some $\beta > 0$ and $k \in \Z$ in~\Cref{thm:tweak_beta_k} such that the assumption therein is satisfied. Let us pretend $\beta = 0$ and $k = \infty$ for now and notice that the assumption becomes $\tau - 2 < \frac{1}{\tau - \frac{1}{f^{(\infty)}(\tau)}} = \frac{1}{f^{(\infty)}(\tau)} = \left(\frac{\tau + \sqrt{\tau^2-4}}{2}\right)^{-1}$. This is satisfied as long as $\tau < \frac{5}{2}$. Hence, for any such $\tau$, we can also choose $\beta$ and $k$ as required, which would then yield $C(\tau) =  \frac{2(k + 1)}{|\log{(1 + \beta)}|}$. 
\end{proof}

%% file: polylogarithmic_upper_bound.tex
In this section we will show that the push-pull protocol can be fast - i.e.\ reach a constant fraction of the vertices in $\mathrm{polylog}(n)$ rounds - in \girgs{} if $\alpha < \frac{\tau - 1}{\tau - 2}$ or $\tau < \phi + 1$. More precisely, we will show that for two random vertices in the giant component, the rumour is transmitted from one to the other within the required time \whp. We will build hierarchies as explained in~\Cref{sec:tools}, see also \Cref{fig:polylog_weak,fig:polylog_strong} for a visual explanation of a single hierarchy step.
One way we can show that a rumour spreads fast is by constructing paths in which vertices of high weight are not found consecutively, which we call \textit{\beady{}}. If each edge has at least one endpoint with small degree, then once one of the endpoints knows the rumour, it will take in expectation few rounds for the other endpoint to be informed as well: if the vertex which knows the rumour first has small degree, then the rumour is pushed fast; if the vertex which does not know the rumour has small degree, it will pull the rumour quickly. These \beady{} paths not only transmit the rumour fast in expectation, but also with high probability. The actual probability bound depends on their length as well as the desired transmission time. We can show a general bound on the probability that an \beady{} path of length $l$ transmits the rumour from start to end within $t$ rounds. This is captured in the following lemma.

\newcommand{\beadyPathStretchFactor}{x}
\begin{lemma}\label{lem:beadyPathGoFast}
    Let $\pi = (u_1, u_2, \dots, u_{l + 1})$ be a $\beadydeg{}$-\beady{} path of length $l$. That is, we have 
    \[\: \mymin{\degreeof{u_{i}}}{\degreeof{u_{i + 1}}} \le \beadydeg{}\]
    for all $1 \le i \le l$ and some fixed $\beadydeg{}$. Let $\rumourtime{\pi}$ be a random variable denoting the time it takes for the rumour to propagate through the entire path $\pi$ after $u_1$ learns the rumour. We have
    \[\expectationof{\rumourtime{\pi}} \le \beadydeg{} l.\]
    Additionally, for any $t \ge 2 \beadydeg{} l$ we have
    \[\prob{\rumourtime{\pi} > t} \le \myexp{-\myOmega{t/k}}.\]
    
\end{lemma}
\begin{proof}
Since each edge needs at most $\beadydeg{}$ rounds to transmit the rumour in expectation, the claim about $\expectationof{\rumourtime{\pi}}$ easily follows by linearity of expectation. To prove the other claim, 
%
note that we can stochastically upper bound $\rumourtime{\pi}$ by a random variable $Z$ following a negative binomial distribution - i.e.\ a random variable counting the number of trials until we get $l$ successes of i.i.d.\ Bernoulli trials - with associated probability $\frac{1}{\beadydeg{}}$ for each trial. Here, we think of each success as the transmission of the rumour across an edge, and each edge has a probability of at least $\frac{1}{\beadydeg{}}$ to transmit the rumour in each round. To further bound $Z$, let us consider a binomial random variable $Z_t$, which counts the number of successes in \emph{exactly} $t$ trials with the same probability $\frac{1}{\beadydeg{}}$. One sees that
\[\prob{Z > t} \le \prob{Z_t \le l}.\]
We can bound the probability $\prob{Z_t \le l}$ using a standard Chernoff bound: using $\expectationof{Z_t} = \frac{t}{\beadydeg{}}$ and $t \ge 2 \beadydeg{} l$ we have
\[
\prob{Z_t \le l} = \prob{Z_t \le \frac{ \beadydeg{} l}{t}\expectationof{Z_t}}
\le \myexp{-(1 -  \beadydeg{} l/t)^{2}\expectationof{Z_t}/2} \le \myexp{-\expectationof{Z_t}/8} = \myexp{-\myOmega{t/k}}.
\]
\end{proof}

\Cref{lem:beadyPathGoFast} will be used to guarantee that the rumour spreads through the gaps/edges within their corresponding timeblocks, as defined in \Cref{def:timing}.

Recall that we will consider renormalized distances in this section, i.e.\ we essentially multiply all distances by a factor $n^{1/d}$ and modify the connection probability so that the graphs we get are the same. Essentially, we are now looking at the GIRG model defined in~\cite{komjathy2023four1}, which is equivalent to ours, but uses a rescaled geometry. This is only a change in viewpoint that makes the formulas cleaner and easier to intuit, given that the original argument was designed for such distances as well. We will generally follow the formalism of hierarchies presented in~\cite{komjathy2023four1}. To that end, we give here verbatim some definitions regarding binary strings which will serve as the basis for defining (our version of) hierarchies.

\begin{defn}[Binary strings]\label{def:strings}
For $\sigma=\sigma_1\ldots \sigma_i\in\{0,1\}^i$ for some $i\ge 1$, we define $\sigma T:=\sigma_1\ldots \sigma_i\sigma_i\in \{0,1\}^{i+1}$, while $\sigma T_0:=\sigma$, and $\sigma T_k:=(\sigma T_{k-1})T$ for any $k\ge 2$. Let $0_i:=0T_{i-1}$ and $1_i:=1T_{i-1}$ be the strings consisting of $i$ copies of $0$ and $1$, respectively. 
Fix an integer $R \ge 1$. Define the \emph{equivalence relation} $\sim_T$ on $\cup_{i=1}^R\{0,1\}^i$, where $\sigma\sim_T \sigma'$ if either $\sigma T_k=\sigma'$ or $\sigma'T_k=\sigma$ for some $k\ge 0$, with  $\{\sigma\}$ be the equivalence class of $\sigma$. Let
\begin{align*}
    \Xi_{i}:=\{ \sigma \in \cup_{j=i}^R\{0,1\}^j: \sigma_{i-1} \neq \sigma_i, \sigma_{j}=\sigma_i \ \forall j\ge i\}, \qquad \Xi_0:=\{\emptyset\},
\end{align*}
with $\emptyset$ the empty string. We say that $\{\sigma\}$ \emph{appears first on level $i$} if any (the shortest) representative of the class $\{\sigma\}$ is contained in $\Xi_i$. 

For $\sigma=\sigma_1\ldots \sigma_i\in\{0,1\}^i$ for some $i\ge 1$, we define $\sigma T^c:=\sigma_1\ldots \sigma_i(1-\sigma_i)\in \{0,1\}^{i+1}$. For $\sigma\in \Xi_i$, we say that 
$(\sigma T_{j-1})T^c\in \{0,1\}^{i+j}$ is the \emph{level-$(i+j)$ sibling} of $\{\sigma\}$. We say that two strings in level $i$ are \emph{newly appearing cousins} on level $i$ if they are of the forms $\sigma01$ and $\sigma 10$ for some $\sigma\in \{0,1\}^{i-2}$. 
\end{defn}

Let us now intuitively explain the definitions, also following~\Cref{fig:1D-hierarchy-2levels}. For this purpose, consider the task of connecting two vertices $u, v$ in dimension 1 (with $u$ to the left of $v$) through a hierarchy. To each string $\sigma$, we will assign some vertex $y_{\sigma}$. A single vertex may be used for multiple strings. The vertex $u$ will be equal to $y_{0}, y_{00}, y_{000}$ and so on. Correspondingly, $v = y_1 = y_{11} = y_{111}$ and so on. Intuitively, this means that $u$ is the ``left'' vertex in the first level of the hierarchy, the left vertex of the left gap of the second level of the hierarchy, the left vertex of the left gap of the left gap and so on. The first vertices $u', v'$ that we find ($u'$ to the left of $v'$) are then $y_{01}$ and $y_{10}$ respectively. The semantics for $u'$ is that we first go the left gap (hence the $0$) and then we look at the right vertex (hence the $1$) which this gap (between $u$ and $u'$) separates. Analogously one can see why $v'$ is given the identifier $10$. In general, to understand what $y_{\sigma}$ means with $\sigma \in \{0, 1\}^{i}$, we take the first $i - 1$ bits and treat them as follows. Start from the initial problem of connecting $u, v$. If the first bit is a $0$, go to the subproblem $u, u'$. If it is a $1$, go to the subproblem $v', v$. Note that the different order between $u, u'$ and $v', v$ here is intentional. Then, analogously move from the current subproblem to the corresponding ``left'' or ``right'' one for all $i - 1$ bits.\footnote{A more extensive example; if the first two bits are both 0, we have moved to the subproblem $u, u''$, where $u'' = y_{001}$ is the vertex closest to $u$ when finding an edge close to $u$ and $u'$. If the third (say final) bit is $0$, the result is $u$, otherwise it is $u''$.} Finally, the last bit simply indexes into the vertices defining the subproblem we arrived at using the $i-1$ first bits.
\begin{figure}[ht]
\centering
\begin{tikzpicture}[scale=1.2, font=\footnotesize]

  \draw[->] (0,0) -- (9.5,0) coordinate (Axis) node[right] {};

  \fill (0,0) circle (1.3pt) node[below,yshift=-2pt]
    {$u = y_{0}$};
  \fill (3,0) circle (1.3pt) node[below,yshift=-2pt]
    {$u' = y_{01}$};
  \fill (6,0) circle (1.3pt) node[below,yshift=-2pt]
    {$v' = y_{10}$};
  \fill (9,0) circle (1.3pt) node[below,yshift=-2pt]
    {$v = y_{1}$};

  \draw[dashed] (3,0) to[bend left=45] node[above,yshift=2pt]
    {} (6,0);

  \draw [decorate, decoration={brace, mirror, amplitude=5pt}]
    (0,-0.6) -- (3,-0.6) node[midway,below,yshift=-0.3cm]
    {Subproblem $(u,u')$};
  \draw [decorate, decoration={brace, mirror, amplitude=5pt}]
    (6,-0.6) -- (9,-0.6) node[midway,below,yshift=-0.3cm]
    {Subproblem $(v',v)$};

  \fill (1.2,0) circle (1.3pt) node[below,yshift=-2pt]
    {$y_{001}$};
  \fill (1.8,0) circle (1.3pt) node[below,yshift=-2pt]
    {$y_{010}$};

  \draw[dashed] (1.2,0) to[bend left=45] node[above,yshift=2pt]
    {} (1.8,0);

  \draw [decorate, decoration={brace, amplitude=4pt}, yshift=6pt]
    (0,0) -- (1.2,0)
    node[midway,above,yshift=0.2cm]
    {$(u,y_{001})$};

  \draw [decorate, decoration={brace, amplitude=4pt}, yshift=6pt]
    (1.8,0) -- (3,0)
    node[midway,above,yshift=0.2cm]
    {$(y_{010},u')$};

  \fill (7.2,0) circle (1.3pt) node[below,yshift=-2pt]
    {$y_{101}$};
  \fill (7.8,0) circle (1.3pt) node[below,yshift=-2pt]
    {$y_{110}$};

  \draw[dashed] (7.2,0) to[bend left=45] node[above,yshift=2pt]
    {} (7.8,0);

  \draw [decorate, decoration={brace, amplitude=4pt}, yshift=6pt]
    (6,0) -- (7.2,0)
    node[midway,above,yshift=0.2cm]
    {$(v',y_{101})$};

  \draw [decorate, decoration={brace, amplitude=4pt}, yshift=6pt]
    (7.8,0) -- (9,0)
    node[midway,above,yshift=0.2cm]
    {$(y_{110},v)$};

\end{tikzpicture}
\caption{An example of the first two levels of a hierarchy in 1 dimension. Dashed arcs represent edges and braces represent subproblems.}
\label{fig:1D-hierarchy-2levels}
\end{figure}
Using this intuition we can make a few observations. When we read two same bits consecutively, we do not actually change the vertex we are currently considering. That is, $y_{\sigma 00} = y_{\sigma 0}$ and $y_{\sigma 11} = y_{\sigma 1}$, for any $\sigma$. This means that the equivalence classes formed by the equivalence relation $\sim_T$ correspond to (distinct) vertices in the hierarchy. Now, the elements of $\Xi_i$ are essentially the vertices which were not there when considering level $i - 1$ of the hierarchy but emerged at level $i$.

Finally, let us appreciate the roles of siblings and cousins. Cousins are the endpoints of the edges we identify while building the hierarchy. That is, $u' = y_{01}, v' = y_{10}$ are cousins, where $(u', v')$ is the first edge used to shrink the distance between the vertices $u, v$ of the original problem. The term siblings is used to refer to two vertices that describe some subproblem. For example, $u, v$ are siblings because they describe the whole problem but $u, u'$ also are, as are $v', v$.

With all that in mind, we would like to build hierarchies such that all cousins are actually connected by edges and deepest siblings (defining the smallest subproblems left unsolved by the hierarchy) are very close in Euclidean distance. It should always hold that $y_{0^R} = u$ and $y_{1^R} = v$ from the original problem. We shall also require that for each pair $(u_i, v_i)$ of cousins and a corresponding round $t_i$ of the push-pull protocol the edge $(u_i, v_i)$ not only exists but is also used for communication. This motivates the following definition of a hierarchy, which is a modification of that in~\cite{komjathy2023four1}. But first, we will give some definitions regarding the timing of the edge communication events we will look for. Intuitively, we will assign to each gap a given interval during which we will show that the rumour moves through the gap. We do this using the natural ordering of the gaps, also leaving time in between for the hierarchy edge communication events. 

\begin{defn}[Timing]\label{def:timing}
Let $R, \timeblock{} \ge 2$ be integers. Consider all strings $\sigma \in \{0, 1\}^{R - 1}$, each one corresponding to a smallest gap when building a hierarchy of depth $R$. Order these strings lexicographically. Define for a string $\sigma$ with rank $i$ in this ordering the \emph{gap timeblock} 
$$\gaptimeblock{\timeblock{}}{\sigma} = \left[4i\timeblock{}, (4i + 1)\timeblock{}\right].$$ For two consecutive strings $\sigma_a, \sigma_b$ with ranks $j, j + 1$ on the ordering, let $e_{(\sigma_a, \sigma_b)} = (y_{\sigma_a 1}, y_{\sigma_b 0})$ be the unique edge between these two gaps. Define the \emph{edge timeblock} 
$$\edgetimeblock{\timeblock{}}{e_{(\sigma_a, \sigma_b)}} = \left[(4j + 2)\timeblock{}, (4j + 3)\timeblock{}\right].$$
\end{defn}

Notice that the defined intervals are all disjoint. The idea is that each gap/edge is given its own timeblock for the rumour to propagate through it. The intervals are defined in such a way that the rumour moves from $u$ to $v$ if all gaps/edges are successful in propagating the rumour within their timeblocks. This motivates the third condition in the hierarchy definition below.

\begin{defn}[Hierarchy]\label{def:hierarchy}
    Let $y_0, y_1 \in\calV$, and $R, \timeblock{} \ge 2$ be integers. Consider a set of vertices $\{y_{\sigma}\}_{\sigma\in\{0,1\}^R}$, divided into \emph{levels} $\calL_i:=\{y_{\sigma} \colon \sigma\in\Xi_i\}$ for $i\in \{1,\ldots,R\}$, satisfying that $y_{\sigma}=y_{\sigma'}$ if $\sigma\sim_T \sigma'$.
  We say that $\{y_{\sigma}\}_{\sigma\in\{0,1\}^R}\subset \calV$ is a \emph{$(\gamma, \timeblock{})$-hierarchy of depth $R$} with $\calL_1=\{y_0, y_1\}$ if it satisfies the following properties:
    \begin{enumerate}[(H1)]
        \item  \label{item:H1} $|y_{\sigma0}-y_{\sigma1}|\le |y_0-y_1|^{\gamma^i}$ for all  $\sigma\in\{0,1\}^i$, $i=0,\ldots,R-1$.
        
        \item\label{item:H2} For each $\sigma\in\{0, 1\}^{i}$ for $i \le R - 2$, the edge $(y_{\sigma01}, y_{\sigma10})$ is present in the graph.

        \item\label{item:H3} For each $\sigma_a \in\{0, 1\}^{R - 1} \setminus \{1^{R - 1}\}$, let $\sigma_b \in\{0, 1\}^{R - 1}$ be the successor of $\sigma_a$ in the lexicographical ordering. Let $e_{(\sigma_a, \sigma_b)}$ be the unique edge between the two gaps, as in \Cref{def:timing}. The edge $e_{(\sigma_a, \sigma_b)}$ is used for communication during some round in $\edgetimeblock{\timeblock{}}{e_{(\sigma_a, \sigma_b)}}$.

        \item\label{item:H4} For each vertex $v$ in the graph, there is at most one $\sigma_v \in\{0, 1\}^{i}$ for $i \le R - 2$ such that the edge $(y_{\sigma_v01}, y_{\sigma_v10})$ required to exist by (H2) is incident to $v$.
    \end{enumerate}
\end{defn}

The hierarchy building process exposes edges of the graph in order to build the path through which the rumour will spread. However, notice that this spreading of the rumour hinges on the degrees of the identified vertices remaining bounded closely by their respective weights. Therefore, we must ensure that no vertex of interest has a large number of incident edges exposed during the procedure. This is mostly why we introduce condition (H4). To ensure this condition holds, we decompose the \emph{vertex} exposure into many rounds, one for each level of the hierarchy. Assuming there are $R$ levels, we consider $R$ independent Poisson Point Processes with intensity $\frac{n}{R}$, the union of which is the process with intensity $n$ that our model is defined on. This so far guarantees that the sets $V_i, V_j$ of vertices used in levels $i, j$ of the hierarchy respectively are disjoint for any $i \neq j$. We are now left with the task of making sure that vertices used \emph{within} a single level are also disjoint, in the sense that the edges required to exist by condition (H2) of \Cref{def:hierarchy} are between disjoint vertices. We show this by controlling the geometric distances spanned by the edges required by condition (H2). In particular, we also lower bound the distance in (H1) by a constant factor of the required upper bound, e.g.\ $\frac{1}{2}$. This achieves the following property. Assuming the original distance is $L$, for any given subproblem $(u', v')$ at level $i$ of the recursion, we guarantee that $\geomdist{x_u'}{x_v'} \ge \frac{1}{2}L^{\hierarchyGamma{}^{i}}$. Now, let $U', V'$ be the set of vertices used ``towards'' $u'$ and $v'$ respectively. That is, we put in $U'$ the vertex $u''$ found close to $u'$ in the first subsequent step of the hierarchy building process and all vertices used to bridge subproblems arising by the task of connecting $u'$ to $u''$. The set $V'$ is defined analogously. Notice that when building the hierarchy, by a gross understatement of (H1), we divide the distance by at least a factor of $16$ each time.\footnote{This is technically true only if $L^{\hierarchyGamma{}^{i}}$ is large enough. This is ensured by stopping the hierarchy recursion once distances become small enough, and the threshold for this is a growing function of $n$.} This means that the furthest away a vertex $u''' \in U'$ can be from $u'$ is at most $\frac{1}{8}L^{\hierarchyGamma{}^{i}}$. An analogous statement holds for any $v''' \in V'$, and combining this with the lower bound on $\geomdist{x_u'}{x_v'}$ and triangle inequality shows that $U' \cap V' = \emptyset$. From this it follows that edges identified throughout a single level of the hierarchy are disjoint, as required by (H4). Apart from edges used in the hierarchy, we only expose non-edges, and thus only add at most 1 incident edge per vertex. We then use a union bound to control the degrees of the identified vertices of interest in the hierarchy over the exposure of the rest of the graph.

\subsection{Hierarchies using weak edges}
\input{fast_const_weight}

\subsection{Hierarchies using strong edges}
\input{fast_mid_weight}

%% file: fast_const_weight.tex
Here, we show how one can construct weak-edge hierarchies in \girgs{} when the following inequality holds:
\begin{equation}\label{eq:typeIcondition}
\alpha < \frac{\tau - 1}{\tau - 2}.
\end{equation}

\paragraph{\textbf{Proof idea.}}  As mentioned before, we are interested in paths where each edge has at least one endpoint with small degree. The actual model studied in \cite{biskup2004scaling} is Long-Range Percolation (\LRP{}), which behaves similarly to a \girg{} with all weights set to $1$, or the intersection of a \girg{} with the set of constant-weight vertices. There, most of the degrees are constant, and the paths constructed in this way already satisfy our requirement. However, as we have mentioned, it is assumed that $\alpha < 2$ for that argument to work in \LRP{}. We will actually make use of vertex weights to build paths under the weaker assumption $\alpha < \frac{\tau - 1}{\tau - 2}$. Note that $\frac{\tau - 1}{\tau - 2} > 2$ for all $\tau \in (2,3)$.

The way in which we modify the argument is rather simple. Instead of looking for any edge that connects the boxes $B_{u}$ and $B_v$ around $u,v\in\mathcal{V}$, we only look at edges between small weight vertices in $B_u$ and high-weight vertices in $B_v$ (we could also reverse the roles of $u$ and $v$, or try both variants, but this does not improve the construction). 

The deeper we are in the recursion (see section \ref{sec:tools}), the fewer vertices we can use to find the required edges at each step, thus making the probability of success smaller. However, at the point where the success probability becomes too small for a union bound, the distances we still have to bridge are sufficiently short with respect to the original distance $\hierarchyN{}$. To connect these gaps, we find weight-increasing paths\footnote{The reason we use non-constant (but small) weights when building the hierarchy is so that we are able to find such weight-increasing paths to connect the gaps.} until the two remaining vertices have such large weight that they are connected by a path of length 2 \whp{}. Since the distance to bridge is small enough, the weights (and therefore the degrees) of the vertices on these weight-increasing paths remain sufficiently small.

In the following, for the sake of exposition we will --momentarily-- only work with weights and ignore that degrees may deviate. 

Hence, let us show the existence of weight-restricted paths. We start with a lemma showing that we can find a single bridge edge for any level of our hierarchy, with a failure probability that depends on the current geometric distance of the subproblem.
\begin{lemma}\label{lem:typeIHierarchyBuildingLemma}
    Let $\mathcal{V}_{\theta}$ be a PPP with intensity $\theta n$ and assume that inequality~\eqref{eq:typeIcondition} holds. Let $u, v$ be two already exposed vertices with $|x_u-x_v| \le \hierarchyN{}$. Let $\hierarchyGamma{} \in \left(\frac{\alpha (\tau - 1)}{\alpha + \tau - 1}, 1\right)$ and let $B_u, B_v$ be the sets of vertices in $\mathcal{V_{\theta}}$ with geometric distance at most $\hierarchyN{}^{\hierarchyGamma{}}$ from $u, v$ respectively. Fix some sufficiently small $\TypeIwMaxEpsilon{} > 0$, let $\polyloglogweight{} \le \hierarchyN{}^{\TypeIwMaxEpsilon{}}$, and define the event 
    \[\hierarchySingleStepEventConst{} = \left\{\exists \; u' \in B_u, \; v' \in B_v \: : \: \isedge{u'}{v'}, \; \weightof{u'} \in \left[\polyloglogweight{}, 2\polyloglogweight{}\right], \; \weightof{v'} \in\left[\hierarchyN{}^{ \frac{d \hierarchyGamma{}}{\tau - 1} - \TypeIwMaxEpsilon}, 2\hierarchyN{}^{ \frac{d \hierarchyGamma{}}{\tau - 1} - \TypeIwMaxEpsilon}\right] \right\}.\]
    Then,
    \[\prob{\hierarchySingleStepEventConst{}} \ge 1-\myexp{- \theta \hierarchyN{}^{\myOmega{\TypeIwMaxEpsilon}}}.\]
\end{lemma}

\begin{rem}
Note that the interval $\left(\frac{\alpha (\tau - 1)}{\alpha + \tau - 1}, 1\right)$ is non-empty by assumption \eqref{eq:typeIcondition}.
\end{rem}

\begin{proof}
The number of vertices $v'$ with weight in the interval $[\hierarchyN{}^{ \frac{d \hierarchyGamma{}}{\tau - 1} - \TypeIwMaxEpsilon}, 2\hierarchyN{}^{ \frac{d \hierarchyGamma{}}{\tau - 1} - \TypeIwMaxEpsilon}]$ and position in $B_{v}$ follows a Poisson distribution with expectation $\theta L^{\myOmega{\TypeIwMaxEpsilon}}$, so at least one such vertex exists with probability $1-\myexp{- \theta \hierarchyN{}^{\myOmega{\TypeIwMaxEpsilon}}}$. Similarly, we find at least $\myOmega{\theta \hierarchyN^{d\hierarchyGamma{} - \myO{\TypeIwMaxEpsilon{}}}}$ vertices $u' 
\in B_{u}$ with weight in $[\polyloglogweight{}, 2\polyloglogweight{}]$ since $\polyloglogweight{} \le L^{\TypeIwMaxEpsilon}$.

The number of edges between two such vertices $u', v'$ can be stochastically lower bounded by a binomial random variable $Z_2$. The probability $p_{u'v'}$ that such an edge exists is
\[
p_{u'v'} \ge  \myOmega{ \girgprob{u'}{v'}} 
\ge \myOmega{\mymin{1}{\hierarchyN^{\alpha \left(\frac{d \hierarchyGamma{}}{\tau - 1} - \TypeIwMaxEpsilon\right) - \alpha d}}}.
\]
In the above we have used the fact that $\geomdist{x_u'}{x_v'} = \myTheta{\hierarchyN{}}$ (using the triangle inequality and the fact that $\gamma<1$) and the lower bound on $\weightof{v'}$. Now, considering that there are $\myOmega{\theta \hierarchyN^{d\hierarchyGamma{} - \myO{\TypeIwMaxEpsilon{}}}}$ potential edges, we have
\[
\expectationof{Z_2} = \myOmega{\theta \hierarchyN^{\alpha \left(\frac{d \hierarchyGamma{}}{\tau - 1} - \TypeIwMaxEpsilon\right) - \alpha d + \hierarchyGamma{} d - \myO{\TypeIwMaxEpsilon{}}}}.
\]
Since $\hierarchyGamma{} > \frac{\alpha (\tau - 1)}{\alpha + \tau - 1}$, we can choose $\TypeIwMaxEpsilon>0$ small enough to make the exponent of $\hierarchyN{}$ in the above expression positive. The statement then follows since $\prob{Z_2 = 0} \le \myexp{-\myOmega{\expectationof{Z_2}}} = \myexp{- \theta \hierarchyN{}^{\myOmega{1}}}$.
\end{proof}

With~\Cref{lem:typeIHierarchyBuildingLemma} at hand, we can iteratively find edges that bridge large geometric distances. At each iteration, we use a separate PPP of intensity $\theta n$ to identify the new vertices and edges at the current level of the hierarchy. We use an inversely loglog value for $\theta$, as we will see that we require loglog many levels. The question is how long can we rely on this process. In the $K$-th level of recursion, the distances have shrunken to roughly $\hierarchyN{}^{\hierarchyGamma{}^{K}}$. Also, the number of edges we are simultaneously looking for in this level is $2^{K}$. Therefore, if we want to union bound the probability of failure in such a level, we would want that $2^{K}e^{-\theta\cdot (\hierarchyN{}^{\hierarchyGamma{}^{K}})^{\myOmega{1}}}$ is $\myo{1}$. This sets an upper bound on $K$ that depends on $\hierarchyN{}$ and $\theta$. In turn, this means that we are left with gaps which have geometric distances that could in principle be somewhat large. Fortunately, we can certainly keep recursing until the distances are $\myO{(\log{\hierarchyN{}})^{\epsilon}}$ for any desired $\epsilon > 0$.

There are two things we need to consider to see why this is a good stopping point. First of all, notice that the geometric distances drop doubly exponentially at each step of the hierarchy building process. Therefore, it takes at most $\myO{\log{\log{\hierarchyN{}}}}$ levels to reach constant distances, let alone the distances we actually aim for. Thus, the failure probability for the entire hierarchy is $\myo{1}$, as desired (the last level of the recursion is the bottleneck). This is because $K$ is at most loglog, while the exponent $\theta\cdot (\hierarchyN{}^{\hierarchyGamma{}^{K}})^{\myOmega{1}}$ will be at least some polylog term (due to the gap distance still remaining) over a loglog term (the factor $\theta$). Additionally, the distances are at this level small enough so that we can afford to connect the gaps without caring much about restricting the weights of vertices.

Recall that in the definition of a hierarchy we also demand that each edge identified is also used to communicate within some prescribed timeframe. To be able to claim this, we must suitably bound the degrees of the considered endpoints. We will argue that with the right exposition procedure, whp all degrees of such endpoints are bounded by $\myO{\log^{2\eps}{(n)}}$. We first explain how we obtain the first edge of the hierarchy, and then proceed recursively. We split the geometric space into small boxes of volume $n^{-2}$. Then most of the boxes will be empty, and the probability that a box contains two vertices is $O(n^{-4})$. By a union bound over all boxes, whp there is no box with more than one vertex, so we will assume that henceforth. 

In order to find the first edge $y_{01}y_{10}$ in the hierarchy, we need to find vertices which satisfy conditions (H1)--(H4) of \Cref{def:hierarchy} and the weight condition in \Cref{lem:typeIHierarchyBuildingLemma}. We go through all pairs of boxes one by one, and for each pair of boxes we reveal whether it contains a pair of vertices which satisfies those conditions.

Note that if a pair of boxes turns out negative, we do \emph{not} reveal any further information. In particular, we do not reveal whether one of the boxes contains a vertex. This means that the coin flips are not independent. Rather, the probability for a pair of boxes depends in a complicated way on previous coin flips, but we don't care. As shown above, whp sooner or later we will get a positive coin flip, and we define $y_{01}y_{10}$ by the first coin flip that turns up positive, in which case we reveal their exact location inside of the two box, and the exact weights. Then we recurse in the same manner for the lower levels of the hierarchy.

Now assume that we have revealed the complete hierarchy by this scheme, and let $y_\sigma$ be a vertex in the hierarchy. Let $\deg_h(y_\sigma)$ be the number of neighbours of $y_\sigma$ in the hierarchy, and let $\deg_o(y_\sigma) := \deg(y_\sigma)-\deg_h(y_\sigma)$ be the number of neighbours of $y_\sigma$ outside of the hierarchy. Let $D_o$ be the distribution of $\deg_o(y_\sigma)$ after revealing the hierarchy, and let $D$ be the distribution of the degree of a vertex $y_\sigma$ \emph{before} revealing the hierarchy, i.e. of a vertex with weight and position of $y_\sigma$ before having exposed any other vertices in the graph. Note that for vertices outside of the hierarchy, the revelation scheme only revealed \emph{non-existence} of certain pairs of vertices. Hence, the distribution $D_o$ is stochastically dominated by $D$, and the latter is a Poisson distribution with expectation $E_\sigma=\Theta(W_{y_\sigma})$. Hence, 
\begin{align*}
\Pr{\deg_o(y_\sigma) > E_\sigma + \log^\eps n} \le \Pr{X > E_\sigma + \log^\eps n \mid X \sim D} = e^{-\Omega(\log^\eps n)} = (\log n)^{-\omega(1)}.
\end{align*}
By a union bound over all $(\log n)^{O(1)}$ vertices in the hierarchy, whp all of them satisfy $\deg_o(y_\sigma) \le E_\sigma + \log^\eps n = O(\log^{2\eps} n)$. The contribution of $\deg_h(y_\sigma)$ can be bounded similarly to the bounds used in~\Cref{sec:ultrafast_euclidean}. Therefore it follows that whp, all vertices in the hierarchy have degrees $O(\log^{2\eps} n)$.

If we now provide each edge with $\log^{4\eps}{n}$ rounds, the number of successful usages of this edge by the low degree vertex is a binomial r.v\ with expectation at least $\log^{\eps}{n} = \myomega{\log{\log{n}}}$. Hence, we can union bound over the polylogarithmically many edges of the hierarchy and guarantee that each one will be used in its own timeframe, as per the hierarchy requirement. Now, let us make the preceding discussion formal.

\newcommand{\lowhierarchyevent}{\ensuremath{\calX_{hierarchy}}}
\begin{lemma}\label{lem:low_weight_hierarchy}
     Let $\mathcal{V}_{1/c}$ be a PPP with intensity $n/c$ for some constant $c$ and assume that inequality~\eqref{eq:typeIcondition} holds. Let $u, v$ be two already exposed vertices. Let $\eps > 0$ be sufficiently small, $\hierarchyGamma{} \in \big(\frac{\alpha (\tau - 1)}{\alpha + \tau - 1}, 1\big)$ and $\lowhierarchyevent{}$ be the event that exposing $\mathcal{V}_{1/c}$ reveals a $(\hierarchyGamma{}, Z)$-hierarchy of depth $R$ with\footnote{The factor $4$ in the exponent of $Z$ will be important later when dealing with the gaps.} $\timeblock{} = \lceil \log^{4\eps}{(n)}\rceil$ and $R \le \lceil \frac{\log_2{\log{n}}}{\log_2{(1/\hierarchyGamma{})}}\rceil$ such that the distances separating the lowest level gaps are at most $\log^{(1/d)\eps}{(n)}$. Then,
     \[\prob{\myneg{\lowhierarchyevent{}}} \le e^{-\log^{\myOmega{1}}{(n)}}.\]
\end{lemma}

\begin{proof}
    It suffices to show instead that the union of $G_1, G_2, \ldots, G_R$ reveals such a hierarchy, where each $G_i$ is built from a new PPP of intensity $\theta_{n, \hierarchyGamma{}} \cdot n$ with $\theta_{n, \hierarchyGamma{}} = \frac{\log_2{(1/\hierarchyGamma{})}}{2c\log_2{\log{n}}}$. That is, we divide the intensity $n/c$ of $\mathcal{V}_{1/c}$ among at most $2\frac{\log_2{\log{n}}}{\log_2{(1/\hierarchyGamma{})}}$ many rounds of exposure. We will use the vertices of $G_i$ and edges between them for level $i$ of the hierarchy. Now, we only carry out new rounds until the distances between vertices drop below $(\log{n})^{(1/d)\eps}$. Note that we control the parameter $\hierarchyN{}$ as described in~\Cref{sec:tools} so that distances are both upper and lower bounded whenever we use~\Cref{lem:typeIHierarchyBuildingLemma} below, mainly so that condition (H4) in~\Cref{def:hierarchy} is satisfied. Consider two vertices $u', v'$ which we want to find a bridge for at level $i$ of the hierarchy. By the (assumed) success of the construction until level $i-1$, we know that $\geomdist{x_u'}{x_v'} \le n^{\hierarchyGamma{}^{i - 1}}$. We also set $\polyloglogweight{} = \left(\log{n}\right)^{(1/d) \eps ^ 2}$ and notice that the assumption upper bounding $\polyloglogweight{}$ in~\Cref{lem:typeIHierarchyBuildingLemma} is satisfied. Thus, we know that with probability at least $1 - e^{- \theta_{n, \hierarchyGamma{}}\cdot (n^{\hierarchyGamma{}^{i - 1}})^{\myOmega{1}}}$, we succeed in finding that bridge for $u', v'$ within $G_i$. We are simultaneously looking for at most $2^i$ such bridges when exposing $G_i$, so the probability of failure at this level is at most $2^{i}e^{-\theta_{n, \hierarchyGamma{}}\cdot  (n^{\hierarchyGamma{}^{i - 1}})^{\myOmega{1}}}$. At the last round, the probability of failure for a single pair $u', v'$ is at most $e^{-\theta_{n, \hierarchyGamma{}} \log^{\myOmega{1}}{(n)}}$. This is because \Cref{lem:typeIHierarchyBuildingLemma} is used with a value of at least $\log^{(1/d)\eps}{(n)}$ as its input $\hierarchyN{}$. In previous rounds, the bound is at least as strong, and we have at most $1 + 2 + 4 + \ldots + 2^R \le 2^{R + 1}$ steps of possible failure overall when building the hierarchy, which means that the entire construction fails with probability at most
    \[2^{R + 1}e^{-\theta_{n, \hierarchyGamma{}} \log^{\myOmega{1}}{(n)}}\]
    Now, notice that $R$ need not be larger than the claimed bound in the statement of the lemma, because if we actually carry out that many rounds, distances become constant, and we stop before that. Plugging in this bound for $R$, we have the desired probability bound for the existence of edges.

    Now, we must also guarantee that each edge is used for communication during its alloted timeblock. Note that each identified edge has one endpoint with weight $\myTheta{\polyloglogweight{}}$, as per~\Cref{lem:typeIHierarchyBuildingLemma}. By the previous discussion on the exposure scheme we may assume that the degree of any such vertex in the graph is at most $\myO{\log^{2\eps}{n}}$. Therefore, the number of times an edge is used during a timeblock of length $Z$ is binomial with expectation at least $\log^{\eps}(n)$. By a union bound and invoking a Chernoff bound on each edge, all edges are used at least once within their timeblocks with at least the claimed probability. Note that we are considering polylogarithmically many timeblocks ($\approx e^{\log\log{n}}$) and the probability bound for each of them is more than strong enough ($\approx e^{-\log^{\myOmega{1}}{n}}$).
\end{proof}

Now that we have established the existence of a hierarchy, we will connect the gaps with (efficient) weight-increasing paths. We will first connect all gaps except the first and last, since those require a different approach due to their initial constant-weight vertices. 

\begin{lemma}\label{lem:weight_increasing_gaps}
    Let $\mathcal{P}$ be a collection of pairs $(u_i, v_i)$ of vertices with $|\mathcal{P}| \le \log^{C}{n}$ for some constant $C$ and such that each vertex has weight at least $\log^{(1/d)\eps^2}{n}$ and $\geomdist{x_{u_i}}{x_{v_i}} \le (\log{n})^{(1/d) \eps}$ for some $\eps > 0$. Moreover, each pair has at least one vertex with weight at most $(\log{n})^{\eps}$. Let $p$ be the probability that after exposing a new $\mathcal{V}_{1/c}$ PPP with $n/c$ for some constant $c$ each pair is connected by a path of constant length where each edge is incident to a vertex of weight at most $(\log{n})^{2 \eps}$. Then,
    \[p \ge 1 - \myo{1}.\]
\end{lemma}

\newcommand{\epsstep}{\ensuremath{\eps_s}}
\begin{proof}
    Let us focus on a single pair $u, v$ and show that the probability of failure is $\myo{\frac{1}{(\log{n})^C}}$. The claim then follows by a union bound. First, we will find a path from $u$ to some $u'$ such that the weight of $u'$ is at least $(\log{n})^{\eps}$ and such that $\geomdist{x_u}{x_u'} \le (\log{n})^{(1/d)(2 \eps - \eps')}$, for a new $\eps'$. Fix a small $\epsstep$. Let $q$ be such that $(1/d)(\eps^2) (1 + \epsstep)^{q} \ge \eps$ and notice that $q$ is a constant. That is, if we raise the starting weight to the exponent $(1 + \epsstep)$ in succession $q$ times, we will reach the destination weight. Now, split the PPP into $q$, each with equal intensity. If $\epsstep$ is small enough, it can be shown with a typical Chernoff-like argument that for a given vertex $u_{curr}$, a vertex $u_{next}$ can be found in the new PPP such that $u_{curr}$ is connected to $u_{next}$, $\weightof{u_{next}} \ge \weightof{u_{curr}}^{1 + \epsstep}$ and $\geomdist{x_{u_{curr}}}{x_{u_{next}}} \le \weightof{u_{curr}}^{(1/d) (\tau - 1 + f(\epsstep))}$, for some $f$ that goes to zero as its argument does, and this fails with probability at most $e^{-{\weightof{u_{curr}}}^{\myOmega{1}}}$. Repeating this step $q$ times, we find the required vertex $u'$ such that $u$ is connected to it by a path of length $q$, $\weightof{u'} \ge \log^{\eps}{n}$ and $\geomdist{x_u}{x_u'} \le (\log{n})^{(1/d) \eps  (\tau - 1 + g(\epsstep))}$ for a new function $g$ like $f$ (by triangle inequality). We can do the same to find the required vertex $v'$. Note that because $\tau < 3$, making $\epsstep$ small enough shows that $\geomdist{x_u}{x_u'}, \geomdist{x_v}{x_v'} \le (\log{n})^{(1/d)(2 \eps - \eps')}$, which implies that $\geomdist{x_u'}{x_v'} \le 3(\log{n})^{(1/d)(2 \eps - \eps')}$ by triangle inequality. Now, for the final step, we again find many vertices of weight $\ge \log^{\eps}{n}$ which are at a distance of at most $(\log{n})^{(1/d) (\tau - 1 + f(\epsstep))}$ from e.g.\ $u'$. All these vertices have constant probability of connecting to both $u', v'$ (since the product of the weights is larger than the distance to the power $d$), so at least one of them does with high probability. To see the required bound on the probability of failure for this single pair $u, v$, note that all failure probabilities mentioned above are of the form $e^{-w^{\myOmega{1}}}$ for a $w$ that is at least $(\log{n})^{\myOmega{1}}$, and there are constantly many events to control per pair. The desired bound follows. The condition that each edge is incident to a low-weight vertex is guaranteed from the fact that we only increase weights when needed and the final vertices used to connect $u', v'$ are of small weight also.
\end{proof}

At this point, we have connected all the gaps except the first and last one, involving $u, v$ respectively. To keep the exposition clear, we have delayed the explanation of what happens with the first and last gaps. In order to guarantee that a similar weight-increasing path can connect them, we must first find paths to vertices $u_{\wzero}, v_{\wzero}$ whose weights are sufficiently high (roughly $\wzero$ which only needs to be $\myomega{1}$) so that the the subsequent search for weight-increasing paths in the (updated) first and last gaps succeeds with high probability. We use~\Cref{claim:go_to_wzero} to achieve this.

Now, we have connected the gaps with paths, but we also need to show that the rumour actually traverses the gaps within the allotted timeblocks. To show this, we will use~\Cref{lem:beadyPathGoFast}. Notice that the paths connecting the gaps have the following property. Firstly, their graph-theoretic length is at most $\myO{\log^{\eps}{(n)}}$ (for all but the first and last the length is even a constant). Secondly, each edge on such a path has an endpoint with weight at most $\myO{\log^{2\eps}{(n)}}$ as per~\Cref{lem:weight_increasing_gaps}.

Now, \Cref{lem:beadyPathGoFast} along with our exposure scheme implies that with high probability all these paths transmit the rumour in time $\myO{\log^{4\eps}{(n)}}$. We can finally collect all the previous lemmata into a theorem proving the transmission of the rumour from start to finish.

\newcommand{\fulltransmissionevent}{\calT{}}
\begin{thm}\label{thm:polylog_plusone}
    Let $u, v$ be two vertices chosen uniformly at random from the giant component. Assume that condition~\eqref{eq:typeIcondition} holds. Let $\eps > 0$ be sufficiently small, $\hierarchyGamma{} \in \big(\frac{\alpha (\tau - 1)}{\alpha + \tau - 1}, 1\big)$ and denote by $\fulltransmissionevent{}$ the event that the rumour is transmitted from $u$ to $v$ within the first $\left(\log{n}\right)^{\myDelta(\gamma) + 4\eps}$ rounds. Then,
    \[\prob{\fulltransmissionevent{}} \ge 1 - \myo{1}.\]
\end{thm}

\begin{proof}

    First, we use~\Cref{claim:go_to_wzero} to find paths from $u, v$ to vertices $u_{\wzero}, v_{\wzero}$ such that the following hold. First of all $\wzero = \myomega{1}$. Additionally, only vertices of weight less than $\log{\log{n}}$ are exposed (this can be guaranteed by choosing $\wzero$ increasing slowly enough with $n$). Similarly, the length of this path can be assumed to be at most $\log^{\eps}{n}$.
    
    Then, we split the exposure of vertices with weight larger than $\wzero$ into $3$ PPPs with intensity $n/3$ each, resulting in graphs $G_1, G_2, G_3$. By~\Cref{lem:low_weight_hierarchy} we obtain in $G_1$ a $(\gamma, \timeblock{})$-hierarchy connecting $u_{\wzero}, v_{\wzero}$ with $\timeblock{} = \lceil \log^{4\eps}{(n)}\rceil$, where each gap has geometric distance at most $\log^{\eps}{(n)}$. By~\Cref{lem:weight_increasing_gaps}, we connect all resulting gaps except the first and the last using $G_2$. Finally, we use $G_3$ to connect the first and last gaps involving $u_{\wzero}, v_{\wzero}$ respectively. The way in which we find these paths in $G_3$ is up to minor modifications the same as that used to connect the rest of the gaps, namely we increase the weights doubly exponentially each time and connect the final vertices by a path of length two. We now only need to union bound the probability that two such weight-increasing paths fail and moreover the failure probability for each is dominated by the first step (precisely because of the doubly exponentially increasing weight sequence). Since $\wzero = \myomega{1}$, this step succeeds with high probability. Notice that all vertices used in $G_1, G_2, G_3$ have weights that were not exposed when building the paths to $u_{\wzero}, v_{\wzero}$, so we never expose any PPP twice.\\

    Now that we have constructed a weight-constrained path around the hierarchy, it suffices to show that the rumour is transmitted through each path connecting the gaps in at most $Z$ rounds, during the gap's allotted timeblock as in~\Cref{def:timing}. This is because there are $O(\log^{\myDelta{}(\hierarchyGamma{})}{n})$ timeblocks as defined in~\Cref{def:timing} and each one covers $Z$ rounds. Notice that all paths constructed have a graph-theoretic length of at most $\myO{\log^{\eps}{(n)}}$ with each edge having an endpoint of weight and degree at most $\log^{2\eps}{(n)}$. Thus, using~\Cref{lem:beadyPathGoFast} with $\beadydeg{} = \myO{\log^{2\eps}{(n)}}$, $l = \myO{\log^{\eps}{(n)}}$ and $t = Z$, we have that the probability of a single path taking more than $Z$ time to transmit the rumour is at most $e^{-\myOmega{Z/k}}$. A union bound over all polylogarithmically many such paths finishes the proof.
\end{proof}

%% file: fast_mid_weight.tex
Now let us sketch how we can build similar hierarchies and paths when we instead have the condition 
\begin{equation}\label{eq:typeIIcondition}
\tau < \phi + 1,
\end{equation}
where $\phi$ is the golden ratio. As we have mentioned, in this case we are not guaranteed the existence of edges with low-weight endpoints when building the hierarchy. Instead, we fix a $\hierarchyGamma{} \in (0, 1)$ and we find many edges whose endpoints have weights roughly $w_{max} = \hierarchyN{}^{d \hierarchyGamma{}(\frac{1}{\tau - 1})}$ and $w_{mid}$ such that $w_{mid} w_{max} \ge \hierarchyN{}^d$ when the ``current'' distance is $\hierarchyN{}$. This means that $w_{mid} \ge \hierarchyN{}^{d(1 - \hierarchyGamma{\frac{1}{\tau - 1}})}$. Notice that all these edges exist with a probability that is independent of $\alpha$, i.e.\ they are strong edges not influenced by the geometry. One can show that \whp{} during a given timeblock (see~\Cref{def:timing}) of appropriate size at least \emph{one} of these edges will be used for communication. After we have established the hierarchy, the gaps are connected in the same way as for the case $\alpha < \frac{\tau - 1}{\tau - 2}$. 

Let us see why the condition $\tau < \phi + 1$ emerges for this argument to succeed. If one looks at a ball of radius $\hierarchyN{}^{\hierarchyGamma{}}$, we would expect to find about $\hierarchyN{}^{d \hierarchyGamma{} } \hierarchyN{}^{d (1 - \hierarchyGamma{} \frac{1}{\tau - 1})(1 - \tau)}= \hierarchyN{}^{d (2\hierarchyGamma{} + 1 -  \tau)}$ vertices of weight roughly $w_{mid}$. Additionally, one vertex of weight $w_{max}$ can be expected in the other such ball. We expect a constant fraction of such edges from the mid-weight vertices to the max-weight vertex to exist. Now, out of those edges, about a $\frac{1}{w_{mid}}$ fraction are expected to be used for communication in a single timeblock\footnote{The effect of the overhead of degrees over weights and of the size of the timeblock in~\Cref{def:timing} cancel each other out.}. This means that the expected number of edges being used in a single timeblock is $\hierarchyN{}^{d\left[ \left(2\hierarchyGamma{} + 1 -  \tau\right) -  \left(1 - \hierarchyGamma{\frac{1}{\tau - 1}}\right)\right]}$. To ensure that the exponent of the previous expression is positive, one must have $\hierarchyGamma{} > \frac{\tau(\tau - 1)}{2\tau - 1}$. For many reasons,, the overall argument only works if $\hierarchyGamma{} < 1$. Therefore, we must ensure that $\frac{\tau(\tau - 1)}{2\tau - 1} < 1$ which is true if $\tau <  \phi + 1$. With these changes and otherwise the same proof as that of~\Cref{thm:intro_polylog}, (which itself is a slight modification of the proof of~\Cref{thm:polylog_plusone}) the proof of~\Cref{thm:fast_strong} follows.

%% file: log_lower_bound.tex
In this section we show a lower bound of $\myOmega{\frac{\log{n}}{\log{\log{n}}}}$ rounds needed to spread the rumour in Euclidean GIRGs. Some important ideas behind the proof are inspired by~\cite{gracar2022chemical, Gracar_2021, Lakis_2024}. We assume throughout that
\begin{equation}\label{eq:log_lowerb_bound_conditions}
    \tau > \frac{5}{2} \quad \text{and} \quad \alpha > \frac{1}{\tau - 2},
\end{equation}
since otherwise we know that ultra-fast transmission happens (excluding boundary cases). We blow up distances by a factor of $n^{1/d}$ for the lower bound proof as is done for~\Cref{sec:polylogupper} because this makes the arguments easier to follow.

\paragraph{\textbf{Coupling to graph distances.}} 
To simplify the proof later, we first couple the rumour spreading process $R$ to a different spreading process $R'$ (the coupling is illustrated in~\Cref{fig:coupling-one-vertex}) that has complete independence between edges. This new process has the property that, conditioned on the degree sequence, edges transfer the rumour independently of one another, even if they share an endpoint. We define the process $R'$ as follows. In each round $t$ (independently of all other rounds), each edge $e = uv$ flips two biased coins that have success probabilities $\min\{1, c\frac{\log{n}}{\deg(u)}\}$ and $\min\{1, c\frac{\log{n}}{\deg(v)}\}$, respectively. If at least one of the two tosses is successful, then the edge $e$ is activated in this round. The set of activated edges in round $t$ forms a subgraph $\mathcal{G}_t \subseteq \mathcal{G}$, and in this round every vertex that has graph distance (at most) $1$ from an informed vertex in $\mathcal{G}_t$ becomes informed itself. Note that the rumour spreading process $R$ can be described in the same way, except that the subgraphs $\mathcal{G}_t$ are generated in a different way (namely the set of activated edges in round $t$ is formed by having each non-isolated vertex $u$ activate exactly one of its incident edges).
We couple $R'$ to $R$ in the following way. For each vertex $u$ that is incident to (at least) one activated edge in $R'$, we choose u.a.r.\ one of these incident activated edges to be activated in $R$. As long as in every round each vertex has at least one incident edge that is activated in $R'$, the coupling is valid, and hence $R'$ informs a superset of the vertices that $R$ informs in each round. By choosing a large enough $c$ in the activation probabilities, this can be seen to happen \whp{} for e.g.\ the first $n$ rounds, which is more than long enough for our purposes.

\begin{figure}[ht]
\centering
\begin{tikzpicture}[scale=1.2, thick]
\node[draw, shape=circle] (u) at (0,0) {$u$};

\def\radius{2}
\node[draw, shape=circle] (v1) at (  90:\radius) {$v_1$};
\node[draw, shape=circle] (v2) at ( 162:\radius) {$v_2$};
\node[draw, shape=circle] (v3) at ( 234:\radius) {$v_3$};
\node[draw, shape=circle] (v4) at ( 306:\radius) {$v_4$};
\node[draw, shape=circle] (v5) at (  18:\radius) {$v_5$};

\draw[cyan,very thick] (u) -- (v1);
\draw (u) -- (v2);
\draw[cyan,very thick] (u) -- (v3);
\draw[cyan,very thick] (u) -- (v4);
\draw (u) -- (v5);

\draw[-{Stealth[scale=1.2]},thick] (2.5,0) -- (3,0);

\begin{scope}[xshift=5.2cm]

\node[draw, shape=circle] (u2) at (0,0) {$u$};

\node[draw, shape=circle] (v1_2) at (  90:\radius) {$v_1$};
\node[draw, shape=circle] (v2_2) at ( 162:\radius) {$v_2$};
\node[draw, shape=circle] (v3_2) at ( 234:\radius) {$v_3$};
\node[draw, shape=circle] (v4_2) at ( 306:\radius) {$v_4$};
\node[draw, shape=circle] (v5_2) at (  18:\radius) {$v_5$};

\draw (u2) -- (v1_2);
\draw (u2) -- (v2_2);
\draw[cyan,very thick] (u2) -- (v3_2);
\draw (u2) -- (v4_2);
\draw (u2) -- (v5_2);

\end{scope}
\end{tikzpicture}

\caption{Illustration of the coupling for a single vertex $u$. 
  In the left figure (process $R'$), three edges $uv_1$, $uv_3$, 
  and $uv_4$ are activated (cyan) by $u$. In the right figure (process $R$), 
  $u$ chooses exactly one of these activated edges (here, $(u, v_3)$) 
  to activate in $R$. Note that the other edges could also be activated by the other incident vertex $v_i$.}
\label{fig:coupling-one-vertex}
\end{figure}

Next, since we want to show a lower bound smaller than $\log{n}$, we use a union bound over the first $\log{n}$ rounds (of $R'$) and only keep edges that have been activated at least one time, since they are otherwise irrelevant for the spreading process. At this point, it suffices to lower bound graph distances in the subgraph formed by these edges that were activated at least once in the first $\log n$ rounds of $R'$. By~\Cref{lem:high_weight_means_high_degree}, we can further upper bound the probability of an edge $uv$ with $\weightof{u} \le \weightof{v}$ existing in this subgraph by $\log^4(n)\weightof{u}^{\alpha - 1} \weightof{v}^{\alpha} \geomdist{u}{v}^{-\alpha d}$. The $-1$ in the exponent of $\weightof{u}$ comes from the requirement that the edge is actually activated. We get a $\log{n}$ factor from the $R$ to $R'$ coupling, one from bounding the degrees, another from the union bound over the first $\log{n}$ rounds and another to swallow all constant factors.

Now for lower bounding the distances on this subgraph we can make the additional assumption that $\alpha <2$. Note that the interval $(\frac{1}{\tau-2},2)$ is non-empty since $\tau > \frac{5}{2}$. This only makes edges more likely to appear in the GIRG model. While it increases degrees (which may slow down rumour spreading), our lower bound on the degrees (from~\Cref{lem:high_weight_means_high_degree}) is not invalidated by this modification of $\alpha$. This is the \emph{only place} in the proof where we explicitly make use of the fact that $\tau > \frac{5}{2}$.

For the remainder, we use $\mathcal{G}'$ to refer to the final graph model, where $\alpha \in (\frac{1}{\tau - 2}, 2)$ and we add each edge in the graph independently with probability 
\[p = \min\bigl\{1,  \log^4(n)\weightof{u}^{\alpha - 1} \weightof{v}^{\alpha} \geomdist{u}{v}^{-\alpha d}\bigr\}.\]
It suffices to show a lower bound for graph distances in $\mathcal{G}'$.

\paragraph{\textbf{The expected number of paths in $\mathcal{G}'$ and the BK inequality.}} Our approach for the lower bound is to use a first moment method. That is, we show that the expected number of paths of length at most $k$ from $u$ to $v$ (both chosen u.a.r.) is $o(1)$ for some $k = \myOmega{\frac{\log{n}}{\log{\log{n}}}}$. To do so, we rely on the BK inequality~\cite{Reimer_2000}. Essentially, this allows us to claim that $p \le p_1 p_2$, where $p$ is the probability that a shortest path from $u$ to $v$ exists that passes through some intermediary vertex $z$ and $p_1, p_2$ are the probabilities for a path from $u$ to $z$ and from $z$ to $v$ to exist respectively, modulo length constraints. This inequality has been used in showing lower bounds for similar graph models in previous works~\cite{Lakis_2024, Gracar_2021}. In particular, it is used for Scale Free Percolation in~\cite{Lakis_2024} and the extension to GIRG is explained (the main difference is that points are arranged in a grid in SFP while they form a PPP in GIRG). Technical details that apply in the same way here are given in Appendix B of~\cite{Lakis_2024}.

\paragraph{\textbf{Probability bounds}.} As a first step, we prove the following lemma. This and the next are inspired from~\cite{gracar2022chemical}.

\begin{lemma}
    Let $u, v$ be two vertices with $\weightof{u} \le \weightof{v}$. The probability $p_k$ that they are connected by a path of length at most $k$ where all intermediate vertices have weight at most $\weightof{u}$ is at most $\log^{4k}(n) \weightof{u} \weightof{v}^{\alpha} \geomdist{u}{v}^{-\alpha d}$.
\end{lemma}
\begin{proof}
We proceed by induction. The base case follows from the definition of the model and the assumption $\alpha < 2$. Recall that we used $\tau > \frac{5}{2}$ to justify this assumption while still having $\alpha > \frac{1}{\tau - 2}$. For the inductive step (assuming the claim is true for $k - 1$), let $z$ be the vertex with the maximum weight in the path from $u$ to $v$. It holds that either $\geomdist{u}{z} \ge \geomdist{u}{v}/2$ or $\geomdist{v}{z} \ge \geomdist{u}{v}/2$. We focus on the latter case since this contributes the dominant term in the upper bound for $p_k$. Integrating over the possible positions and weights of $z$ (and summing over the possible number of steps before and after $z$) and using the BK inequality along with the induction hypothesis, we will see that 
\begin{align*}
p_k &\le \myO{ \left(\int\limits_1^{\weightof{u}} \weightof{z}^{-\tau + 1 +\frac{1}{\alpha}} \mathrm{d}\weightof{z}\right) \weightof{u} \weightof{v}^{\alpha} \geomdist{u}{v}^{-\alpha d}\sum\limits_{i=1}^{k} \log^{4i/\alpha}{(n)} \log^{4(k-i)}{(n)}}.
\end{align*}

\begin{figure}
\centering
\begin{tikzpicture}[scale=1.0, every node/.style={font=\small}]

  \node[circle, fill, inner sep=1pt, label=above:$u$] (u) at (-2,3) {};
  \node[circle, fill, inner sep=1pt, label=above:$v$] (v) at (6,5) {};
  \node[circle, fill, inner sep=1pt, label=above:$z$] (z) at (2,2) {};

      \draw[dashed, gray] (u) circle (2);
    \fill[pattern=north east lines, pattern color=gray!50] (u) circle (2);

  \draw (u) -- ++(-2,0) node[midway, above] {\(r\)};
  \node[align=center] at (-3,5.5) {\( r = \left(\log^{4i/\alpha}(n) \weightof{z}^{1/\alpha} \weightof{u}\right)^{1/d} \)};

  \draw[decorate, decoration={snake, amplitude=1mm, segment length=9mm}, thick] (u) -- (z) node[midway, above right] {\(\geomdist{u}{z}\)} node[midway, below right, yshift=-5pt] {\(i\) steps};
  \draw[decorate, decoration={snake, amplitude=1mm, segment length=9mm}, thick] (z) -- (v) node[midway, below right] {\(\geomdist{z}{v} \ge \geomdist{u}{v}/2\)} node[midway, above left] {\(k - i\) steps};;

\end{tikzpicture}
\caption{A schematic illustration of the geometric setup in the proof. 
We consider a path from \(u\) to \(v\) passing through an intermediate vertex \(z\) with smaller weight than both $u, v$. 
The dashed circle around \(u\) denotes the region in which \(z\) lies if \(\log^{4i}(n)\,\weightof{z}\,\weightof{u}^\alpha \,\geomdist{u}{z}^{-\alpha d} \ge 1.\), i.e.\ the probability bound from the induction hypothesis is 1. Note that this is in general not comparable to $\geomdist{u}{v}$.}
\label{fig:probability-lemma}
\end{figure}

Let us explain where the various terms come from. For a visual aid, see~\Cref{fig:probability-lemma}. The sum over $i$ corresponds to the different lengths the path from $u$ to $z$ may have. For each such $i$, we use the induction hypothesis to claim that the path from $u$ to $z$ exists with probability at most $\log^{4i}(n) \weightof{z} \weightof{u}^{\alpha} \geomdist{u}{z}^{-\alpha d}$. Now, there are two cases depending on the value of $\geomdist{u}{z}$. Either this upper bound is at most $1$, or larger than 1. In the latter case, we simply replace the bound by $1$ (which we can always do since we are bounding a probability). To see when this happens, note that
\begin{align*}
    \log^{4i}(n) \weightof{z} \weightof{u}^{\alpha} \geomdist{u}{z}^{-\alpha d} \ge 1 \iff
    \geomdist{u}{z} \le  \left(\log^{4i/\alpha}(n) \weightof{z}^{1/\alpha} \weightof{u}\right)^{1/d}.
\end{align*}
Now, by the intensity of the Poisson Point Process (which is 1 in the model with blown-up distances), we expect to have as many $z$ in this ball as its volume, which is $\myTheta{\log^{4i/\alpha}(n) \weightof{z}^{1/\alpha} \weightof{u}}$, and this explains the terms $\log^{4i/\alpha}(n) \weightof{z}^{1/\alpha} \weightof{u}$. This dominates the contribution over all possible positions of $z$, since outside this ball the probability bound decays with exponent smaller than $-1$.
For the path from $z$ to $v$, we simply use the induction hypothesis and get a term $\log^{4(k-i)}(n) \weightof{z} \weightof{v}^{\alpha} \geomdist{v}{z}^{-\alpha d} \le 2^{\alpha d}\log^{4(k-i)}(n) \weightof{z} \weightof{v}^{\alpha} \geomdist{u}{v}^{-\alpha d}$, since we have assumed $\geomdist{v}{z} \ge \geomdist{u}{v}/2$. 

So far, we have integrated out the position of $z$ and are left with an integral over its weight (recall that the density of the power law is $\weightof{z}^{-\tau}$), in particular we have shown (we can multiply all the probabilities together by the BK inequality):
\begin{align*}
    p_k \le \myO{ \left(\int\limits_1^{\weightof{u}} \weightof{z}^{-\tau + 1 +\frac{1}{\alpha}}\right) \weightof{u} \weightof{v}^{\alpha} \geomdist{u}{v}^{-\alpha d}\sum\limits_{i=1}^{k} \log^{4i/\alpha}{(n)} \log^{4(k-i)}{(n)}}.
\end{align*}
The integral for $\weightof{z}$ amounts to a constant since $-\tau + 1 + \frac{1}{\alpha} < - 1$ (as $\alpha >\tfrac{1}{\tau-2}$). This, along with a rearrangement of the logarithmic terms, gives:
\begin{align*}p_k &\le \log^{4k}(n) \weightof{u} \weightof{v}^{\alpha} \geomdist{u}{v}^{-\alpha d} \myO{\sum\limits_{i=1}^{k} \log^{-(1 - 1/\alpha)i}{(n)}}.\end{align*}
Noticing that the sum is a geometric series sum (with the first term removed), it is $\myo{1}$ with respect to $n$ (since the factor of the geometric series is $(\log{n})^{-(1-1/\alpha)}$ and $\alpha > 1$), and hence the claim is true for $n$ large enough.
\end{proof}

Now, we prove a similar lemma but where we relax the restriction of the weights of intermediary vertices somewhat, allowing $\weightof{z}$ to reach $\weightof{v}$ and not just $\weightof{u}$.
\begin{lemma}\label{lem:higher_bounded_weights}
    Let $u, v$ be two vertices with $\weightof{u} \le \weightof{v}$. The probability $p_k$ that they are connected by a path of length at most $k$ where all intermediate vertices have weight at most $\weightof{v}$ is at most $\log^{4k}(n) \weightof{v}^{\alpha + 1} \geomdist{u}{v}^{-\alpha d}$.
\end{lemma}
\begin{proof}
The proof is very similar in structure to that of the previous lemma. Recall that $z$ is assumed to be the vertex of maximum weight connecting $u$ to $v$. If $\weightof{z} \le \weightof{u}$, then one can use the bound from the previous lemma (both subpaths would need to use vertices of weight smaller than $\weightof{z}$ and thus smaller than $\weightof{u}$ and $\weightof{v}$ also) and bound the total contribution of such paths by e.g.\ $\frac{1}{2}\log^{4k}(n) \weightof{v}^{\alpha + 1} \geomdist{u}{v}^{-\alpha d}$. Let us now turn to the case $\weightof{z} > \weightof{u}$. If $z$ is closer to $u$ than $v$, the calculation has the same integral (but now up to $\weightof{v}$) for $\weightof{z}$, which converges. If not, we get instead the integral $\int\limits_1^{\weightof{v}} \weightof{z}^{-\tau + \alpha + 1}$ and a term $\weightof{v}^{1 + \frac{1}{\alpha}}$ from the induction hypothesis (by the argument with the geometric ball in which the probability bound is 1 for $v$ connecting to $z$). These multiply to give a constant multiple of $\weightof{v}^{\alpha + 2 - \tau + 1 + \frac{1}{\alpha}}$ which is at most $\weightof{v}^{\alpha + 1}$ by the assumption $\alpha > \frac{1}{\tau - 2}$. Again, the (geometric) sum of logarithmic terms swallows up any constants for $n$ large enough.
\end{proof}

Now we are ready to prove our lower bound theorem.

\loglowerboundthm*
\begin{proof}
Let $u$ be the vertex where the rumour starts spreading and $v$ be another vertex chosen uniformly at random. We will show that the probability that $u$ informs $v$ within $k$ rounds is $\myO{n^{-2\rho}}$. This will then imply that the expected number of informed vertices is $\myO{n^{1 - 2\rho}}$ and thus by Markov inequality it is \whp{} at most $\myO{n^{1 - \rho}}$. Indeed, notice that with probability\footnote{We avoid specifying the $\myo{1}$ for simplicity, but it can be shown to be growing slowly enough (or equivalently that we can choose $\rho$ and $\eps$ accordingly based on this probability bound later in the proof).} $1 - \myo{1}$ we have $\geomdist{u}{v}^d \ge n^{1 - \eps}$ for any $\eps > 0$ and also that the weights of both $u$ and $v$ are at most $n^{\eps}$. Under this assumption, let $z$ be the vertex that is of maximum weight in the path between $u$ and $v$. WLOG assume that $z$ is closer to $v$ than $u$, and thus $\geomdist{z}{v}^d \ge 2^{-d}n^{1 - \eps}$. 
Now, we only need to consider weights for $z$ up to $n^{\frac{1}{\tau - 1} + \eps}$ (no higher weight exists \whp{}). Integrating over the weight of $z$ and noticing that it has to connect to $v$ within $k$ steps,~\Cref{lem:higher_bounded_weights} shows that for a single vertex $z$, the probability that it succeeds in connecting $u$ with $v$ is at most $(\log{n})^{4k}n^{(\alpha + 2 -\tau)(\frac{1}{\tau - 1} + \eps)} n^{\eps} 2^d n^{-\alpha(1 - \eps)}$. Taking a union bound over the $O(n)$ candidates for $z$, simple algebra shows that the probability that $u$ informs $v$ is at most $2^d (\log{n})^{4k} n^{\alpha \frac{2-\tau}{\tau - 1} + \frac{1}{\tau - 1} + f(\eps)}$ for a continuous function $f$ with $f(0) = 0$. We can choose $k$ to be a small enough multiple of $\frac{\log{n}}{\log{\log{n}}}$ so that $(\log{n})^{4k}$ is at most $n^{\eps}$. Since $\alpha > \frac{1}{\tau - 2}$, there exists some $\rho > 0$ for which we can choose small enough $\eps$ such that this probability is at most $\myO{n^{-2\rho}}$, which finishes the proof.
\end{proof}

%% file: polynomial_lower_bound.tex
There exist parameter regimes on \girgs{} where at least a polynomial number of rounds is necessary for any constant fraction of the vertices to be informed. This is recorded in Theorem~\ref{thm:intro_slow}, whose proof is the main objective of this section. More precisely, it can be shown that after $n^{\myOmega{1}}$ rounds, only $\myo{n}$ vertices are informed of the rumour \whp{}. The conditions that define this parameter regime are
\begin{equation}\label{eq:slow_conditions}
    \tau > \phi + 1 \quad \text{and} \quad \alpha > \frac{\tau -1}{\tau - 2}.
\end{equation}
For the rest of this section, we will always assume that these conditions hold. Note that an implication of the lower bound on $\alpha$ above (and the assumption $\tau\in(2,3)$) is that $\alpha -\tau > -1$. This will be useful in estimating certain integrals. Additionally, one can deduce that $\alpha > 2$, which is a good sanity check, as for $\alpha < 2$ the diameter of the graph induced by constant-weight vertices (in which case rumour spreading behaves very similarly to graph distances) is at most polylogarithmic in $n$ and these vertices constitute a constant fraction of the entire set of vertices. Now, we begin formalizing the ideas presented in~\Cref{sec:tools}.

\begin{defn}[Long edge]\label{def:long_edge}
Let $e = \edge{u}{v}$ be an edge. We say that $e$ is \longEdgeEpsilon{}-\emph{long}, if $\geomdist{x_u}{x_v} \ge n^{- \longEdgeEpsilon{}}$ for some $\longEdgeEpsilon{} > 0$. Most of the time $\longEdgeEpsilon{}$ will be clear from context (and sometimes irrelevant for the discussion), and we may simply say that $e$ is \emph{long}.
\end{defn}

As mentioned, we want to characterize long edges based on their minimum incident degree. For technical reasons, we will actually base the definition on the \emph{weights} of the incident vertices rather than the degrees, but this is almost the same thing. Note that we can assume \whp{} that the maximum weight in the graph is at most $\wmax{}\coloneqq n^{\frac{1}{\tau - 1} + \maxWeightEpsilon{}}$ for any $\maxWeightEpsilon{} > 0$, so in the following we will not concern ourselves with larger weights.

\begin{defn}[Slowdown of long edge]
    Let $e = \edge{u}{v}$ be a long edge as in Definition \ref{def:long_edge}. We define the quantity $s_e \coloneqq \min\{\weightof{u}, \weightof{v}\}$ as the \emph{\slowdown{}} of $e$. When $\rateof{e} \in [2^i, 2^{i + 1})$, we will say that $e$ has a \emph{\slowdown{} of level $i$}.
\end{defn}

The first step is to upper bound the expected number of \emph{long edges with level $i$ \slowdown{}}. Once again, we only need to do that for $i$ such that $2^{i} \le \wmax$.

\begin{lemma}\label{lem:long_edges_expectation}
Let $L_i$ refer to the number of $\longEdgeEpsilon{}$-long edges of \slowdown{} level $i$ in the graph for some $\longEdgeEpsilon{} > 0$ and $i\in\N{}$ with $2^{i} \le \wmax$. We have
\[
\expectationof{L_i} \le \myO{n^{3 - \tau + 2d\longEdgeEpsilon{}}}.
\]
Moreover\footnote{It is insightful to consider here that when $2^i \ge n^{1 - \frac{1}{\tau - 1}}$ (i.e.\ roughly when the condition $2^{i} \wmax{} \le n$ is not satisfied), the chance of some level $i$ edge being chosen in a round is roughly upper bounded by $\frac{n^{3 - \tau}}{n^{1 - \frac{1}{\tau - 1}}}$, which can be shown to be ${n^{-\myOmega{1}}}$ if $\tau > \phi + 1$.}, if $2^{i+1} \wmax{} \le n^{1-d\longEdgeEpsilon{}}$, we have
\[
\expectationof{L_i} \le \myO{\left(n^{1 -  \frac{\tau-2}{\tau - 1}\alpha + d\longEdgeEpsilon{}(\alpha+1)}\right) \cdot \left(2^{(i + 1)(\alpha - \tau + 1)}\right)}.
\]
\end{lemma}

\begin{proof}
We will first show that regardless of $i$, one has $\expectationof{L_i} \le n^{3 - \tau + 2d\longEdgeEpsilon{}}$. Since $L_i$ is a sum of indicator random variables $X_e$, where $X_e$ is 1 if the edge $e$ exists in the graph and is of \slowdown{} level $i$, it suffices to bound the probability $\prob{X_e = 1}$ under the assumption that $e$ is $\longEdgeEpsilon{}$-long.  We have
\begin{align}\label{eq:prob-long-edge-level-i}
\prob{X_e = 1} &\le \myO{\int\limits_{w_u = 2^{i}}^{2^{i+1}} w_{u}^{-\tau} \mleft[\int\limits_{w_{v} = 1}^{\frac{n^{1 - d \longEdgeEpsilon{}}}{w_{u}}} w_{v}^{-\tau}
\frac{w_{u}^{\alpha} w_{v}^{\alpha}}{n^{\alpha - d \longEdgeEpsilon{}\alpha}}dw_v
+ 
\int\limits_{w_{v} = \frac{n^{1 - d \longEdgeEpsilon{}}}{w_{u}}}^{\wmax{}} w_{v}^{-\tau}
dw_v\mright]dw_u}.
\end{align}
The two inner integrals in \eqref{eq:prob-long-edge-level-i} are the result of splitting at the value $\breakvalue{w_{v}} = \frac{n^{1 - d \longEdgeEpsilon{}}}{w_{u}}$ of $w_{v}$ where the minimum in the connection probability \eqref{eq:girg-connection} is taken by $1$, assuming the smallest distance $n^{-\delta}$ between the endpoints $u$ and $v$. For $w_{v} < \breakvalue{w_{v}}$, the exponent of $w_{v}$ inside the integral is $\alpha - \tau > -1$ (the inequality is a consequence of $\alpha>\tfrac{\tau-1}{\tau-2}$ in \eqref{eq:slow_conditions}). For $w_{v} \ge \breakvalue{w_{v}}$, the exponent becomes $-\tau < -1$. If $I_1$ is the value of the integral between $1$ and $\breakvalue{w_{v}}$ and $I_2$ is the value of the integral between $\breakvalue{w_{v}}$ and $\wmax{}$, it is easy to see that $I_1 = \myTheta{I_2}$. This is because $I_1$ is $\myTheta{F_1(\breakvalue{w_{v}})}$ and $I_2$ is $\myTheta{F_2(\breakvalue{w_{v}})}$ where $F_2$ and $F_1$ are the antiderivative functions of the integrated expression when the minimum is taken by $1$ and the other term respectively. This is because the first integral has exponent more than $-1$ and the second less than $-1$. But the functions $F_1$ and $F_2$ are the same at $\breakvalue{w_{v}}$. This means that we can deduce
\begin{align*}
\prob{X_e = 1} \le \myO{\int\limits_{w_{u} = 2^{i}}^{2^{i+1}} w_{u}^{-\tau} \int\limits_{w_{v} = \frac{n^{1 - d \longEdgeEpsilon{}}}{w_{u}}}^{\wmax{}} w_{v}^{-\tau}
dw_vdw_u}
& = \myO{\int\limits_{w_{u} = 2^{i}}^{2^{i+1}} w_{u}^{-\tau} \mleft(\frac{n^{1 - d \longEdgeEpsilon{}}}{w_{u}}\mright)^{1 - \tau}dw_u}\\
&
\le \myO{n^{1 - \tau + (\tau-1)d\longEdgeEpsilon{}}}
\end{align*}
since
\[
\int\limits_{w_{u} = 2^{i}}^{2^{i+1}} w_{u}^{-1} = \log{2^{i + 1}} - \log{2^i} = \log(2^{i+1}/2^i) =\log2.
\]
The claim $\expectationof{L_i} \le \myO{n^{3 - \tau + 2d\longEdgeEpsilon{}}}$ follows as there are $\myO{n^2}$ pairs of vertices and $\tau-1<2$.

Now, let us turn our attention to the cases where $2^{i+1} \wmax{} \le n^{1-d\longEdgeEpsilon{}}$. Intuitively, here we do not really need to consider the probability of connection being equal to 1. More formally, for $w_u\in[2^i, 2^{i+1})$, we can lower bound the value $\frac{n^{1 - d \longEdgeEpsilon{}}}{w_{u}} \ge \frac{n^{1 - d \longEdgeEpsilon{}}}{2^{i+1}} \ge \wmax{}$, and hence the second inner integral in \eqref{eq:prob-long-edge-level-i} disappears, and we get
\[
\prob{X_e = 1} \le \myO{\int\limits_{w_{u} = 2^{i}}^{2^{i+1}} w_{u}^{-\tau} \int\limits_{w_{v} = 1}^{\wmax{}} w_{v}^{-\tau} \frac{w_{u}^{\alpha} w_{v}^{\alpha}}{n^{\alpha - d \longEdgeEpsilon{}\alpha}}dw_v dw_u}
\le \myO{n^{-\alpha + d \longEdgeEpsilon{}\alpha} (2^{i+1})^{\alpha - \tau + 1} (\wmax{})^{\alpha - \tau + 1}},
\]
where we used that $\alpha-\tau>-1$ to compute the integrals. Inserting the definition of $\wmax{} = n^{\frac{1}{\tau - 1} + \maxWeightEpsilon{}}$ and considering the $\myO{n^2}$ bound for the number of pairs of vertices again yields the required result for \maxWeightEpsilon{} small enough.
\end{proof}

Before we continue, let us consider the implication of the second bound above. Let us focus on constant \slowdown{} edges (i.e.\ when $i=\myO{1}$), in which case $\expectationof{L_i} \le \myO{n^{1 -  \frac{\tau-2}{\tau - 1}\alpha + d\longEdgeEpsilon{}(\alpha+1)}}$. Since each such edge has a constant chance of being used in each round, we would like to be able to claim that \whp{} there are no such edges. This is the case if $1 -  \frac{\tau-2}{\tau - 1}\alpha < 0$, or equivalently $\alpha > \frac{\tau - 1}{\tau - 2}$ as in \eqref{eq:slow_conditions}.

Given that there are logarithmically many levels we need to consider (until we cover the maximum weight in the graph), we can show by Markov inequality that \whp{} the actual $L_i$ values are \emph{simultaneously} at most their respective expectations up to a $n^{\markovEdgesEpsilon{}}$ factor, where $\markovEdgesEpsilon{}>0$ can be chosen arbitrarily small. 

\begin{lemma}\label{lem:long_edges_whp}
    Let $L_i$ be defined as in Lemma \ref{lem:long_edges_expectation}. For any $\markovEdgesEpsilon > 0$, it holds \whp{} that for all $i\in\N$ such that $2^i \le \wmax{}$ we have
    \[
    L_i \le \expectationof{L_i}n^{\markovEdgesEpsilon}.
    \]
\end{lemma}

Now, we further need to make sure that the degrees of the vertices incident to the long edges are roughly what is predicted by their weights (recall that we need to lower bound the degrees for our overall argument to work). We can afford a logarithmic overhead, and so the following lemma suffices.

\begin{lemma}\label{lem:high_weight_means_high_degree}
    There exists a constant $\logConstant$ such that \whp{} for all non-isolated vertices $u$ in the graph, we have
    \[\degreeof{u} \ge \frac{\weightof{u}}{\logConstant \log{n}}.\]
\end{lemma}
\begin{proof}
    If $\weightof{u} \le \logConstant \log{n}$, the claim is true for $u$ deterministically as it is not isolated. If not, then one can see using a standard Chernoff bound (by setting $\logConstant$ large enough) that the claim is true with probability e.g.\ at least $1 - \myO{n^{-2}}$ for a given vertex $u$ with $\weightof{u}\ge\logConstant \log{n}$. Taking a union bound shows that the inequality holds simultaneously \whp{} for all considered $u$, of which there are at most $\myO{n}$.
\end{proof}

Having established all of the above, it is relatively straightforward to show that no long edge is used in the first polynomially many rounds of the protocol.

\begin{lemma}\label{lem:no_long_edges_used}
    There exists $\longEdgeEpsilon{}>0$ such that no edge of geometric length larger than $n^{- \longEdgeEpsilon{}}$ is used to spread the rumour within the first $n^{\longEdgeEpsilon{}/2}$ rounds \whp.
\end{lemma}
\begin{proof}
Let us first focus on the probability that some edge of \slowdown{} level $i$ is used, where $2^{i+1} \wmax{} > n^{1-d\delta}$. We can claim by Lemmata \ref{lem:long_edges_expectation} and \ref{lem:long_edges_whp} that for any desired $\markovEdgesEpsilon{}$ we have \whp{}
\[
L_i \le \myO{n^{3 - \tau + 2d\longEdgeEpsilon{} + \markovEdgesEpsilon{}}}.
\]
Each edge counted in $L_i$ has endpoints with minimum weight at least $2^i > \frac{n^{1-d\longEdgeEpsilon{}}}{2\wmax{}} \ge n^{1 - \frac{1}{\tau - 1} - \maxWeightEpsilon{}-(d+1)\longEdgeEpsilon{}}$. By Lemma \ref{lem:high_weight_means_high_degree}, we also have \whp{} that the degrees of its endpoints are at least $n^{1 - \frac{1}{\tau - 1} - \maxWeightEpsilon{}-(d+2)\longEdgeEpsilon{}}$. It follows by a union bound that in a single round of the protocol, the event that such an edge is used to spread the rumour has probability at most
\begin{equation}\label{eq:heavy_long_edges_bound}
\myO{\frac{n^{3 - \tau + 2d\longEdgeEpsilon{} + \markovEdgesEpsilon{}}}{n^{1 - \frac{1}{\tau - 1} - \maxWeightEpsilon{}-(d+2)\longEdgeEpsilon{}}}}
= \myO{n^{2 - \tau + \frac{1}{\tau - 1} + (3d+2)\longEdgeEpsilon{} + \markovEdgesEpsilon{} + \maxWeightEpsilon{}}}.
\end{equation}

Now, let us consider lower \slowdown{} levels, i.e.\ when $2^{i+1} \wmax{} \le n^{1-d\delta}$. By the same argument but using the second expectation bound in Lemma \ref{lem:long_edges_expectation}, we have that the event that some level $i$ \slowdown{} edge is used in a given round is at most
\[
\myO{\frac{\left(n^{1 -  \frac{\tau-2}{\tau - 1}\alpha + d\longEdgeEpsilon{}(\alpha+1)}\right) \left(2^{(i + 1)(\alpha - \tau + 1)}(\logConstant \log{n})\right)}{2^i}}.
\]

Now, depending on whether $\alpha > \tau$ or not (both of which can occur in our parameter regime), this expression is maximized either for $i = 1$ or for the largest $i$ satisfying $2^{i+1} \wmax{} \le n^{1-d\delta}$. We can assume that equality is possible here by relaxing the $i$ to take real values, since this only worsens our bound. In the first case (i.e.\ $i=1$), one sees that the expression amounts to \begin{equation}\label{eq:light_long_edges_bound}
    \myO{n^{1 -  \frac{\tau-2}{\tau - 1}\alpha + d\longEdgeEpsilon{}(\alpha+1)} \log{n}}.
\end{equation}
In the other case, one actually recovers the same bound as in \eqref{eq:heavy_long_edges_bound}, up to some differences in the arbitrarily small constants in the exponent (which do not matter, as we see in the remainder of the proof).

Finally, notice that the bounds in equations \eqref{eq:heavy_long_edges_bound} and \eqref{eq:light_long_edges_bound} are such that the constants $\longEdgeEpsilon{}, \markovEdgesEpsilon{}$ and $\maxWeightEpsilon{}$ can be chosen small enough so that the final exponents of $n$ are smaller than $0$. We have already discussed why this is true for equation $\eqref{eq:light_long_edges_bound}$ immediately after Lemma \ref{lem:long_edges_expectation}. For equation \eqref{eq:heavy_long_edges_bound}, notice that one needs $2 - \tau + \frac{1}{\tau - 1} < 0$, which is equivalent to $\tau > \phi + 1$ (for $\tau>2$). Fix these constants and let $\roundsEpsilon > 0$ be such that we then have for every $i$
\[
\prob{\text{some long edge of \slowdown{} level $i$ is used in a round}} \le \myO{n^{-2\roundsEpsilon{}}}.
\]
By a union bound over the logarithmically many levels, this implies that
\[
\prob{\text{some long edge is used in a round}} \le \myO{n^{-2\roundsEpsilon{}} \log{n}}.
\]
Finally, by a union bound on all $n^{\roundsEpsilon{}}$ rounds, we can see that the probability that some long edge is used in these rounds is at most
\[
\myO{n^{-\roundsEpsilon} \log{n}} \le \myO{n^{-\roundsEpsilon{} / 2}},
\]
which concludes the proof by choosing $\longEdgeEpsilon{}<\roundsEpsilon{}$.
\end{proof}

We can finally prove~\Cref{thm:intro_slow}, which we restate here for convenience.

\slowTheorem*

\begin{proof}
By Lemma \ref{lem:no_long_edges_used}, there exists $\longEdgeEpsilon{} > 0$ such that \whp{} no edges of geometric length larger than $n^{- \longEdgeEpsilon{}}$ are used to spread the rumour in the first $n^{\longEdgeEpsilon{}/2}$ rounds. Thus, the maximum possible geometric distance travelled by the rumour is $n^{ - \longEdgeEpsilon{}/2}$. Moreover, at this distance from the source lie at most $\myO{n^{1 - \longEdgeEpsilon{}d/2}}$ many vertices \whp{}. Hence, setting $\slowTheoremEpsilon{}_1 = d \longEdgeEpsilon{} / 2$ and $\slowTheoremEpsilon{}_2 = \longEdgeEpsilon{}/2$ finishes the proof.
\end{proof}

%% file: mcd_girgs.tex
The aim of this section is to show that in MCD-GIRGs, rumour spreading is always ultra-fast when $2<\tau<3$. More precisely, we exploit the geometry and scale-freeness of the model to derive a $\myO{\log \log n}$ upper-bound for rumour spreading. To achieve this, we consider the following four-phase process. 

In Phase $1$, we bound the time until a vertex of weight at least $\log^{*4}n$ is reached from the starting vertex with the help of the \emph{bulk lemma} from~\cite{bringmann2016average}. In Phase $2$, we construct a weight-increasing path with the goal of reaching a vertex of weight at least $\sqrt{\log{\log{n}}}$. Here, the vertices' weights (and hence the degrees) are small enough to ensure that the rumour spreads in $o(\log \log n)$ rounds. Phase $3$ is concerned with reaching a vertex of weight $\Omega(n^{\frac{\tau-2}{\tau-1} + \varepsilon'})$. We still leverage weight-increasing paths but now we use intermediate-constant-weight vertices to ensure efficient transmission as the vertex degrees start becoming larger. Thus, if every edge in this path is used at the specified time with high probability, the rumour spreads to the highest-weight vertex in $O(\log \log n)$ rounds. The final phase is almost identical to the previous one with the goal of reaching the highest weight vertex. Then, using the symmetry of the push-pull protocol, we conclude that the rumour spreads to the target vertex in $O(\log \log n)$ rounds.

As in previous sections, the technical difficulty of the proof lies in bounding the degrees of the small-weight vertices. To achieve this goal in the later phases, we maintain the set $\mathcal{U}$ of previously exposed vertices and prove that their influence remains limited. To this end, we need to expose vertices conservatively and restrict this exposure to a small volume.

\vspace{1em}
\paragraph{\textbf{Ball of Influence}}\label{par:MCD-bof} We will frequently encounter the ball of influence of vertices. Recall that for an arbitrary vertex $v$, the ball of influence $I(v)$ is the region where it always forms strong ties, i.e., it is the locus of points $x$ where $V_{min}(x - x_v) \leq \frac{W_v}{n}$. It follows from the definition of MCD that $r \leq V_{min}(r) \leq d\cdot r$, and hence for the sake of simplicity, we redefine the ball of influence to be the locus of points $x$ that satisfy $|x - x_v|_{min} \leq \frac{W_v}{n}$ instead.

The name ``ball of influence" might be misleading as in the MCD-geometry, balls have an unusual shape. Instead, $I(v)$ is better thought of as the union of $d$ ``(hyper-)plates". Each such plate is the region between two parallel hyperplanes (constrained to the unit hypercube), orthogonal to a unique coordinate axis. Indeed, such a plate along dimension $j \in [d]$ is simply the locus of points $x$ such that $|x_j - x_{v, j}| \leq \frac{W_v}{n}$. We refer to this region as $I_j(v)$, and it is straightforward to verify from the definitions that $I(v) = \cup_{j = 1}^d I_j(v)$.

\vspace{1em}
\paragraph{\textbf{Phase 1}}\label{par:MCD-phase1} 
Suppose that the rumour starts at a vertex $v$ (in the giant component) with weight $W_v$. As outlined before, we begin by bounding the time it takes to reach a vertex of weight $\log^{*4}{n}$. We expose the $k$-neighbourhood of $v$ (vertices at a graph distance of at most $k$ from $v$) to find such a vertex, while simultaneously making sure that the total volume that we have exposed is small. By leveraging the bulk lemma from~\cite{bringmann2016average}, we know that for $k = \mathrm{poly}(\log^{*4} n)$ the $k$-neighbourhood contains a path to a vertex of weight at least $\log^{*4} n$.

We first provide a lemma that allows us to bound the number of vertices in the $k$-neighbourhood of a given vertex. 

\begin{lemma}[Size of Neighbourhood]
    \label{lem:neigh_size}
    There exists a constant $c_0$ such that the following holds : Let $w \geq c_0$ be a weight, $V_{< w}$ the set of vertices of weight strictly less than $w$, and $G_{< w} = G[V_{< w}]$ the subgraph induced by this vertex set. Consider an arbitrary vertex $v \in V_{< w}$. Let $N_k(v)$ be the set of vertices which are at a (graph) distance of at most $k$ from $v$ in $G_{< w}$. Then, $|N_k(v)| \leq (c_0w)^{k}$ with probability at least $1 - \mathcal{O} \left(\frac{w^{2k}}{2^{w}} \right)$.
\end{lemma}
\begin{proof}
    We will prove the claim by induction on $k$. For the base case ($k = 0$), the claim holds trivially as $N_0(v) = \{v\}$ deterministically. 
    
    Suppose that the claim is true for some $k \geq 0$. We define $M_j(v)$ to be the set of vertices in $G_{< w}$ that are at distance exactly $j$ from $v$. Notice that for $i \neq j$, $M_i(v)$ and $M_j(v)$ are disjoint, and we have $N_k(v) = \bigcup_{j = 0}^{k} M_j(v)$. Thus, we have $N_{k + 1}(v) \setminus N_k(v) = M_{k + 1}(v)$. Additionally, $M_{k + 1}(v)$ is precisely the set of neighbours of $M_{k}(v)$ that are not in $N_k(v)$.

    Consider arbitrary vertices $x \in M_k(v)$ and $y \in V_{< w} \setminus N_k(v)$. Since we have not yet revealed the position of $y$, $\pr[x \sim y \mid N_k(v)] \leq \myO{\frac{W_x W_y}{n}}$. Thus, using a Chernoff bound, the number of neighbours of $x$ outside $N_k(v)$ is more than $c_0w$ with probability $2^{-\Omega(w)}$, for some appropriate constant $c$. Therefore the size of $M_{k + 1}(v)$ is at most $|M_k(v)| \cdot c_0w \leq |N_k(v)| \cdot c_0w \leq (c_0w)^{k + 1}$. By using the induction hypothesis and union bound, the failure probability is at most 
    \[ \mathcal{O} \left(\frac{w^{2k}}{2^{w}} \right) + \frac{|M_k(v)|}{2^{\Omega(w)}} \leq \mathcal{O} \left(\frac{w^{2k}}{2^{w}} \right) + \frac{(c_0w)^{k + 1}}{2^{\Omega(w)}} = \mathcal{O} \left(\frac{w^{2(k + 1)}}{2^{w}} \right).\]
\end{proof}

Hence, we get the following lemma.
\begin{lemma}\label{lem:mcd-phase-1}
    Let the rumour start at an arbitrary vertex $v$ in the giant component. With high probability, the rumour reaches a vertex $u$ with weight $W_u \ge \log^{*4}{n}$ in $o(\log \log n)$ rounds. Additionally, we expose at most $\myO{ \log^{*3} n}$ many vertices in this process.
\end{lemma}
\begin{proof}
    Assume the weight of the starting vertex satisfies $W_v < \log^{*4}n$, as otherwise the lemma holds trivially. By~\cite[Lemma~5.5]{bringmann2016average}, called the \emph{bulk lemma}, whp there exists a path $P$ from $v$ to a vertex $u$ of weight $W_u \ge \log^{*4}n$, such that the path length is $\myO{(\log^{*4}n)^{c}}$ for some constant $c > 0$ and every vertex on the path leading to $u$ has weight less than $\log^{*4}n$.

    Since any vertex on $P$ has a weight of $O(\log^{*4}n)$, its degree is $O((\log^{*4}n)^2)$ with probability $1 - e^{-\myOmega{\log^{*4}n}}= 1 - \frac{1}{\Omega(\log^{*3} n)}$ (using Chernoff's bound). By a union bound, all exposed vertices have a degree of at most $O((\log^{*4}n)^2)$ with probability $1 - o(1)$. Thus, using \Cref{lem:beadyPathGoFast}, the probability that the rumour is not pushed to $u$ (through $P$) in $\Omega((\log^{*4}n)^4)$ rounds is $e^{-\Omega(\log^{*4}n)} = \myO{1/\log^{*3}{n}} = o(1)$.
    
    

    Notice that it suffices to only expose the vertices in the $k$-neighbourhood of $v$ in $G_{< \log^{*4}n}$, and by \Cref{lem:neigh_size}, the size of such a neighbourhood is at most $(c_0\log^{*4}n)^{k}$ with probability $1 - o(1)$. We know from the bulk Lemma that $k = \myO{(\log^{*4}n)^c} = o \left(\frac{\log^{*3}n}{\log^{*5}n} \right)$, and thus $(\log^{*4}n)^{2k} = o(\log^{*3}n)$.

    To expose only the $k$-neighbourhood, one can use the following technique - First, we expose all vertices in $\log^{*3}n \cdot I(v)$, which is just the ball of influence of $v$ scaled up by a factor of $\log^{*3}n$. The probability that a neighbour in $G_{< \log^{*4}n}$ of $v$ is outside this region is at most $\myO{\left(\frac{(\log^{*4}n)^2}{n \cdot \frac{(\log^{*3}n)\cdot \log^{*4}n}{n}}\right)^{\alpha}} = \frac{1}{(\log^{*3}n)^{\Omega(1)}}$. Since $v$ has at most $\myO{\log^{*4}n}$ neighbours in $G_{< \log^{*4}n}$, none of them lie outside this region with probability $1 - \frac{1}{(\log^{*3}n)^{\Omega(1)}}$. Hence, we can recursively do the same for the neighbours of $v$ until we uncover the $k$-neighbourhood of $v$ in $G_{< \log^{*4}n}$. Since $|N_k(v)| \leq (2\log^{*4}n)^{2k} = o(\log^{*3}n)$, by using a union bound we conclude that this technique uncovers the $k$-neighbourhood with probability $1 - o(1)$.
\end{proof}

\vspace{1em}
\paragraph{\textbf{Phase 2}}\label{par:MCD-phase2} 
Having reached a vertex of weight at least $\log^{*4}{n}$, we aim to reach a weight of at least $\sqrt{\log{\log{n}}}$. We achieve this by constructing a weight-increasing path, where in each step the rumour spreads from a vertex with weight $w$ to one of weight $\Theta(w^{\beta})$, where $1 < \beta < \frac{1}{\tau-2}$. For convenience we define $\eps > 0$, such that $\beta = \frac{1 - \eps}{\tau - 2}$. To do this, we slowly uncover the (scaled up) ball of influence of the current vertex until we uncover enough vertices of weight $\Theta(w^{\beta})$. This way, we do not expose more vertices than necessary.

Before we start, we provide a simple, yet important lemma, which we will use in this as well as the following phases several times.
\begin{lemma}[Doubly Exponential Bound]
    \label{lem:doubly_exp_bound}
    For any $0 < r < 1$ and $b > 1$, we have
    \[ \sum_{n = 0}^{\infty} r^{b^n} \leq \frac{r}{1 - r^{b - 1}}\]
    Furthermore, if $r = r(n) =o(1)$ and $b$ is a constant, then the right-hand side of the above expression is $\myO{r(n)}$.
\end{lemma}
\begin{proof}
    Notice that $r^{b^n} \leq r^n$ for all sufficiently large $n$. Thus, $S \coloneqq \sum_{n = 1}^{\infty} r^{b^n}$ clearly converges as the geometric series $\{ r^n\}_{n \geq 0}$ converges. Consequently
    \[ S = \sum_{n = 0}^{\infty} r^{b^n} = r + r^{b - 1} \cdot \sum_{n = 1}^{\infty} r^{b^n - (b - 1)} \leq r + r^{b - 1} \cdot S\]
    where we have used the fact that $b^n - b^{n - 1} = b^{n - 1}(b - 1) \geq b - 1$. Thus, we get the required claim after rearranging the inequality. We obtain the second part of the claim as a direct corollary.
\end{proof}
We are now prepared to tackle Phase 2.
\begin{lemma}
    \label{lem:big_to_large}
    Let the rumour be in a current vertex $v$ with weight $W_v \ge \log^{*4}{n}$. With high probability, the rumour reaches a vertex u of weight $W_u \ge \sqrt{\log{\log{n}}}$ in $o(\log \log n)$ rounds. We expose $o(\log^{*3}n)$ vertices in this step.
\end{lemma}
\begin{proof}
    We want to construct a weight-increasing path of length $k$ from the starting vertex $v$ of weight $W_v \ge \log^{*4}{n}$ to a vertex $u$ of weight at least $\sqrt{\log{\log{n}}}$. For every $0 \leq i < k$, we require the vertices along the path to satisfy $W_{v_{i+1}} = \Omega(W_{v_{i}}^{\beta})$, where $v_i$ is the $i$'th vertex along the path.

    Consider an arbitrary vertex $v$ of weight $W_v$. We want to investigate the number of neighbours of $v$ of weight $W_v^{\beta}$. In light of this, we consider the region for which vertices of weight $W_v^\beta$ have connection probability $\Omega(1)$. This is just $I(v)$ scaled up by a factor of $W_v^{\beta}$. By definition, the volume of this region is $\frac{W_v^{(1+ \beta)}}{n}$ and hence the expected number of vertices in here is $\frac{n W_v^{(1+ \beta)}}{n} = W_v^{(1+ \beta)}$. The expected number of vertices in this region of weight at least $W_v^{\beta}$ is $W_v^{(1+ \beta)} (W_v^\beta)^{1- \tau} = W_v^{1+\frac{1-\eps}{\tau-2}(2-\tau)} = W_v^{\eps}$. With the help of a Chernoff bound, we obtain concentration with a failure probability of $e^{- \Omega(W_v^{\eps})}$. 
    We do not need all such vertices, and hence we slowly expose this region until we find $(\log^{*5}n)^2$ many of them. Since $v$ has weight $W_v$, the number of neighbours is in $O(W_v)$ with probability $\Omega(1)$. Hence, $v$ pushes the rumour to one of the $(\log^{*5}n)^2$ neighbours of weight at least $W_v^{\beta}$ in a single round with probability $\Omega(W_v^{-1} \cdot (\log^{*5}n)^2)$. The probability that the rumour is not pushed to one such neighbour within the next $W_v$ rounds is 
    \[\left(1- \Omega\left( \frac{(\log^{*5}n)^2}{W_v}\right)\right)^{W_v} \leq e^{-\Omega((\log^{*5}n)^2)} = \frac{1}{(\log^{*4}n)^{\omega(1)}}.\]

    To bound the length $k$ of the path, notice that we have exponential growth in each step and the target is to reach a weight of at least $\log{\log{n}}$. Therefore $k= O(\log^{*4}n)$. Thus the number of rounds Phase 2 requires is at most $O(\log^{*4}{n} \cdot \sqrt{\log\log{n}}) = o(\log{\log{n}})$. By union bound, the failure probability is at most $\frac{k}{(\log^{*4}n)^{\omega(1)}} = o(1)$.

    Finally, we note that the total number of exposed vertices is at most $k \cdot (\log^{*5}n)^2 = o(\log^{*3}n)$.
\end{proof}

\vspace{1em}
\paragraph{\textbf{Phase 3}}\label{par:MCD-phase3} After reaching a vertex of weight at least $\sqrt{\log \log n}$, the vertex degrees become too large to efficiently transmit to a vertex of higher degree directly. We circumvent this problem by using intermediate constant-weight vertices, i.e.\ we construct a path where the weight increases every two hops from $w$ to $\Theta(w^\beta)$ but the intermediate vertices are of constant weight. Just as in Phase $2$, $1 < \beta < \frac{1}{\tau-2}$. We also again define $\eps > 0$, such that $\beta = \frac{1 - \eps}{\tau - 2}$. The constant-weight vertices can pull the rumour from a high-weight vertex and then push it to another one in constant time. We call the process of transmitting the rumour from a vertex of weight $w$ to one of weight $\Theta(w^{\beta})$ a \emph{step}.

A critical aspect is to ensure that these constant-weight vertices also have constant degrees. As we have already uncovered some vertices in the previous phases, this is not true in general. By carefully excluding the region where the previously uncovered vertices have significant influence, the constant-weight vertices outside this region still have constant degree with at least a constant probability. 

The phase works as follows: suppose that the rumour has reached a vertex $v$ of weight $w$. We uncover the set $S$ of all constant-weight vertices in $I_1(v)$. Next, we uncover $(\log^{*3}n)^2$ vertices of weight $\Theta(w^{\beta})$ located in the union of the ball of influences of vertices in $S$. We also make sure that these vertices are sufficiently far apart in dimension $1$, so that their plates of influence (along dimension $1$) do not intersect. We show that with this process we can find a path of length two connecting $v$ to a vertex of weight $\Theta(w^{\beta})$, see Figure \ref{fig:mcd} for an illustration.
For the majority of the following lemmas, we assume that we are at some arbitrary step of Phase $3$, where we aim to transmit the rumour from a vertex $v$ with weight $W_v$ to a vertex of weight $\Theta(W_v^{\beta})$. We begin by investigating the number of vertices of constant degree in the ball of influence of $v$.
\begin{lemma}[Many constant-weight vertices]\label{lem:cells_const_wt_vert}
    Let $v$ be an arbitrary vertex, and $S \subseteq I_1(v)$ be such that the volume of $S$ is at least a $(1 - o(1))$ fraction of the volume of $I_1(v)$. Then $S$ contains $\Theta(W_v)$ vertices of weight $O(1)$ with probability $1 - e^{-\Omega(W_v)}$.
\end{lemma}
\begin{proof}
    As the volume of $I_1(v)$ is $\frac{2W_v}{n}$, the expected number of vertices in $S$ is at least $(1 - o(1))2W_v \geq W_v$. Since vertex positions are i.i.d., by Chernoff's bound, the number of such vertices is at least $W_v$ with probability $1 - e^{-\Omega(W_v)}$. The probability that the weight of a vertex in $I_1(v)$ is in $O(1)$, is $\Theta(1)$ (as the weights follow a power-law distribution). Since vertex weights are i.i.d., by a Chernoff bound, the number of such vertices is at least $\Theta(W_v)$ with probability $1- e^{-\Omega(W_v)}$. Finally, with a union bound over both events, we obtain the required claim.
\end{proof}

At any step $i$ of Phase $3$, let $\calU$ be defined as the set of all vertices that have been exposed in Phases $1$, $2$ and the first $i - 1$ steps of Phase $3$. We want to bound the number of vertices in $\calU$.

\begin{claim}[Number of bad vertices]
    \label{claim:no_bad_vertices}
        Consider the $i$-th step of Phase $3$. We are considering a vertex $v$ with weight $W_v$. The number of previously exposed vertices satisfies $|\calU| = O(W_v^{1 / \beta})$ and the number of non-constant-weight vertices in $\calU$ is at most $W_v^{o(1)}$.
\end{claim}
\begin{proof}
    By Lemma~\ref{lem:mcd-phase-1}, we know that Phase $1$ has exposed at most $O(\log^{*3}n)$ vertices. By Lemma \ref{lem:big_to_large}, the number of vertices exposed in Phase $2$ is in $o(\log^{*3}n)$. Now, in every step $1 \leq j < i$ of Phase $3$, when we consider a vertex $u$, we expose at most $(\log^{*3}n)^2$ vertices of weight at most $W_u \leq W_v^{\beta^{j - i}}$ (because $W_v \geq W_u^{\beta^{i - j}}$).
    Additionally, for one such vertex $u$, we expose all the constant-weight vertices in $I_1(u)$, which amount to $O(W_u)$ by Lemma \ref{lem:cells_const_wt_vert}.
    Thus, with the help of Lemma \ref{lem:doubly_exp_bound}, the total number of vertices uncovered before the $i^{th}$ step of Phase $3$ is $O(\log^{*3}n + i \cdot (\log^{*3}n)^2 + W_v^{1 / \beta})$. To prove the required claim, it now suffices to show that $i \cdot (\log^{*3}n)^2 = W_v^{o(1)}$. First, note that $W_v \geq \sqrt{\log \log n}$. For $i \leq \log^{*3}n$, clearly $i$ is exponentially smaller than $W_v$. After $i \geq \log^{*3}n$ steps, $W_v \geq \log n$ and hence, it is exponentially larger than $i =  \myO{\log \log n}$.
\end{proof}

We previously briefly touched on the issue of the degrees of constant-weight vertices being influenced by the previously exposed vertices $\calU$. To circumvent this issue, we will only consider constant-weight vertices that have a ``large enough distance'' from the previously exposed ones. In light of this, we define the \emph{extended ball of influence} $I^+(u)$ for a vertex $u$. Just like $I(u)$, $I^+(u)$ is the union of $d$ ``plates" $I^+_j(u)$ for $j \in [d]$. For $j \neq 1$, we define $I^+_j(u)$ to be the locus of points $x$ such that $|x_{u, j} - x_j| \leq \frac{W_v^{(1 + \delta)/ \beta}}{n}$, where $W_v=\myTheta{W_u^{\beta}}$ is the weight of the vertex considered in the corresponding step, for some sufficiently small $\delta > 0$. $I^+_1(u)$ is the locus of points $x$ such that $|x_{u, 1} - x_1| \leq \frac{W_u \cdot (W'_u)^{\delta}}{n}$, where $W'_u = \max\left\{W_u, (\log \log n)^{1 / \beta}\right\}$.
Notice that $I^+(v)$ is just a scaled up version of $I(v)$, except that the scaling factor is different in dimension $1$. This is quite clear in dimension $1$, and in the other dimensions this comes from the fact that $W_v^{(1+\delta)/\beta} = \myTheta{W_u^{1+\delta}} = \myomega{W_u}$.

Next, we want to bound the volume of the ``bad" region for this step. For arbitrary $j \in [d]$, let $I^+_j(\calU) \coloneqq \cup_{u \in \calU} I^+_j(u)$ be the union of $I^+_j(u)$ for all previously uncovered vertices. $I^+(\calU)$ is the union of all such $I^+_j(\calU)$. We define the ``good" region to be $I^{good}(v) = I_1(v) \setminus I^+(\calU)$.

\begin{lemma}[Bad influence]
    \label{lem:bad_inf}
        Consider an arbitrary step of Phase $3$ with the current vertex being $v$, such that $W_v = O(n^{\frac{1 - \varepsilon'}{1 + \beta}})$. Then, the volume of $I^{good}(v)$ is at least a $(1 - o(1))$ fraction of the volume of $I_1(v)$.
\end{lemma}
\begin{proof}
    We start by showing that for $j \neq 1$, at most an $o(1)$ fraction of the volume of $I_1(v)$ is intersected by $I^+_j(\calU)$. Consider an arbitrary vertex $u$ that was previously uncovered. Notice that by Lemma \ref{claim:no_bad_vertices}, $|\calU| = O(W_v^{1 / \beta})$. The width of $I^+_j(u)$ (along dimension $j$) for any vertex $u \in \calU$ is
    $\frac{2W_v^{(1 + \delta)/\beta}}{n}$. Thus, the total fraction of volume of $I_1(v)$ that $I^+_j(\calU)$ intersects is at most
    \[\myO{|\calU| \cdot \frac{W_v^{(1 + \delta)/ \beta}}{n}} = \myO{ \frac{W_v^{2(1 + \delta) / \beta}}{n}} = \myO{n^{\frac{2(1 - \varepsilon')(1 + \delta)}{(1 + \beta)\beta} - 1}} = \myO{n^{\delta - \varepsilon'}} = o(1)\]
    where we have used the fact that $W_v = \myO{n^{\frac{1 - \varepsilon'}{1 + \beta}}}$, $\beta \geq 1$ and have chosen $\delta$ to be smaller than $\varepsilon'$. 

    Next, we show that at most an $o(1)$ fraction of the volume of $I_1(v)$ is intersected by $I^+_1(\calU)$. Notice that for any vertex $v'$ in Phase $3$, and for all constant-weight vertices $q$ in $I_1(v')$, it follows from the definition that $I^+_1(q) \subseteq I^+_1(v')$. Thus, we can ignore all constant-weight vertices uncovered in Phase $3$ in the following computation. 
    Consider an arbitrary previously uncovered vertex $u \in \calU$. Notice that the width of $I^+_1(u)$ (along dimension $1$) is $\frac{2W_u \cdot (W_u')^{\delta}}{n} \leq \frac{2W_v^{1/\beta} \cdot W_v^{\delta /\beta}}{n} = \frac{2W_v^{(1 + \delta)/\beta}}{n}$. Recall that the total number of vertices with non-constant weight in $\calU$ is at most $W_v^{o(1)}$ (by Lemma \ref{claim:no_bad_vertices}). Thus, the total width of $I^+_1(\calU)$ is $\frac{W_v^{o(1) + (1 + \delta) / \beta}}{n} =o(\frac{W_v}{n})$, as for sufficiently small $\delta$, $\frac{1 + \delta}{\beta} < 1$. This concludes our proof.
\end{proof}

Next, we investigate the degrees of the constant-weight vertices in $I^{good}(v)$.

\begin{lemma}[Constant Degree]
    \label{lem:const_deg}
    Let $v$ be the vertex that we consider in an arbitrary step $i$ of Phase $3$. Consider an arbitrary vertex $s \in I^{good}(v)$, whose weight is in $O(1)$. The expected degree of $s$ is in $O(1)$.
\end{lemma}
\begin{proof}
    The number of vertices, which have not yet been exposed in previous rounds is in $\Omega(n)$. 
    This combined with the constant weight of $s$ results in $O(1)$ expected number of edges to unexposed vertices. We now aim to bound the number of connections to the previously exposed vertices $\calU$. Note, that for any $u \in \calU$ and any dimension $j \neq 1$, we have $|x_{u, j} - x_{s, j}| \geq \frac{W_v^{(1 + \delta)/ \beta}}{n}$, since $s \in I^{good}(v)$ implies that $s \notin I^+_j(u)$. This is relatively large, and hence the bound for $|x_u - x_s|_{min}$ that we consider in the rest of the proof will always come from dimension $1$.
    
    We start by investigating the number of connections to vertices, which have been uncovered in Phase $3$. We uncover two types of vertices in this step: Those of constant weight and those of ``large'' weight. For any large weight vertex $u$, we have $|x_{u, 1} - x_{s, 1}| \geq \frac{W_{u}^{1 + \delta}}{n}$ by the definition of $I^+_1(u)$. Thus, we also have $|x_{u} - x_{s}|_{min} \geq \frac{W_{u}^{1 + \delta}}{n}$. Therefore the probability that $u$ is connected to $s$ is at most 
    \[O \left (\left(\frac{W_{u} W_s}{n \cdot |x_{u} - x_s|_{min}} \right)^{\alpha} \right) = O \left ( \left(\frac{W_{u}}{n \cdot (W_{u}^{1 + \delta}/n)}  \right)^{\alpha} \right) = O (W_u^{-\delta \alpha}) =  \myO{W_u^{-\delta}},\]
    as $\alpha > 1$. The number of large weight vertices in each previous step amounts to at most $(\log^{*3}n)^2$. Therefore the number of connections added from large weight vertices in a single step in Phase $3$ is $O(W_u^{-\delta} \cdot (\log^{*3}n)^2)$. By Lemma~\ref{lem:doubly_exp_bound}, the expected total number of such connections is in $o(1)$.
    
    Let us turn to the connections to constant-weight vertices. Each revealed constant-weight vertex $q$ in a step is in $I_1(u)$, where $u$ is the corresponding high-weight vertex of the step. Thus, $|x_{q, 1} - x_{s, 1}| \geq \frac{W_{u}^{1 + \delta}}{n}$. The resulting connection probability is in $O\left ( \left(\frac{W_q W_s}{n \cdot |x_q - x_s|_{min}} \right)^{\alpha} \right) = O(W_{u}^{-(1 + \delta)\alpha}) = \myO{W_{u}^{-(1 + \delta)}}$. By Lemma~\ref{lem:cells_const_wt_vert} the number of constant-weight vertices revealed by each step in Phase $3$ is $\Theta(W_u)$, where $u$ is the corresponding high-weight vertex of the step. Together the number of connections added in a single step is $\myO{W_u \cdot W_{u}^{-(1 + \delta)}}$. With Lemma~\ref{lem:doubly_exp_bound}, the expected number of connections added from Phase $3$ is in $o(1)$.

    For Phase $2$, we use a similar argument to the high-weight vertices in Phase $3$. Once again with \ref{lem:doubly_exp_bound}, the expected number of connections to vertices uncovered in Phase $2$ is $\Theta (W_{v}^{-\delta}) = o(1)$.

    The total number of vertices uncovered in Phase $1$ is in $\myO{ \log^{*3}n}$. Consider one such vertex $u$ with weight $W_u \leq \log^{*3}n$. We know that by definition $|x_{u, 1} - x_{s, 1}| \geq \frac{W_{u'}^{\delta}}{n} = \frac{(\log \log n)^{\delta / \beta}}{n}$. Thus, the probability that $u$ is connected to $s$ is at most $\Theta \left ( \left(\frac{W_u W_s}{n |x_u - x_s|_{min}} \right)^{\alpha} \right) = \Theta \left( \left( \frac{\log^{*3}n}{(\log \log n)^{\delta / \beta}}  \right)^{\alpha} \right)$. The expected number of such connections is at most $\Theta \left( \frac{(\log^{*3}n)^{1 + \alpha}}{(\log \log n)^{(\delta \alpha) / \beta}}\right) = o(1)$.
\end{proof}

Intuitively, the constant-weight vertices in $I_1(v)$ are few in number and are spread out. To quantify this, we partition the ground space into cells in the following fashion: we partition dimension $2$ into $L = \left \lceil \frac{n}{W_v^{\beta}} \right \rceil$ intervals such that every interval has size roughly equal to $\frac{W_v^{\beta}}{n}$. To be more precise, a cell is a region of the form $[0, 1) \times J \times [0, 1) .. \times [0, 1)$, where $J$ is of the form $\left [ \frac{a}{L}, \frac{a + 1}{L} \right )$ for some integer $0 \leq a < L$. The cell partition of the entire space also induces a partition of $I_1(v)$.

\begin{lemma}[Vertices are spread out across cells]\label{lem:vert_spread_out_cells}
    Consider an arbitrary vertex $v$ with weight $W_v = O\big(n^{\frac{1 - \varepsilon'}{1 + \beta}}\big)$ for some small $\varepsilon' > 0$. Let $c > \frac{1}{(1 + \beta)\varepsilon'} + 1$ be a constant. Then, no $c$ arbitrary unexposed constant-weight neighbours of $v$ lie in the same cell partitions of $I_1(v)$, with probability $1 - n^{-\Omega(1)}$.
\end{lemma}
\begin{proof}
    Recall that there are $\Theta(W_v)$ many constant-weight neighbours of $v$ in $I_1(v)$, by Lemma \ref{claim:no_bad_vertices}. The probability that an arbitrary such vertex lies in the cell partition $C \cap I_1(v)$ (for some arbitrary cell $C$) is $\frac{1}{L} = \Theta \left( \frac{W_v^{\beta}}{n}\right)$. Thus, the probability that $c$ such vertices lie in the same cell is at most
    \[ \binom{\Theta(W_v)}{c} \cdot \left(\Theta \left( \frac{W_v^{\beta}}{n}\right) \right)^{c - 1} = \myO{ \frac{W_v^{c + (c - 1)\beta}}{n^{c - 1}}} = \myO{ n^{\frac{1 - \varepsilon'}{1 + \beta} \cdot (c + (c - 1)\beta) - (c - 1)}} = \myO{ n^{\frac{1 - \varepsilon'}{1 + \beta}  - (c - 1)\varepsilon'}},\]
    where, in the penultimate step, we made use of our assumption that $W_v = O(n^{\frac{1 - \eps'}{1 + \beta}})$. With $c > \frac{1}{(1 + \beta)\varepsilon'} + 1$, the claim follows.
\end{proof}
Bringing our results together, we show that the rumour spreads from a vertex of weight $w$ to one of weight $\Theta(w^{\beta})$ over a constant-weight vertex.

\begin{lemma}[Weight-increasing path]
    \label{lem:wt_inc_path}
    Suppose that at the end of some round $t$, the rumour has spread to a vertex $v$ in the giant component with weight $W_v \geq \sqrt{\log \log n}$ and $W_v = O(n^{\frac{1 - \varepsilon'}{1 + \beta}})$ for some small $\varepsilon' > 0$. Then, there exists a vertex $u$ such that $W_u = \Theta(W_v^{\beta})$ and the rumour has spread to $u$ at the end of round $t + 2$ with probability $1 - e^{- \Omega(W_v^{\varepsilon})} - n^{-\Omega(1)}$.
\end{lemma}
\begin{proof}
    By Lemma \ref{lem:cells_const_wt_vert}, there are $\Theta(W_v)$ vertices with $\myO{1}$ weight in $I^{good}(v)$ with probability $1 - e^{- \Omega(W_v)}$. By Lemma \ref{lem:const_deg}, the degree of each such vertex is in expectation and hence also with constant probability in $\myO{1}$ (using Markov's inequality). By Chernoff's, there are $\Theta(W_v)$ vertices with $\myO{1}$ degree (and weight) in $I^{good}(v)$ with probability $1 - e^{- \Omega(W_v)}$. For any vertex $q$ in $I^{good}(v) \subseteq I_1(v)$, we have $\left(\frac{W_v W_q}{n \cdot |x_v - x_q|_{min}} \right) \geq  \frac{W_v}{n \cdot W_v/n} \geq 1$. Hence, the probability that $q$ is connected to $v$ is in $\Theta(1)$. Using a Chernoff bound, $v$ is connected to $\Theta(W_v)$ vertices with constant degree, with probability at least $1 - e^{- \Omega(W_v)}$. Similarly, consider any vertex $u$ in a Cell $C$, such that $W_u = \Theta(W_v^{\beta})$. For any other vertex $q$ in $C$, the connection probability to $u$ is $\Theta(1)$ since $\left(\frac{W_u W_q}{n \cdot |x_u - x_q|_{min}} \right) \geq  \Theta \left( \frac{W_v^{\beta}}{n \cdot W_v^{\beta}/n} \right) = \Theta(1)$.
    
    Now, let $S$ be the set of cells $C$ such that $C \cap I^{good}(v)$ contains at least one vertex with degree $\myO{1}$. From Lemma \ref{lem:vert_spread_out_cells}, the number of cells contained in $S$ are at least a constant fraction of the total number of constant degree vertices in $I^{good}(v)$, with probability $1 - n^{- \Omega(1)}$. The total volume of region covered by the cells in $S$ is thus given by $\frac{\Theta(W_v)}{L} = \Theta \left( \frac{W_v^{1 + \beta}}{n}\right)$. The expected number of unexposed vertices having weight $\Theta(W_v^{\beta})$ lying in $S$ is given by $\Theta \left( \frac{W_v^{1 + \beta}}{n}\right) \cdot \Theta(n W_v^{\beta(1 - \tau)}) = \Theta \left( W_v^{1 + \beta(2 - \tau)}\right) =  \Theta \left( W_v^{\varepsilon}\right)$. Using a Chernoff bound, the number of such vertices is $\Theta(W_v^{\varepsilon})$ with probability $1 - e^{- \Omega(W_v^{\varepsilon})}$.

    Thus, by union bound, with probability $1 - e^{- \Omega(W_v^{\varepsilon})} - n^{-\Omega(1)}$, there exist $\Theta(W_v^{\varepsilon})$ vertices $u$ with $W_u = \Theta(W_v^{\beta})$, such that $u$ and $v$ share a common constant degree neighbour $q$, with probability $\Theta(1)$. We only expose $(\log^{*3}n)^2$ such vertices $u$ of weight $\Theta(W_v^{\beta})$. We also make sure that these vertices are sufficiently far apart in dimension $1$, so that their extended plates of influence (along dimension $1$) do not intersect. This is true by an argument similar to Lemma \ref{lem:vert_spread_out_cells}, as these vertices having weight $\Theta(W_v^{\beta})$ will be spread out in dimension $1$.
    
    Recall that $v \sim q \sim u$ with probability $\Theta(1)$. At round $t + 1$, $q$ can pull the rumour from $v$ with probability $\Theta(1)$. At round $t + 2$, it can push the rumour to $u$ with probability $\Theta(1)$. In summary, with probability $\Theta(1)$, the rumour reaches $u$ at the end of round $t + 2$. Since there are $\Theta(n^{\eps})$ such vertices $u$, with probability at least $1 - (1 - \Theta(1))^{(\log^{*3}n)^2} = 1 - e^{-\Omega((\log^{*3}n)^2)} = 1 - 1/(\log \log n)^{\omega(1)}$ the rumour has reached some such vertex in round $t + 2$. A union bound gives us the required claim.
\end{proof}

Finally, we show that starting at a vertex with sufficiently large weight, the rumour reaches a vertex of weight $\omega \left(n^{\frac{\tau - 2}{\tau - 1} + \varepsilon'} \right)$ in $\myO{\log \log n}$ rounds.

\begin{lemma}[Large to Core]
    \label{lem:large_to_core}
    Suppose that the rumour starts at a vertex $v$ in the giant component with weight $W_v \geq \sqrt{\log \log n}$. Then, it reaches a vertex $u$ with weight $W_u = \omega \left(n^{\frac{\tau - 2}{\tau - 1} + \varepsilon'}\right)$ in at most $(2 + o(1)) \cdot \frac{\log \log n}{|\log (\tau - 2)|}$ rounds with probability $1 - o(1)$.
\end{lemma}
\begin{proof}
    First, we will show that the rumour reaches a vertex of weight $\Omega \left(n^{\frac{1 - \varepsilon'}{1 + \beta}}\right)$. By iterated applications of Lemma \ref{lem:wt_inc_path}, we can find a series of weight-increasing vertices $v = v_0, v_1, \dots, v_k = u$, where for every $0 \leq i < k$, we have $W_{v_{i + 1}} = \Omega(W_{v_i}^{\beta})$, $W_{v_{k - 1}} = \myO{n^{\frac{1 - \varepsilon'}{1 + \beta}}}$ and $W_u = \Omega \left(n^{\frac{1 - \varepsilon'}{1 + \beta}}\right)$. Notice that
    \[ n \geq W_u = W_{v_k} \geq W_{v_{k - 1}}^{\beta} \geq ... \geq W_{v_0}^{\beta^k} = W_v^{\beta^k} \geq e^{\beta^k} \]
    and hence, we have $k \leq (1 + o(1)) \cdot \frac{\log \log n}{|\log \beta |} = (1 + o(1)) \cdot \frac{\log \log n}{|\log \beta |}$, where we get the last equality by taking $\beta = \big( \frac{1}{\tau - 2}\big)^{1 - o(1)}$ (by choosing a sufficiently small constant $\varepsilon$). Additionally, we can choose $\varepsilon'$ to be sufficiently small so that $\frac{1 - \varepsilon'}{1 + \beta} > \frac{\tau - 2}{\tau - 1} + \varepsilon'$, which is the required exponent in the claim. Thus, since it takes $2$ rounds for the rumour to spread from $v_{i}$ to $v_{i + 1}$ for every $0 \leq i < k$, the whole process takes at most $2k$ rounds.

    The failure probability is given by 
    \[ \sum_{i = 0}^{k - 1} \left( e^{- \Omega(W_{v_i}^{\varepsilon})} + n^{-\Omega(1)}\right) \leq o(1) + \sum_{i = 0}^{\infty} \myO{\frac{1}{(W_v)^{\varepsilon \beta^i}}} \leq o(1) + \myO{\frac{W_v^{-\varepsilon}}{1 - W_v^{-\varepsilon(\beta - 1)}}} = o(1) \]
    where we have used Lemma \ref{lem:doubly_exp_bound} and the fact that $e^{-x} \leq \frac{1}{x}$ for all $x > 1$.
\end{proof}

\vspace{1em}
\paragraph{\textbf{Phase 4}}\label{par:MCD-phase4} In this short phase, we show that the rumour spreads to the highest weight vertex in just $2$ additional steps.

\begin{lemma}[Core to highest]
    \label{lem:core_to_high}
    Suppose that at the end of some arbitrary round $t$, the rumour has spread to an arbitrary vertex $v$ in the giant component with weight satisfying $W_v = \omega \left(n^{\frac{\tau - 2}{\tau - 1} + \varepsilon'}\right)$. Then, the rumour has spread to the vertex $v_{max}$ with the highest weight $W_{max}$ at the end of round $t + 2$ with probability $1 - n^{- \Omega(1)}$.
\end{lemma}
\begin{proof}
    Recall that with high probability $W_{max} = \Omega \left( n^{\frac{1}{\tau - 1} - \varepsilon'}\right)$. 
    First, consider the region $J(v) \coloneqq I^{good}(v) \cap I_2(v_{max})$. The volume of this region is $\frac{2W_v}{n} \cdot \frac{2W_{max}}{n} = \omega \left( \frac{1}{n} \right)$. Thus, the expected number of vertices with $\myO{1}$ weight in $I(v)$ is $\omega(1)$. By Chernoff's, the number of such vertices is $\omega(1)$ with probability $1 - e^{-\omega(1)} = 1 - o(1)$. By Lemma \ref{lem:const_deg}, the degree of each such vertex is in expectation and hence also with constant probability in $\myO{1}$ (using Markov's inequality). By Chernoff's, there are $\omega(1)$ vertices with $\myO{1}$ degree (and weight) in $I^{good}(v)$ with probability $1 - o(1)$.

    For any vertex $q$ in $J(v) \subseteq I^{good}(v) \subseteq I_1(v)$, we have $\left(\frac{W_v W_q}{n \cdot |x_v - x_q|_{min}} \right) \geq  \frac{W_v}{n \cdot \frac{W_v}{n}} \geq 1$. Hence, the probability that $q$ is connected to $v$ is $\Theta(1)$. Similarly, since $q$ is in $J(v) \subseteq I_2(v_{max})$, we have $\left(\frac{W_{max} W_q}{n \cdot |x_{v_{max}} - x_q|_{min}} \right) \geq  \frac{W_{max}}{n \cdot \frac{W_{max}}{n}} \geq 1$. Hence, the probability that $q$ is connected to $v_{max}$ is $\Theta(1)$.

    The rest of the proof is analogous to Lemma \ref{lem:wt_inc_path}.
\end{proof}

Finally, we combine all our results to show that rumour spreading can be done in $\myO{\log \log n}$ rounds for MCD-GIRGs.
\begin{proof}[Proof of Theorem \ref{thm:intro_mcd_ultrafast}]
    Suppose that the rumour starts at the vertex $v$. By~\Cref{claim:go_to_wzero}, it spreads to a vertex with weight at least $\log^{*4}n$ in $o(\log \log n)$ rounds. By Lemma \ref{lem:big_to_large}, it spreads to a vertex with weight at least $\sqrt{\log \log n}$ in $o(\log \log n)$ rounds. Next, by Lemma \ref{lem:large_to_core}, it spreads to a vertex with weight $\omega \big(n^{\frac{\tau - 2}{\tau - 1}}\big)$ in $(2 + o(1)) \cdot \frac{\log \log n}{|\log (\tau - 2)|}$ rounds, and to the vertex with the highest weight in $2$ additional rounds (by Lemma \ref{lem:core_to_high}). 

    Now, notice that the push-pull protocol is completely symmetric in the direction of rumour-spread. Thus, the time taken to reach the highest-weight vertex from $v$ is the same as the time taken to reach $u$ from the highest-weight vertex. Thus, the required claim follows by adding together all the individual spread times. 
\end{proof}

\begin{proof}[Proof of Theorem \ref{thm:intro_mcd_fast}]
The proof follows the same steps as the proof of Theorem~\ref{thm:intro_mcd_ultrafast}, so we only give a sketch of proof. In fact, the proof is substantially simpler since we do not care about constant factors. Moreover, while we had to deal with various weights in the ultra-fast regime, and had to alternate between vertices of large and of constant weight, here it suffices to use vertices of constant weight in all rounds.

We will restrict ourselves to vertices of weight $\le M$, where we choose the constant $M$ such that a vertex $v$ of weight $w_v = 1$ has in expectation at least $4$ neighbours of weight at most $M$ in each of its $d$ plates $I_i(v)$. As before, we will alternate between hyperplates, but now we consider sets of vertices of small weights, instead of a single vertex of large weight. Assume that in round $k$ we have found a set $S_k$ of vertices which have distance at most $k$ from the starting vertex. Moreover, for dimensions $1$ and $2$, we keep two sets $U_1, U_2 \subseteq [0,1]$ of hyperplates that are already used along this dimension.

Assume first that $|S_k| = o(n)$. If $k$ is even then we will ensure that the positions in $S_k$ along coordinate $1$ are independent and uniformly random in $[0,1]\setminus U_1$, and we will consider the plates $R_k:= \cup_{v\in S_k} I_1(v)$ along dimension $1$. For odd $k$ the positions will be independently at random in $[0,1]\setminus U_2$ along dimension $2$ and we use $R_k := \cup_{v\in S_k} I_2(v)$. As for the ultra-fast case, even if we remove the already uncovered space $U_1 \times [0,1]^{d-1}$ or $[0,1]\times U_2 \times [0,1]^{d-2}$ from $R_k$, obtaining $R_k'$, we still retain a $(1-o(1))$ fraction of the volume because the positions in $S_k$ are random. We define $S_{k+1}$ as the vertices of weight at most $M$ in $R_k'$. Then in expectation the set $S_{k+1}$ has size at least $(4-o(1))|S_k| \ge 2|S_k|$, and the exact size is given by a Poisson distribution. Since all of these vertices have connection probability one to a vertex in $S_k$, they have distance at most $k+1$ from the starting vertex of the rumour. Moreover, as they are defined as the set of vertices in one hyperplate along one dimension, the other coordinates are uniformly random in $[0,1] \setminus U_1$ (or $[0,1]\setminus U_2$) and independent between different vertices in $S_{k+1}$. 
Thus we can couple the process to a branching process with a $\text{Po}(2)$ offspring distribution, and this process takes $O(\log n)$ rounds to reach size $\Omega(n)$.

Once the process has reached size $\eps n$ in some round $k_0 = O(\log n)$, we abandon the alternating scheme and simply grow the set of exposed vertices by a breadth-first search (BFS) among vertices of weight at most $M$, starting with $S_{k_0}$. It was shown in~\cite{lengler2017existence} that the giant component in MCD-GIRGs does not have sublinear separators. More precisely, there exists a constant $\eta >0$ such that every subset of the giant component of size at least $\eps n$ has at least $\eta n$ neighbours, provided that the complement in the giant component has also size at least $\eps n$. Hence, the BFS set grows in each round by at least $\eta n$, and thus reaches all but $\eps n$ of the giant component in time $k_0 + O(1) = O(\log n)$. Since the target vertex $v$ was random in the giant component, it can thus be reached from the starting vertex $u$ in $O(\log n)$ round with probability at least $1-\eps$, for an arbitrary $\eps >0$. Moreover, since all vertices on the path have constant weight, whp the rumour spreads in time $O(\log n)$ from $u$ to $v$ along this path.
\end{proof}